\documentclass[10pt,reqno]{amsart}

\usepackage[dvipsnames]{xcolor}
\usepackage{amsmath,amssymb,amsthm,amsfonts,enumerate,tikz,bm}
\usepackage{mathrsfs}
\usetikzlibrary{matrix,arrows,positioning,calc}
\usepackage[all]{xy}
\usepackage[bookmarksnumbered,colorlinks]{hyperref}
\usepackage{url}
\numberwithin{equation}{section}
\usepackage{graphicx}
\usepackage[T1]{fontenc}
\usepackage{textcomp}
\usepackage{palatino,helvet}

\usepackage{footmisc}

\newtheorem{definition}{Definition}[section]
\newtheorem{remark}[definition]{Remark}
\newtheorem{example}[definition]{Example}
\newtheorem{theorem}[definition]{Theorem}
\newtheorem{proposition}[definition]{Proposition}
\newtheorem{lemma}[definition]{Lemma}
\newtheorem{corollary}[definition]{Corollary}
\theoremstyle{remark}

\usepackage{paralist}

\usepackage{scrextend}
\deffootnote{2em}{0em}{\thefootnotemark\quad}

\newcommand{\mbS}{\mathbb{S}}
\newcommand{\mbZ}{\mathbb{Z}}
\newcommand{\mcA}{\mathcal{A}}
\newcommand{\mcB}{\mathcal{B}}
\newcommand{\mcC}{\mathcal{C}}

\newcommand{\mcE}{\mathcal{E}}
\newcommand{\mcF}{\mathcal{F}}
\newcommand{\mcG}{\mathcal{G}}
\newcommand{\mcI}{\mathcal{I}}
\newcommand{\mcL}{\mathcal{L}}
\newcommand{\mcP}{\mathcal{P}}
\newcommand{\mcS}{\mathcal{S}}
\newcommand{\mcX}{\mathcal{X}}
\newcommand{\mcY}{\mathcal{Y}}
\newcommand{\mcGP}{\mathcal{GP}}
\newcommand{\mcDP}{\mathcal{DP}}
\newcommand{\mcGI}{\mathcal{GI}}
\newcommand{\mcDI}{\mathcal{DI}}
\newcommand{\mfC}{\mathfrak{C}}
\newcommand{\mfF}{\mathfrak{F}}
\newcommand{\mfP}{\mathfrak{P}}
\newcommand{\mfGF}{\mathfrak{GF}}

\newcommand{\Mod}{\mathsf{Mod}}
\newcommand{\fgMod}{\mathsf{mod}}
\newcommand{\Ch}{\mathsf{Ch}}
\newcommand{\Qcoh}{\mathfrak{Qcoh}}
\newcommand{\Thick}{\mathsf{Thick}}

\newcommand{\resdim}{{\rm resdim}}
\newcommand{\coresdim}{{\rm coresdim}}
\newcommand{\pd}{{\rm pd}}
\newcommand{\id}{{\rm id}}
\newcommand{\Gpd}{{\rm Gpd}}
\newcommand{\Gid}{{\rm Gid}}
\newcommand{\Dpd}{{\rm Dpd}}
\newcommand{\Did}{{\rm Did}}

\newcommand{\Hom}{{\rm Hom}}
\newcommand{\Ext}{{\rm Ext}}

\newcommand{\Ker}{{\rm Ker}}
\newcommand{\Coker}{{\rm CoKer}}

\newcommand{\lfpone}{$\mathsf{(lFp1)}$}
\newcommand{\lfptwo}{$\mathsf{(lFp2)}$}
\newcommand{\lfpthree}{$\mathsf{(lFp3)}$}

\newcommand{\GPaone}{$\mathsf{(GPa1)}$}
\newcommand{\GPatwo}{$\mathsf{(GPa2)}$}
\newcommand{\GPathree}{$\mathsf{(GPa3)}$}
\newcommand{\GPafour}{$\mathsf{(GPa4)}$}

\newcommand{\lccpone}{$\mathsf{(lccp1)}$}
\newcommand{\lccptwo}{$\mathsf{(lccp2)}$}
\newcommand{\lccpthree}{$\mathsf{(lccp3)}$}
\newcommand{\rccpone}{$\mathsf{(rccp1)}$}
\newcommand{\rccptwo}{$\mathsf{(rccp2)}$}
\newcommand{\rccpthree}{$\mathsf{(rccp3)}$}

\newcommand{\lcfpone}{$\mathsf{(lcFp1)}$}
\newcommand{\lcfptwo}{$\mathsf{(lcFp2)}$}
\newcommand{\lcfpthree}{$\mathsf{(lcFp3)}$}
\newcommand{\lcfpfour}{$\mathsf{(lcFp4)}$}

\newcommand{\lABone}{$\mathsf{(lABc1)}$}
\newcommand{\lABtwo}{$\mathsf{(lABc2)}$}
\newcommand{\lABthree}{$\mathsf{(lABc3)}$}

\newcommand{\lcABone}{$\mathsf{(lcABc1)}$}
\newcommand{\lcABtwo}{$\mathsf{(lcABc2)}$}
\newcommand{\lcABthree}{$\mathsf{(lcABc3)}$}

\makeatletter
\def\@seccntformat#1{%
  \protect\textup{\protect\@secnumfont
    \ifnum\pdfstrcmp{section}{#1}=0 \scshape\bfseries\fi
    \ifnum\pdfstrcmp{subsection}{#1}=0 \bfseries\fi
    \csname the#1\endcsname
    \protect\@secnumpunct
  }%
}
\makeatother

\begin{document}

\title{Cut cotorsion pairs}

\author{Mindy Huerta}
\address[M. Huerta]{Instituto de Matem\'aticas. Universidad Nacional Aut\'onoma de M\'exico. Circuito Exterior, Ciudad Universitaria. CP04510. Mexico City, MEXICO}
\email{mindy@matem.unam.mx}

\author{Octavio Mendoza}
\address[O. Mendoza]{Instituto de Matem\'aticas. Universidad Nacional Aut\'onoma de M\'exico. Circuito Exterior, Ciudad Universitaria. CP04510. Mexico City, MEXICO}
\email{omendoza@matem.unam.mx}

\author{Marco A. P\'erez}
\address[M. A. P\'erez]{Instituto de Matem\'atica y Estad\'istica ``Prof. Ing. Rafael Laguardia''. Facultad de Ingenier\'ia. Universidad de la Rep\'ublica. CP11300. Montevideo, URUGUAY}
\email{mperez@fing.edu.uy}

\maketitle

\begin{abstract}
We present the concept of cotorsion pairs cut along subcategories of an abelian category. This provides a generalization of complete cotorsion pairs, and represents a general framework to find approximations restricted to certain subcategories. We also exhibit some connections between cut cotorsion pairs and Auslander-Buchweitz approximation theory, by considering relative analogs for Frobenius pairs and Auslander-Buchweitz contexts. Several applications are given in the settings of relative Gorenstein homological algebra, chain complexes and quasi-coherent sheaves, but also to characterize some important results on the Finitistic Dimension Conjecture, the existence of right adjoints of quotient functors by Serre subcategories, and the description of cotorsion pairs in triangulated categories as co-$t$-structures.   
\end{abstract}


\pagestyle{myheadings}
\markboth{\rightline {\scriptsize M. Huerta, O. Mendoza and M. A. P\'{e}rez}}
         {\leftline{\scriptsize Cut cotorsion pairs}}


\section*{\textbf{Introduction}}

Given two classes of objects $\mcA$ and $\mcB$ in an abelian category $\mcC$, it is not always possible for these classes to form a complete cotorsion pair $(\mcA,\mcB)$ in $\mcC$. For example, if $\mcA = \mcGP(R)$ denotes the class of Gorenstein projective modules over a ring $R$, and $\mcB = \mcP(R)^\wedge$ the class of $R$-modules with finite projective dimension, under some assumptions (for instance, if $R$ is an $n$-Iwanaga-Gorenstein ring) we can guarantee that $(\mcGP(R),\mcP(R)^\wedge)$ is a complete cotorsion pair, and so $\mcGP(R) = {}^{\perp_1}(\mcP(R)^\wedge)$ and $(\mcGP(R))^{\perp_1} = \mcP(R)^\wedge$. In general the pair $(\mcGP(R),\mcP(R)^\wedge)$ is not necessarily a complete cotorsion pair over an arbitrary ring $R$. However, by using Auslander-Buchweitz approximation theory, it is known that every $R$-module with finite Gorenstein projective dimension has a Gorenstein projective precover whose kernel has finite projective dimension (see \cite{ABtheory,BMPS}). Moreover, the equalities $\mcGP(R) = {}^{\perp_{1}}(\mcP(R)^{\wedge}) \cap  \mcGP(R)^{\wedge}$ and $\mcP(R)^\wedge = (\mcGP(R))^{\perp_1} \cap \mcGP(R)^\wedge$ also hold true. Hence, along the class $\mcGP(R)^\wedge$ of $R$-modules with finite Gorenstein projective dimension, $(\mcGP(R),\mcP(R)^\wedge)$ can be regarded, in some sense, as a complete cotorsion pair.

The first main goal of the present article is to specify a meaning under which two classes $\mathcal{A}$ and $\mathcal{B}$ of objects in an abelian category $\mathcal{C}$ form a complete cotorsion pair restricted to another class $\mathcal{S}$ of objects in $\mathcal{C}$. Specifically, orthogonality relations between $\mathcal{A}$ and $\mathcal{B}$, and the existence of special $\mathcal{A}$-precovers and special $\mathcal{B}$-preenvelopes, will be restricted to objects in $\mathcal{S}$. These ``local'' properties will be formally presented in the concept of \emph{complete cotorsion pairs cut along $\mathcal{S}$} (or \emph{cut cotorsion pairs}, for short). Many properties of this concept derive in a general language for cotorsion theory and relative homological algebra, which in particular covers some well-known results on complete cotorsion pairs in abelian categories.

A recent approach to the idea of relativizing cotorsion pairs with respect to a class of objects $\mcS \subseteq \mcC$ was proposed in \cite{BMPS}, under the name of \emph{$\mcS$-cotorsion pairs}, where the authors consider cotorsion pairs $(\mcA,\mcB)$ relative to a thick subcategory $\mcS \subseteq \mcC$, where also one needs $\mcA, \mcB \subseteq \mcS$. In this work, it is established an interplay between relative cotorsion pairs, left Frobenius pairs and left weak Auslander-Buchweitz contexts. Specifically, the latter two concepts are in one-to-one correspondence, while left weak AB-contexts coincide with the class of cotorsion pairs $(\mathcal{F,G})$ relative to the smallest thick subcategory containing $\mathcal{F}$, and where $\mcG$ is injective relative to $\mcF$. On the other hand, the cut cotorsion pairs proposed in the present article are a generalization of $\mcS$-cotorsion pairs, in the sense that for the former concept it is not required that $\mcS$ is thick or $\mcA, \mcB \subseteq \mcS$ either. So it is natural to think of a more general version of the just mentioned interplay. Our second main goal will be to present relative versions of Frobenius pairs and weak AB-contexts, which we shall call \emph{cut Frobenius pairs} and \emph{cut weak AB-contexts}, so that the previous interplay can be extended to the context of cut cotorsion pairs.

Cut cotorsion pairs, cut Frobenius pairs and cut weak AB-contexts are useful to describe several situations related to approximation theory. We shall support this claim presenting several examples in the context of relative Gorenstein homological algebra, in part motivated by the behavior of Gorenstein projective and projective modules mentioned at the very beginning, but also for a better understanding of the new concepts and results. More complex examples are exhibited in the end of this article, for particular abelian categories such as chain complexes and quasi-coherent sheaves. Moreover, some applications are given with the purpose to describe some well-known results in the study of finitistic dimensions of rings, right adjoints of Serre quotients, and cotorsion pairs and co-t-structures in  triangulated categories.


\subsection*{Organization}

The present article is organized as follows. In Section \ref{sec:preliminaries} we recall some preliminary notions from relative homological algebra. Among these, the most important are the concepts of Frobenius pairs, cotorsion pairs and Gorenstein objects relative to GP-admissible pairs. 

Section \ref{sec:cut_cotorsion} is devoted to present the main concept of the present article: complete left and right cotorsion pairs $(\mcA,\mcB)$ cut along subcategories $\mcS$ of an abelian category $\mcC$ (see Definition \ref{def:cut_cotorsion_pair}). We give in Proposition \ref{prop:cuts_from_Frobenius_pairs} some examples of such pairs coming from left Frobenius pairs. These will cover several complete cut cotorsion pairs made up of relative Gorenstein modules, such as Gorenstein projective, Ding projective and AC-Gorenstein projective modules (Examples \ref{ex:GP} and \ref{ex:GP2}). An equivalent description of complete cut cotorsion pairs is given in Proposition \ref{prop:characterization_ccp}. Some properties more focused on the cuts $\mcS$ than the classes $\mcA$ and $\mcB$ are proven in Proposition \ref{prop:properties_of_cuts}. In particular, the union property (that is, the union of cotorsion cuts for a pair $(\mcA,\mcB)$ is a cotorsion cut for $(\mcA,\mcB)$) allows us to define the maximal cotorsion cut for a pair $(\mcA,\mcB)$ (see Definition \ref{def:maximal_cut}). An explicit description of maximal cuts is shown in Theorem \ref{theo:maximal_S}, for the case where the orthogonality relation $\Ext^1_{\mcC}(\mcA,\mcB) = 0$ is satisfied. The  properties appearing in Proposition \ref{prop:properties_of_cuts} basically constitute a method to obtain new cotorsion cuts from old ones. On the other hand, in Proposition \ref{prop:a_new_cotorsion_cut} we show how to construct a complete cut cotorsion pair from classes $\mcA$, $\mcB$ and $\mcS$ satisfying a series of mild conditions. One of these conditions will be important to motivate and understand the definition of cut Frobenius pairs and cut weak AB contexts. 

The concepts of Frobenius pairs and weak AB contexts cut along subcategories are the main topic of Section \ref{sec:cut_Frobenius_and_cut_AB} (Definitions \ref{def:relative_Frobenius_pair} and \ref{def:cut_AB_context}). GP-admissible pairs $(\mcX,\mcY)$ satisfying certain conditions are the main source to obtain Frobenius pairs cut along  Gorenstein objects relative to $(\mcX,\mcY)$, as we show in Proposition \ref{prop:induced_cut_cotorsion_pair}. On the other hand, in the case where $\mcX$ and $\mcY$ have some closure properties, we see in Example \ref{ex:From_GP_to_AB} that it is possible to obtain three different types of weak AB contexts cut along the class of objects with finite $\mcX$-resolution dimension. The most important result in this section is Theorem \ref{theo:GP_admissible_weak_AB}, where it is shown that it is not possible to obtain non-trivial weak AB contexts from Gorenstein objects relative to hereditary complete cotorsion pairs. This points out the need of a cut version of weak AB contexts, more suitable for Gorenstein objects relative to GP-admissible pairs \cite{BMS}.  

In Section \ref{sec:versus} we prove two correspondence theorems between cut Frobenius pairs, cut weak AB contexts and certain complete cut cotorsion pairs. More specifically, in Theorem \ref{theo:correspondence_1} we establish a one-to-one correspondence between Frobenius pairs and weak AB contexts cut along a class $\mcS$ closed under kernels of epimorphisms and cokernels of monomorphisms. In order for this correspondence to be bijective and well defined, we shall need to consider cut Frobenius pairs and cut AB contexts under certain equivalence relations (see Definition \ref{def:Frobenius_and_AB_relations}). If in addition $\mcS$ is closed under extensions and direct summands, we also obtain in Theorem \ref{theo:correspondence_2} another bijective correspondence between (equivalence classes of) weak AB contexts cut along $\mcS$ and complete cotorsion pairs $(\mcF,\mcG)$ cut along the smallest thick subcategory containing $\mcF$, with $\mcG$ injective relative to $\mcF$, and such that $\mcF \cap \mcG \cap \mcS$ is both a relative generator and cogenerator of $\mcF \cap \mcG$. 

Finally, in Section \ref{sec:applications} we present more detailed examples of complete cut cotorsion pairs, cut Frobenius pairs and cut weak AB contexts. Our first example shows how to induce certain complete cut cotorsion pairs in chain complexes from complete cut cotorsion pairs in an abelian category, in a similar way as occurs with complete cotorsion pairs (as it appears in Gillespie's \cite{GillespieFlat}). We show in Propositions \ref{cutcomplexes} and \ref{cutcomplexes2} that, under certain conditions, $(\mcA,\mcB)$ is a complete cotorsion pair cut along $\mcS$ if, and only if, the class of exact complexes with cycles in $\mcS$ is a cotorsion cut for a certain pair of classes of complexes involving $\mcA$ and $\mcB$. One good point of our results is that the ambient category in which we consider cut cotorsion pairs is not required to have enough projective or injective objects. So in particular, the theory of cut cotorsion pairs can be applied in such settings like the category $\Qcoh(X)$ of quasi-coherent sheaves over a scheme $X$. Motivated by a recent result by Christensen, Estrada and Thompson \cite{CET}, where they show that Gorenstein cotorsion sheaves over a semi-separated noetherian scheme $X$ are the right half of a hereditary complete cotorsion pair, we give an example of a complete right cut cotorsion pair in $\Qcoh(X)$ that has to do with the description of Gorenstein cotorsion sheaves along certain subcategories of $\Qcoh(X)$. The present article closes with several applications in the field of representation theory of algebras. The first one shown in Propositions \ref{prop:findimR} and \ref{findimR-mod} characterizes the finiteness of the (small) finitistic dimension of a ring $R$ in terms of the existence of a certain complete cut cotorsion pair built up of (finitely generated) $R$-modules with finite projective dimension. Another application follows a recent work by Ogawa \cite{Ogawa}, and shows in Theorem \ref{theo:Serre} that for certain Serre subcategories $\mcS$, the quotient functor $\mcC \longrightarrow \mcC / \mcS$ admits a right adjoint if, and only if, the right orthogonal complement of $\mcS$ relative to the Hom functor is a right cotorsion cut for a certain pair completely determined by $\mcS$. Consequences of this equivalence will constitute other applications in the settings of modules over (ex)triangulated categories and co-t-structures (Propositions \ref{prop:defect}, \ref{prop:fp-defect} and \ref{prop:triangle_defect}).


\subsection*{Conventions}

Throughout, $\mcC$ will always denote an abelian category (not necessarily with enough projective and injective objects), unless otherwise specified. The main examples of such categories considered in this article will be:
\begin{itemize}
\item $\Mod(R)$ = left $R$-modules over an associative ring with identity $R$. For simplicity, all modules over $R$ will be left $R$-modules. 

\item $\fgMod(\Lambda)$ = finitely generated modules over an Artin algebra $\Lambda$. 

\item $\Ch(\mcC)$ = chain complexes of objects in $\mcC$. For the case where $\mcC = \Mod(R)$, the corresponding category of chain complexes of $R$-modules will be denoted by $\Ch(R)$. Objects in $\Ch(\mcC)$ are denoted as $X_\bullet$, $X_m$ denotes the $m$-th component of $X_\bullet$ in $\mcC$, and $Z_m(X_\bullet)$ denotes the $m$-th cycle of $X_\bullet$ in $\mcC$. 

\item $\Mod(\mathfrak{A}^{\rm op})$ = right $\mathfrak{A}$-modules. Here, $\mathfrak{A}$ is a skeletally small additive category. A right $\mathfrak{A}$-module is a contravariant additive functor $\mathfrak{A} \longrightarrow \mathsf{Ab}$, where $\mathsf{Ab} = \Mod(\mathbb{Z})$ denotes the category of abelian groups. 

\item $\Qcoh(X)$ = quasi-coherent sheaves over a semi-separated scheme $X$.
\end{itemize}

Subcategories of $\mcC$ are always assumed to be full, and classes of objects of $\mcC$ are regarded as (full) subcategories of $\mcC$. In any case, we write $\mcS \subseteq \mcC$ to denote that $\mcS$ is a subcategory of $\mcC$ or a class of objects in $\mcC$. If $S$ is an object of $\mcS$, we write $S \in \mcS$. Given two objects $X, Y \in \mcC$, we denote by $\Hom_{\mcC}(X,Y)$ the group of morphisms $X \to Y$. In case $X$ and $Y$ are isomorphic, we write $X \simeq Y$. The notation $F \cong G$, on the other hand, is reserved to denote the existence of a natural isomorphism between two functors $F$ and $G$. Monomorphisms and epimorphisms in $\mcC$ may sometimes be denoted by using the arrows $\rightarrowtail$ and $\twoheadrightarrow$, respectively. 

We shall refer to commutative diagrams whose rows and columns are exact sequences as \emph{solid diagrams}. 

Finally, we point out that the definitions and results presented in this article have their corresponding dual statements, which will be omitted for simplicity. Moreover, although the new concepts in Sections \ref{sec:cut_cotorsion}, \ref{sec:cut_Frobenius_and_cut_AB} and \ref{sec:versus} below will be stated for abelian categories, they carry over to any extriangulated category.


\section{\textbf{Preliminaries}}\label{sec:preliminaries}

Let us recall some preliminary notions from relative homological algebra and Auslander-Buchweitz approximation theory.


\subsection*{Resolution dimensions}

Let $\mcB \subseteq \mcC$ be a class of objects of $\mcC$. Given an object $C \in \mcC$ and a nonnegative integer $m \geq 0$, a \emph{$\mcB$-resolution of $C$ of length $m$} is an exact sequence
\[
0 \to B_m \to B_{m-1} \to \cdots \to B_1 \to B_0 \to C \to 0
\]
where $B_k \in \mcB$ for every integer $0 \leq k \leq m$. The \emph{resolution dimension of $C$ with respect to $\mcB$} (or the \emph{$\mcB$-resolution dimension of $C$}, for short), denoted $\resdim_{\mcB}(C)$, is defined as the smallest nonnegative integer $m \geq 0$ such that $C$ has a $\mcB$-resolution of length $m$. If such $m$ does not exist, we set $\resdim_{\mcB}(C) := \infty$. Dually, we have the concepts of \emph{$\mcB$-coresolutions of length $m$} and of \emph{coresolution dimension with respect to $\mcB$}, denoted by $\coresdim_{\mcB}(C)$. With respect to these two homological dimensions, we shall frequently consider the following two classes of objects in $\mcC$:
\begin{align*}
\mcB^\wedge_m & := \{ C \in \mcC \text{ : } \resdim_{\mcB}(C) \leq m \}, & & \text{and} & \mcB^\wedge & := \bigcup_{m \geq 0} \mcB^\wedge_m, \\
\mcB^\vee_m & := \{ C \in \mcC \text{ : } \coresdim_{\mcB}(C) \leq m \}, & & \text{and} & \mcB^\vee & := \bigcup_{m \geq 0} \mcB^\vee_m.
\end{align*}


\subsection*{Orthogonality with respect to extension functors} 

In any abelian category $\mcC$, we can define the extension bifunctors $\Ext^i_{\mcC}(-,-) \colon \mcC^{\rm op} \times \mcC \longrightarrow \mathsf{Ab}$, with $i \geq 1$, in the sense of Yoneda. We shall also identify $\Ext^0_{\mcC}(-,-)$ with the hom bifunctor $\Hom_{\mcC}(-,-)$. The reader can check for instance  \cite{Sieg} for a detailed treatise on this matter.

Given two classes of objects $\mcA, \mcB \subseteq \mcC$ and $i \geq 0$, the notation $\Ext^i_{\mcC}(\mcA,\mcB) = 0$ will mean that $\Ext^i_{\mcC}(A,B) = 0$ for every $A \in \mcA$ and $B \in \mcB$. In the case where $\mcA = \{ M \}$ or $\mcB = \{ N \}$, we shall write $\Ext^i_{\mcC}(M,\mcB) = 0$ and $\Ext^i_{\mcC}(\mcA,N) = 0$, respectively. Recall that the \emph{right $i$-th orthogonal complement of $\mcA$} is defined by
\[
\mcA^{\perp_i} := \{ N \in \mcC \mbox{ : } \Ext^i_{\mcC}(\mcA,N) = 0 \},
\]
and the \emph{total right orthogonal complement of $\mcA$} by 
\[
\mcA^\perp := \bigcap_{i \geq 1} \mcA^{\perp_i}.
\]
Dually, we have the \emph{$i$-th} and \emph{the total left orthogonal complements ${}^{\perp_i}\mcB$ and ${}^{\perp}\mcB$ of $\mcB$}, respectively.


\subsection*{Relative homological dimensions}

Given a class $\mcX \subseteq \mcC$ and $M \in \mcC$, the \emph{relative projective dimension of $M$ with respect to $\mcX$} is defined as
\[
\pd_{\mcX}(M) := \mbox{min}\{ n \geq 0 \text{ : } \Ext_{\mcC}^i(M,\mcX) = 0 \mbox{ for every } i > n \}.
\]
Furthermore, the \emph{relative projective dimension of a class $\mcY \subseteq \mcC$ with respect to $\mcX$} is defined as
\begin{align*}
\pd_{\mcX}(\mcY) & := \mbox{sup}\{ \pd_{\mcX}(Y) \text{ : } Y \in \mcY \}.
\end{align*}
Dually, we denote by $\id_{\mcX}(M)$ and $\id_{\mcX}(\mcY)$ the \emph{relative injective dimension of $M$} and \emph{$\mcY$}, respectively, \emph{with respect to $\mcX$}. It can be seen that $\pd_{\mcX}(\mcY) = \id_{\mcY}(\mcX)$. If $\mcX = \mcC$, we just write $\pd(M)$, $\pd(\mcY)$, $\id(M)$ and $\id(\mcY)$, for the (absolute) projective and injective dimensions.


\subsection*{Resolving and coresolving classes}

Let $\mcP$ and $\mcI$ denote the classes of projective and injective objects in $\mcC$, respectively. It is said that a class $\mcX$ is \emph{resolving} if $\mcP \subseteq \mcX$ and if it is closed under extensions and under epi-kernels (that is, under taking kernels of epimorphisms between objects in $\mcX$). If the dual properties hold true, then we get \emph{coresolving} classes. A class is \emph{left thick} if it is closed under extensions, epi-kernels and direct summands. \emph{Right thick classes} are defined dually. Finally, a class is \emph{thick} if it is both left and right thick. For a class $\mcX$ of objects in $\mcC$, we shall denote by $\Thick(\mcX)$ the smallest thick subcategory of $\mcC$ containing $\mcX$.


\subsection*{Precovers and preenvelopes}

Let $\mcX$ be a class of objects in $\mcC$ and $C \in \mcC$ be an object. An \emph{$\mcX$-precover of $C$} is a morphism $f \colon X \to C$ with $X \in \mcX$ such that the induced group homomorphism $\Hom_{\mcC}(X',f) \colon \Hom_{\mcC}(X',X) \to \Hom_{\mcC}(X',C)$ is epic for every $X' \in \mcX$. An $\mcX$-precover $f \colon X \to C$ is \emph{special} if it is epic and $\Ker(f) \in \mcX^{\perp_1}$. The dual concept is called (\emph{special}) \emph{$\mcX$-preenvelope}.


\subsection*{Cotorsion pairs}

Two classes $\mcX$ and $\mcY$ of objects in $\mcC$ form a cotorsion pair $(\mcX,\mcY)$ if they are complete with respect to the orthogonality relation defined by the vanishing of the functor $\Ext^1_{\mcC}(-,-)$ (see for instance \cite{EJ1,EJ2,GillespieFlat,GT}). For the purpose of this article, it comes handy to split this concept as follows.

\begin{definition}\label{def:cotorsion_pair}
Let $\mcX$ and $\mcY$ be two classes of objects in $\mcC$. The pair $(\mcX,\mcY)$ is a \textbf{left cotorsion pair} in $\mcC$ if $\mcX = {}^{\perp_1}\mcY$. If in addition, for every $C \in \mcC$ there exists a short exact sequence $0 \to Y \to X \to C \to 0$ with $X \in \mcX$ and $Y \in \mcY$, then $(\mcX,\mcY)$ is a \textbf{complete left cotorsion pair}. Dually, we have the notions of (\textbf{complete}) \textbf{right cotorsion pairs} in $\mcC$. Finally, $(\mcX,\mcY)$ is a (\textbf{complete}) \textbf{cotorsion pair} in $\mcC$ if it is both a (complete) left and right cotorsion pair. A pair $(\mathcal{X,Y})$ is called \textbf{hereditary} if $\Ext^i_{\mcC}(\mathcal{X,Y}) = 0$ for every $i \geq 1$. 
\end{definition}

\begin{remark}\label{rem:cotorsion_pair} \
\begin{enumerate}
\item If $(\mcX,\mcY)$ is a complete left cotorsion pair in $\mcC$, then every object of $\mcC$ has a special $\mcX$-precover. 

\item $(\mcP,\mcC)$ is a complete left cotorsion pair if, and only if, $\mcC$ has enough projective objects. Dually, $(\mcC,\mcI)$ is a complete right cotorsion pair if, and only if, $\mcC$ has enough injective objects. 

\item If $(\mcX,\mcY)$ is a hereditary cotorsion pair, then $\mcX$ is resolving and $\mcY$ is coresolving. Moreover, in an abelian category $\mcC$ with enough projective (resp., injective) objects, $(\mcX,\mcY)$ is hereditary if, and only if, $\mcX$ is resolving (resp., $\mcY$ is coresolving).
\end{enumerate}
\end{remark}

\begin{example}
There are some well-known important examples of hereditary complete cotorsion pairs:
\begin{enumerate}
\item The \textbf{flat} or \textbf{Enochs' cotorsion pair} in $\Mod(R)$ given by $(\mcF(R),(\mcF(R))^{\perp_1})$, where $\mcF(R)$ denotes the class of all flat $R$-modules. 

\item From \cite[Corollary 4.2]{EE}, \cite[Lemma 4.25]{SaorinStovicek} and \cite[Lemma A.1]{EfimovPositselski}, we have the non-affine version of the previous example, for quasi-compact and semi-separated schemes $X$, given by the pair $(\mathfrak{F}(X),(\mathfrak{F}(X))^{\perp_1})$ in $\Qcoh(X)$, where $\mathfrak{F}(X)$ denotes the class of quasi-coherent flat sheaves over $X$. 

\item From \cite{EJ1}, if $R$ is an Iwanaga-Gorenstein ring, we have the pairs $(\mcGP(R),\mcP(R)^\wedge)$ and $(\mcI(R)^\wedge,\mcGI(R))$ in $\Mod(R)$, where $\mcGP(R)$ and $\mcGI(R)$ denote the classes of Gorenstein projective and Gorenstein injective $R$-modules.

Similar assertions hold for the classes $\mcDP(R)$ and $\mcDI(R)$ of Ding projective and Ding injective $R$-modules, provided that $R$ is a Ding-Chen ring (see Gillespie's \cite{GillespieDC} for details).
\end{enumerate}
\end{example}

In this article we shall propose a relative version for Definition \ref{def:cotorsion_pair}, and so it is important that the reader is familiar with the notion and properties of hereditary complete cotorsion pairs in abelian categories.


\subsection*{Frobenius pairs}

The concept of left and right Frobenius pairs was introduced in \cite[Definition 2.5]{BMPS} from the notion of (co)generators in Auslander-Buchweitz approximation theory. Given two classes $\mcX$ and $\omega$ of objects in $\mcC$, recall that $\omega$ is said to be a \emph{relative cogenerator in $\mcX$} if $\omega \subseteq \mcX$ and if for every $X \in \mcX$ there exists a short exact sequence $0 \to X \to W \to X' \to 0$  where $W \in \omega$ and $X' \in \mcX$.

\begin{definition}\label{def:Frobenius_pair}
A pair $(\mcX,\omega)$ of classes of objects in $\mcC$ is a \textbf{left Frobenius pair} if the following conditions hold true:
\begin{enumerate}
\item[\lfpone] $\mcX$ is left thick.

\item[\lfptwo] $\omega$ is closed under direct summands.

\item[\lfpthree] $\omega$ is an $\mcX$-injective (that is, $\id_{\mcX}(\omega) = 0$) relative cogenerator in $\mcX$. 
\end{enumerate}
The notions of \textbf{relative generator} and \textbf{right Frobenius pair} in $\mcC$ are dual.
\end{definition}


\subsection*{Relative Gorenstein objects}

Most of our examples in this article will be built from Gorenstein objects relative to certain pairs $(\mcX,\mcY)$ of classes of objects in $\mcC$ (see Definition \ref{def:GP-admissible}). Before specifying how these Gorenstein objects are defined, recall that a chain complex $X_\bullet = (X_m)_{m \in \mbZ} \in \Ch(\mcC)$ is said to be \emph{$\Hom_{\mcC}(-,\mcY)$-acyclic} if the induced complex of abelian groups $\Hom_{\mcC}(X_\bullet,Y) = (\Hom_{\mcC}(X_m,Y))_{m \in \mbZ}$ is exact for every $Y \in \mcY$. \emph{$\Hom_{\mcC}(\mcY,-)$-acyclic complexes} are defined dually. The following concept is due to \cite[Definition 3.2]{BMS}.

\begin{definition}\label{def:relative_Gorenstein_objects}
Let $(\mcX,\mcY)$ be a pair of classes of objects in $\mcC$. An object $C \in \mcC$ is \textbf{$\bm{(\mcX,\mcY)}$-Gorenstein projective} if $C$ is the 0-th cycle of an exact and $\Hom_{\mcC}(-,\mcY)$-acyclic complex $X_\bullet \in \Ch(\mcC)$ where $X_m \in \mcX$ for every $m \in \mbZ$. In this case, we write $C = Z_0(X_\bullet)$. Dually \textbf{$\bm{(\mcX,\mcY)}$-Gorenstein injective objects} are defined as 0-cycles of exact and $\Hom_{\mcC}(\mcX,-)$-acyclic complexes with components in $\mcY$.
\end{definition}

Following \cite{BMS}, let us denote by $\mcGP_{(\mcX,\mcY)}$ and $\mcGI_{(\mcX,\mcY)}$ the classes of $(\mcX,\mcY)$-Goren-stein projective and $(\mcX,\mcY)$-Gorenstein injective objects of $\mcC$, respectively. For example, $\mcGP_{(\mcP,\mcP)}$ and $\mcGI_{(\mcI,\mcI)}$ are precisely the classes of Gorenstein projective and Gorenstein injective objects of $\mcC$, which we shall write as $\mcGP$ and $\mcGI$, for simplicity. Moreover, Definition \ref{def:relative_Gorenstein_objects} also covers the following examples of relative Gorenstein projective and injective objects:
\begin{itemize}
\item Ding projective and Ding injective modules, in the sense of \cite[Definitions 3.2 and 3.7]{GillespieDC}, by setting the pairs $(\mcX,\mcY) = (\mcP(R),\mcF(R))$ and $(\mcX,\mcY) = (\text{FP-}\mcI(R),\mcI(R))$, respectively. Here, $\text{FP-}\mcI(R)$ stands for the class of FP-injective (or absolutely pure) $R$-modules. 

\item Gorenstein AC-projective and Gorenstein AC-injective modules, in the sen-se of \cite[Sections 5 and 8]{BGH}, by setting the pairs $(\mcX,\mcY) = (\mcP(R),\mcL(R))$ and $(\mcX,\mcY) = (\text{FP}_\infty\text{-}\mcI(R),\mcI(R))$, respectively. Here, $\mcL(R)$ and $\text{FP}_\infty\text{-}\mcI(R)$ denote the classes of level and $\text{FP}_\infty$-injective (or absolutely clean) $R$-modules (see \cite[Definition 2.6]{BGH}). These classes of relative Gorenstein modules will be denoted by $\mcGP_{\text{AC}}(R)$ and $\mcGI_{\text{AC}}(R)$, for simplicity. 

\item Gorenstein flat sheaves over a noetherian and semi-separated scheme $X$, by setting $(\mcX,\mcY) = (\mathfrak{F}(X),\mathfrak{F}(X) \cap (\mathfrak{F}(X))^{\perp_1})$. See Murfet and Salarian's \cite[Theorem 4.18]{MurfetSalarian}. The class of Gorenstein flat sheaves over $X$ will be denoted by $\mathfrak{GF}(X)$. In particular, the latter holds in the affine case $X = \text{Spec}(R)$ provided that $R$ is a commutative noetherian ring. 
\end{itemize}

Many useful properties of Gorenstein objects relative to $(\mcX,\mcY)$ are obtained in the case where $(\mcX,\mcY)$ is a GP-admissible or a GI-admissible pair \cite[Definitions 3.1 and 3.6]{BMS}. We recall this notion for further referring.

\begin{definition}\label{def:GP-admissible}
A pair $(\mcX,\mcY)$ of classes of objects in $\mcC$ is \textbf{GP-admissible} if the following conditions are satisfied:
\begin{enumerate}
\item[\GPaone] $\pd_{\mcY}(\mcX) = 0$.

\item[\GPatwo] $\mcC$ has enough $\mcX$-objects, that is, for every object $C \in \mcC$ there exists an epimorphism $X \twoheadrightarrow C$ with $X \in \mcX$. 

\item[\GPathree] $\mcX$ and $\mcY$ are closed under finite coproducts, and $\mcX$ is closed under extensions. 

\item[\GPafour] $\mcX \cap \mcY$ is a relative cogenerator in $\mcX$. 
\end{enumerate}
A pair $(\mcX,\mcY)$ satisfying the dual conditions is called \textbf{GI-admissible}.
\end{definition}

\begin{example}\label{ex:GP_admissible_pairs} \
\begin{enumerate}
\item Every hereditary complete cotorsion pair $(\mcX,\mcY)$ is a GP-admissible pair, and also induces the GP-admissible pair $(\mcX,\mcX \cap \mcY)$.

\item The pairs $(\mcP(R),\mcF(R))$, $(\mcP(R),\mcL(R))$ and $(\mathfrak{F}(X),(\mathfrak{F}(X))^{\perp_1})$ are GP-admissible for any ring $R$ and any scheme $X$. Dually, the pairs $(\text{FP-}\mcI(R),\mcI(R))$ and $(\text{FP}_\infty\text{-}\mcI(R),\mcI(R))$ are clearly GI-admissible. 
\end{enumerate}
\end{example}


\section{\textbf{Complete cut cotorsion pairs and cotorsion cuts}}\label{sec:cut_cotorsion}

In this section we present a relative notion of cotorsion pair, which depends on subcategories of $\mcC$. In our context, ``relative'' will mean that for two classes of objects $\mcA$ and $\mcB$ in $\mcC$, we can regard the pair $(\mcA,\mcB)$ as a cotorsion pair \emph{cut along} a subcategory $\mcS \subseteq \mcC$. This concept will cover left and right cotorsion pairs in Definition \ref{def:cotorsion_pair} as particular cases. We shall also see how some well-known properties of complete cotorsion pairs are transferred to the relative context resulting from the following definition.

\begin{definition}\label{def:cut_cotorsion_pair}
Let $\mcS$, $\mcA$ and $\mcB$ be classes of objects in $\mcC$. We say that $(\mcA,\mcB)$ is a \textbf{left cotorsion pair cut along $\bm{\mcS}$} if the following conditions are satisfied:
\begin{enumerate}
\item[\lccpone] $\mcA$ is closed under direct summands.

\item[\lccptwo] $\mcA \cap \mcS = {}^{\perp_1}\mcB \cap \mcS$.
\end{enumerate}
A left cotorsion pair $(\mcA,\mcB)$ cut along $\mcS$ is \textbf{complete} if in addition the following holds:
\begin{enumerate}
\item[\lccpthree] For every $S \in \mcS$, there exists an exact sequence $0 \to B \to A \to S \to 0$ such that $A \in \mcA$ and $B \in \mcB$. 
\end{enumerate}
Dually, we say that $(\mcA,\mcB)$ is a (\textbf{complete}) \textbf{right cotorsion pair cut along $\bm{\mcS}$} if it satisfies the dual conditions, labeled as \rccpone, \rccptwo \ (and \rccpthree). Finally, $(\mcA,\mcB)$ is a (\textbf{complete}) \textbf{cotorsion pair cut along $\bm{\mcS}$} if it is both a (complete) left and right cotorsion pair cut along $\mcS$. 

In case there is no need to refer to the class $\mcS$, we shall simply say that $(\mcA,\mcB)$ is a (complete) left and/or right \textbf{cut cotorsion pair}. 

If $(\mcA,\mcB)$ is a complete (left or right) cotorsion pair cut along $\mcS$, we may sometimes refer to $\mcS$ as a (\textbf{left or right}) \textbf{cotorsion cut for $\bm{(\mcA,\mcB)}$}. If $\mcA$ is a class of objects in $\mcC$ closed under direct summands, we shall denote by ${\rm lCuts}(\mcA,\mcB)$ the class of left cotorsion cuts for $(\mcA,\mcB)$. Similarly, we shall denote by ${\rm rCuts}(\mcA,\mcB)$ and ${\rm Cuts}(\mcA,\mcB)$ the classes of  right cotorsion cuts and cotorsion cuts for $(\mcA,\mcB)$, respectively, provided that $\mcA$ and $\mcB$ are closed under direct summands.  
\end{definition}

\begin{remark} \
\begin{enumerate}
\item Notice that the previous definition coincides with Definition \ref{def:cotorsion_pair} by taking $\mcS = \mcC$.  Furthermore, in case $\mcA$ and $\mcB$ are subclasses of a thick subcategory $\mcS \subseteq \mcC$, $(\mcA,\mcB)$ is a complete left cotorsion pair cut along $\mcS$ if, and only if, $(\mcA,\mcB)$ is a left $\mcS$-cotorsion pair in the sense of \cite[Definition 3.4]{BMPS}. This implies that several of our results proved below will recover some facts from the theory of (relative) cotorsion pairs \cite{BMPS}.

\item Let $\mcA$ be a class of objects in $\mcC$ closed under direct summands. It is possible to find another class $\mcB$ of objects in $\mcC$ such that ${\rm lCuts}(\mcA,\mcB) = \emptyset$. Consider for instance $\mcA$ the class of all objects in $\mcC$ isomorphic to $0$, and $\mcB := \mcC - \mcA$ the class of nonzero objects in $\mcC$. Notice that there is no class $\mcS \subseteq \mcC$ satisfying condition \lccpthree. Dually, one can find a pair $(\mcA,\mcB)$ with $\mcB$ closed under direct summands for which ${\rm rCuts}(\mcA,\mcB) = \emptyset$. Nevertheless, for any two classes $\mcA$ and $\mcB$ of objects in $\mcC$ closed under direct summands, one has ${\rm Cuts}(\mcA,\mcB) \neq \emptyset$. Indeed, if $\mcA$ is a class closed under direct summands, a sufficient condition to have ${\rm lCuts}(\mcA,\mcB) \neq \emptyset$ is that $\mcB$ is a pointed subcategory of $\mcC$ (that is, $0 \in \mcB$). Similarly, ${\rm rCuts}(\mcA,\mcB) \neq \emptyset$ if $\mcA$ is pointed and $\mcB$ is closed under direct summands. It suffices to take $\mcS := \{ 0 \}$. 
\end{enumerate}
\end{remark}

Now let us give some examples of complete cut cotorsion pairs which are not necessarily complete cotorsion pairs. Frobenius pairs and relative Gorenstein objects will be the main source to construct our first examples. More complicated examples will be displayed and detalied in Section \ref{sec:applications}.

\begin{proposition}\label{prop:cuts_from_Frobenius_pairs}
Let $(\mcX,\omega)$ be a left Frobenius pair in $\mcC$. The following assertions hold:
\begin{enumerate}
\item $(\mcX,\omega^\wedge)$ is a complete cotorsion pair cut along $\mcX^\wedge$ and $\omega^\wedge$.

\item $(\omega,\mcX^{\perp_1})$ is a complete cotorsion pair cut along $\omega^\wedge$.

\item $(\omega,\mcX^{\perp_1})$ is a complete left cotorsion pair cut along $\mcX^\wedge$ if, and only if, $\mcX^\wedge = \omega^\wedge$.
\end{enumerate}
\end{proposition}

\begin{proof} 
First, note by \lfpone, \lfptwo \ and \cite[Theorem 2.11 and Proposition 2.13]{BMPS} that $\mcX$, $\omega$ and $\omega^\wedge$ are closed under direct summands. Also, it is clear that the same holds for $\mcX^{\perp_1}$. Moreover, the class $\mcX^\wedge$ is thick by \cite[Theorem 2.11]{BMPS}. 
\begin{enumerate}
\item By the previous comments, we have that the pair $(\mcX,\omega^\wedge)$ satisfies \lccpone \ and \rccpone. Moreover, from \cite[Theorem 2.8]{BMPS} we clearly obtain \lccpthree \ and \rccpthree. Finally, conditions \lccptwo \ and \rccptwo \ follow from \cite[Part 1. of Proposition 2.7]{BMPS}, \lccpone, \rccpone, \lccpthree \ and \rccpthree. Hence, $\mcX^\wedge \in {\rm Cuts}(\mcX,\omega^\wedge)$. The assertion $\omega^\wedge \in {\rm Cuts}(\mcX,\omega^\wedge)$ can be easily deduced from the previous.

\item We already have conditions \lccpone \ and \rccpone \ for $(\omega,\mcX^{\perp_1})$. On the one hand, for every object $C \in \omega^\wedge$ it is clear the existence of a short exact sequence $0 \to K \to W \to C \to 0$, where $W \in \omega$ and $K \in \omega^\wedge$. Since $\omega^\wedge \subseteq \mcX^{\perp_1}$ by \cite[Part 1. of Proposition 2.7]{BMPS}, we have that $K \in \mcX^{\perp_1}$. Thus, \lccpthree \ follows. On the other hand, \rccpthree \ is immediate. 

Regarding \lccptwo, note that the containment $\omega \cap \omega^\wedge \subseteq {}^{\perp_1}(\mcX^{\perp_1}) \cap \omega^\wedge$ follows by \lfpthree. The converse containment follows by \lccpthree \ and the fact that $\omega$ is closed under direct summands. Finally, condition \rccptwo \ follows by \lfpthree \ and \rccpthree. 

\item The implication $(\Leftarrow)$ follows from part 2. For the direct implication, suppose that $\mcX^\wedge \in {\rm lCuts}(\omega,\mcX^{\perp_1})$. Then, for every $C \in \mcX^\wedge$ there exists a short exact sequence $0 \to K \to W \to C \to 0$ with $W \in \omega$ and $K \in \mcX^{\perp_1}$. By \lfpone, \lfpthree \ and \cite[Theorem 2.16]{BMPS}, we can note that $K \in \mcX^{\perp_1} \cap \mcX^\wedge = \omega^\wedge$. It then follows that $C \in \omega^\wedge$.  
\end{enumerate}
\end{proof}

\begin{remark}
Part (3) of Proposition \ref{prop:cuts_from_Frobenius_pairs} suggests that there should exist a left Frobenius pair $(\mcX,\omega)$ in $\mcC$ for which $\mcX^\wedge \not\in {\rm lCuts}(\omega,\mcX^{\perp_1})$. This is for instance the case of the left Frobenius pair $(\mcGP(R),\mcP(R))$ \cite[Proposition 6.1]{BMPS}. Indeed, if $R$ is an Iwanaga-Gorenstein ring with infinite global dimension, and $M \in \mcGP(R) - \mcP(R)$, then it is not possible to construct a short exact sequence $0 \to K \to P \to M \to 0$ with $P$ projective and $K \in \mcGP(R)^{\perp_1}$.
\end{remark}

Recall from \cite[Definition 3.3]{BMS} that, for a pair $(\mcX,\mcY)$ of classes of objects of $\mcC$, the \emph{$(\mcX,\mcY)$-Gorenstein projective dimension} of an object $C \in \mcC$, which we denote by $\Gpd_{(\mcX,\mcY)}(C)$, is defined as the $\mcGP_{(\mcX,\mcY)}$-resolution dimension of $C$:
\[
\Gpd_{(\mcX,\mcY)}(C) := \resdim_{\mcGP_{(\mcX,\mcY)}}(C).
\]
Note that setting $(\mcX,\mcY) = (\mcP(R),\mcP(R))$ and $(\mcX,\mcY) = (\mcP(R),\mcF(R))$ yields the \emph{Gorenstein projective} and the \emph{Ding projective dimensions} of an $R$-module $C$, which we denote by $\Gpd(C)$ and $\Dpd(C)$ for simplicity. The \emph{$(\mcX,\mcY)$-Gorenstein injective}, \emph{Gorenstein injective} and \emph{Ding injective dimensions} $\Gid_{(\mcX,\mcY)}(C)$, $\Gid(C)$ and $\Did(C)$, are defined dually.

\begin{example}\label{ex:GP}
We know from the previous remark that $(\mcGP(R),\mcP(R))$ is a left Frobenius pair over any ring $R$. So it follows by parts (1) and (2) of Proposition \ref{prop:cuts_from_Frobenius_pairs} that $(\mcGP(R),\mcP(R)^\wedge)$ is a complete cotorsion pair cut along $\mcGP(R)^\wedge$ and $\mcP(R)^\wedge$, and that $(\mcP(R),\mcGP(R)^{\perp_1})$ is a complete cotorsion pair cut along $\mcP(R)^\wedge$. Similar results hold for the left Frobenius pairs $(\mcDP(R),\mcP(R))$ and $(\mcGP_{\rm AC}(R),\mcP(R))$ (see \cite[Corollary 6.11 and Proposition 6.12]{BMPS}).
\end{example}

\begin{remark}
There are important differences between the notions of $\mcS$-cotorsion pairs $(\mcA,\mcB)$ \cite[Definition 3.4]{BMPS} and complete cotorsion pairs $(\mcA,\mcB)$ cut along $\mcS$. In the former, $\mcS$ is taken as a thick subcategory of $\mcC$ and $\mcA,\mcB \subseteq \mcS$. The latter containments do not occur for instance in the previous example, since Gorenstein projective $R$-modules may have infinite projective dimension. 

More examples of complete cut cotorsion pairs which are not relative cotorsion pairs are given below in Proposition \ref{prop:GPXY_cotorsion_pair}, Corollary \ref{coro:GPXY_cotorsion_pair} and Example \ref{ex:GPXY_cotorsion_pair}. 
\end{remark}

We shall mention a couple of extra properties for the previous example after showing the following general result.

\begin{proposition}\label{prop:cuts_from_Frobenius_and_GP-admissoble_pairs}
The following hold for every left Frobenius pair $(\mcX,\omega)$ in $\mcC$:
\begin{enumerate}
\item The following conditions are equivalent:
\begin{enumerate}
\item[(a)] $(\mcX,\omega^\wedge)$ is a complete left cotorsion pair in $\mcC$.

\item[(b)] $(\mcX,\omega^\wedge)$ is a complete cotorsion pair in $\mcC$. 

\item[(c)] $\mcC = \mcX^\wedge$.
\end{enumerate}

\item For every $n \geq 0$, $(\mcX,\mcX^\perp)$ is a complete cotorsion pair cut along $\mcX^\wedge_n$.
\end{enumerate}
\end{proposition}

\begin{proof}  
In the first part, let us first assume condition (a). We need to verify that $\omega^\wedge = \mcX^{\perp_1}$ and that for every $C \in \mcC$ there exists a short exact sequence of the form $0 \to C \to H \to X \to 0$ with $H \in \omega^\wedge$ and $X \in \mcX$. For the latter, let $C \in \mcC$ and consider a short exact sequence $0 \to K \to X \to C \to 0$ with $X \in \mcX$ and $K \in \omega^\wedge$. On the other hand, since $\omega$ is a relative cogenerator in $\mcX$, there is a short exact sequence $0 \to X \to W \to X' \to 0$ with $W \in \omega$ and $X' \in \mcX$. Taking the push-out of $W \leftarrow X \to C$ yields the following solid diagram:
\begin{equation}\label{eqsolid}
\parbox{1.75in}{
\begin{tikzpicture}[description/.style={fill=white,inner sep=2pt}] 
\matrix (m) [ampersand replacement=\&, matrix of math nodes, row sep=2.5em, column sep=2.5em, text height=1.25ex, text depth=0.25ex] 
{ 
K \& X \& C \\
K \& W \& H \\
{} \& X' \& X' \\
}; 
\path[->] 
(m-1-2)-- node[pos=0.5] {\footnotesize$\mbox{\bf po}$} (m-2-3) 
; 
\path[>->]
(m-1-1) edge (m-1-2)
(m-2-1) edge (m-2-2)
(m-1-2) edge (m-2-2)
(m-1-3) edge (m-2-3)
;
\path[->>]
(m-1-2) edge (m-1-3)
(m-2-2) edge (m-2-3)
(m-2-2) edge (m-3-2)
(m-2-3) edge (m-3-3)
;
\path[-,font=\scriptsize]
(m-1-1) edge [double, thick, double distance=2pt] (m-2-1)
(m-3-2) edge [double, thick, double distance=2pt] (m-3-3)
;
\end{tikzpicture} 
}
\end{equation} 
Note that $H \in \omega^\wedge$. Then, the right-hand column in \eqref{eqsolid} is the desired short exact sequence for the right completeness of $(\mcX,\omega^\wedge)$. In order to show $\omega^\wedge = \mcX^{\perp_1}$, note that the containment $\omega^\wedge \subseteq \mcX^{\perp_1}$ is clear since $\id_{\mcX}(\omega^\wedge) = \id_{\mcX}(\omega) = 0$. The converse containment follows from the right completeness and the fact that $\omega^\wedge$ is closed under direct summands (see \cite[Proposition 2.13]{BMPS}). Hence, the implication (a) $\Rightarrow$ (b) follows. 

The implication (b) $\Rightarrow$ (c) follows by the left completeness of the cotorsion pair $(\mcX,\omega^\wedge)$ and the containment $\omega^\wedge \subseteq \mcX^\wedge$. On the other hand, by \cite[Theorem 3.6]{BMPS} we know that $(\mcX,\omega^\wedge)$ is a $\mcX^\wedge$-cotorsion pair (that is, a complete cotorsion pair in the exact subcategory $\mcX^\wedge$). Thus, the implication (c) $\Rightarrow$ (a) is clear. 

Now for the assertion $\mcX^\wedge_n \in {\rm Cuts}(\mcX,\mcX^{\perp})$, we already know that $\mcX$ is closed under direct summands. In order to show \lccptwo, note that the containment $\mcX \cap \mcX^\wedge_n \subseteq {}^{\perp_1}(\mcX^\perp) \cap \mcX^\wedge_n$ is clear. For the converse, if we take $C \in {}^{\perp_1}(\mcX^\perp) \cap \mcX^\wedge_n$, then by \cite[Theorem 2.8]{BMPS} there exists a short exact sequence $0 \to K \to X \to C \to 0$ with $X \in \mcX$ and $K \in \omega^\wedge_{n-1}$\footnote{Note that for the case $n = 0$ we simply take $K = 0$.}. Note also that $K \in \mcX^\perp$ since $\id_{\mcX}(\omega^\wedge_{n-1}) = \id_{\mcX}(\omega) = 0$. Then, the previous sequence splits and so $C \in \mcX$. On the other hand, for \rccptwo \ $\mcX^\perp \cap \mcX^\wedge_n = \mcX^{\perp_1} \cap \mcX^\wedge_n$, the containment ($\subseteq$) is clear. Now if $C \in \mcX^{\perp_1} \cap \mcX^\wedge_n$, by \cite[Theorem 2.8]{BMPS} we can find a short exact sequence $0 \to C \to H \to C' \to 0$ where $H \in \omega^\wedge_n$ and $C' \in \mcX$, which is split and so $C$ is a direct summand of $H$. This in turn implies that $C \in \mcX^\perp$. The previous arguments also show \lccpthree \ and \rccpthree. 
\end{proof}

\begin{corollary}\label{coro:cuts_from_Frobenius_and_GP-admissoble_pairs}
The following hold for every GP-admissible pair $(\mcX,\mcY)$ in $\mcC$ with $\omega := \mcX \cap \mcY$ closed under direct summands: 
\begin{enumerate} 
\item The following conditions are equivalent:
\begin{enumerate}
\item[(a)] $(\mcGP_{(\mcX,\mcY)},\omega^\wedge)$ is a complete left cotorsion pair in $\mcC$.

\item[(b)] $(\mcGP_{(\mcX,\mcY)},\omega^\wedge)$ is a complete cotorsion pair in $\mcC$.

\item[(c)] $\mcC = \mcGP^\wedge_{(\mcX,\mcY)}$. 
\end{enumerate}

\item For every $n \geq 0$, $(\mcGP_{(\mcX,\mcY)},(\mcGP_{(\mcX,\mcY)})^\perp)$ is a complete cotorsion pair cut along $(\mcGP_{(\mcX,\mcY)})^\wedge_n$. 
\end{enumerate}
\end{corollary}

\begin{proof}
It follows after applying Proposition \ref{prop:cuts_from_Frobenius_and_GP-admissoble_pairs} to the pair $(\mcGP_{(\mcX,\mcY)},\omega)$, which is left Frobenius by \cite[Corollary 4.10]{BMS}.
\end{proof}

\begin{remark}
Although in all of our examples of GP-admissible pairs $(\mcX,\mcY)$, the class $\omega := \mcX \cap \mcY$ is closed under direct summands, another proof of Corollary \ref{coro:cuts_from_Frobenius_and_GP-admissoble_pairs} can be obtained without assuming this property. Indeed, consider the pair $(\mcGP_{(\mcX,\mcY)},(\mcGP_{(\mcX,\mcY)})^\perp)$. The closure under direct summands of $(\mcGP_{(\mcX,\mcY)})^\perp$ is clear, and the same property holds for $\mcGP_{(\mcX,\mcY)}$ due to \cite[Corollary 3.33]{BMS}. Also, the following containments are clear:
\begin{align*}
\mcGP_{(\mcX,\mcY)} \cap (\mcGP_{(\mcX,\mcY)})^\wedge_n & \subseteq {}^{\perp_1}((\mcGP_{(\mcX,\mcY)})^\perp) \cap (\mcGP_{(\mcX,\mcY)})^\wedge_n, \\
(\mcGP_{(\mcX,\mcY)})^\perp \cap (\mcGP_{(\mcX,\mcY)})^\wedge_n & \subseteq (\mcGP_{(\mcX,\mcY)})^{\perp_1} \cap (\mcGP_{(\mcX,\mcY)})^\wedge_n.
\end{align*} 
On the other hand, by \cite[Corollary 4.3]{BMS} for every $C \in {}^{\perp_1}((\mcGP_{(\mcX,\mcY)})^\perp) \cap (\mcGP_{(\mcX,\mcY)})^\wedge_n$ there exists a short exact sequence $0 \to K \to G \to C \to 0$ with $G \in \mcGP_{(\mcX,\mcY)}$ and $K \in \omega^\wedge_{n-1}$. Also, we can note from \cite[Corollary 3.15]{BMS} that $\omega^\wedge_{n-1} \subseteq (\mcGP_{(\mcX,\mcY)})^{\perp_1}$. Then, \lccptwo \ and \lccpthree \ follow. The containment 
\[
(\mcGP_{(\mcX,\mcY)})^\perp \cap (\mcGP_{(\mcX,\mcY)})^\wedge_n \supseteq (\mcGP_{(\mcX,\mcY)})^{\perp_1} \cap (\mcGP_{(\mcX,\mcY)})^\wedge_n
\] 
follows as in the proof of part (2) of Proposition \ref{prop:cuts_from_Frobenius_and_GP-admissoble_pairs}. Hence, \rccptwo \ follows, and \rccpthree is a consequence of \cite[Corollaries 3.15 and 4.3]{BMS}. 
\end{remark}

\begin{example}\label{ex:GP2} \
\begin{enumerate}
\item From Corollary \ref{coro:cuts_from_Frobenius_and_GP-admissoble_pairs} we can note that it is not always possible to extend a cotorsion cut associated to a pair to the whole category $\mcC$. Indeed, consider the complete cotorsion pair $(\mcGP(R),\mcP(R)^\wedge)$ cut along $\mcGP(R)^\wedge$ from Example \ref{ex:GP}. Then, we have that $(\mcGP(R),\mcP(R)^\wedge)$ is a complete cotorsion pair in $\Mod(R)$ if, and only if, $\Mod(R) = \mcGP(R)^\wedge$. The latter equality occurs, for instance, if $R$ an Iwanaga-Gorenstein ring, but it is not true in general. 

\item We can also characterize when the complete cotorsion pair $(\mcP(R),\mcGP(R)^{\perp_1})$ cut along $\mcP(R)^\wedge$ is a complete cotorsion pair in $\Mod(R)$. Specifically, the pair $(\mcP(R),\mcGP(R)^{\perp_1})$  is a complete cotorsion pair in $\Mod(R)$ if, and only if, $\mcP(R) = \mcGP(R)$. The latter equality occurs, for instance, over any ring with finite global dimension. 
\end{enumerate}
\end{example}

In \cite[Proposition 3.5]{BMPS} it is given an alternative description of relative cotorsion pairs. Following the spirit of this result, we present the following characterization for cut cotorsion pairs. Its proof is straightforward.

\begin{proposition}\label{prop:characterization_ccp}
Let $\mcS$, $\mcA$ and $\mcB$ be classes of objects in $\mcC$. Then, $(\mcA,\mcB)$ is a complete left cotorsion pair cut along $\mcS$ if, and only if, $\mcA$ and $\mcB$ satisfy the following conditions:
\begin{enumerate}
\item $\mcA$ is closed under direct summands;

\item $\Ext^1_{\mcC}(\mcA \cap \mcS,\mcB) = 0$; and

\item for every $S \in \mcS$ there exists an exact sequence $0 \to B \to A \to S \to 0$ with $A \in \mcA$ and $B \in \mcB$.
\end{enumerate}
\end{proposition}

\begin{remark}
Regarding condition (3) in Proposition \ref{prop:characterization_ccp}, in the case $\Ext^1_{\mcC}(\mcA,\mcB) = 0$, the morphism $\varphi \colon A \to S$ is an $\mcA$-precover. 
\end{remark}


\subsection*{Getting new cotorsion cuts and pairs from old ones}

In the following result, we show that the class ${\rm lCuts}(\mcA,\mcB)$ is closed under restrictions, arbitrary unions and  intersections.

\begin{proposition}\label{prop:properties_of_cuts}
Let $\mcA$ and $\mcB$ be two classes of objects in $\mcC$. 
\begin{enumerate}
\item \textbf{Restriction:} If $(\mcA,\mcB)$ is a (complete) left cotorsion pair cut along $\mcS$ and $\mcX \subseteq \mcS$, then $(\mcA,\mcB)$ is a (complete) left cotorsion pair cut along $\mcX$. 
\end{enumerate}
If $\{ \mcS_i \}_{i \in I}$ is a nonempty family of classes of objects in $\mcC$, then the following hold:
\begin{enumerate}
\item[(2)] \textbf{Unions:} $(\mcA,\mcB)$ is a (complete) left cotorsion pair cut along $\mcS := \bigcup_{i \in I} \mcS_i$ if, and only if, $(\mcA,\mcB)$ is a (complete) left cotorsion pair cut along $\mcS_i$, for every $i \in I$. 

\item[(3)] \textbf{Intersections:} If $(\mcA,\mcB)$ is a (complete) left cotorsion pair cut along $\mcS_i$ for every $i \in I$, then $(\mcA,\mcB)$ is a (complete) left cotorsion pair cut along $\mcS := \bigcap_{i \in I} \mcS_i$.
\end{enumerate}
\end{proposition}

\begin{proof} 
Part (1) is a consequence of Definition \ref{def:cut_cotorsion_pair}, and (3) is easy to prove. The ``only if'' part of (2) is a consequence of (1). Now for the ``if'' part, suppose that $(\mcA,\mcB)$ is a complete left cotorsion pair cut along each $\mcS_i$. The validity on \lccpone \ is independent of the cuts $\mcS_i$, and \lccpthree \ clearly holds for $\mcS$. Finally, \lccptwo \ follows by the distributive law of intersections with respect to unions. 
\end{proof}

\begin{remark} \
\begin{enumerate}
\item It is not always true that cotorsion cuts can be extended to a bigger class, as shown in Example \ref{ex:GP2} (1). 

\item The converse of the intersection property does not hold in general. It suffices to consider the pair $(\mcGP(R),\mcP(R)^\wedge)$ from Example \ref{ex:GP} and $\mcS_1 := \mcGP(R)^\wedge$ and $\mcS_2 := \Mod(R)$.

\item By the union property and its dual, we can note that $(\mcA,\mcB)$ is a complete cotorsion pair in $\mcC$ if, and only if, there exists a family $\{ \mcS_i \}_{i \in I}$ of classes of objects in $\mcC$ such that $\mcC = \bigcup_{i \in I} \mcS_i$ and that $(\mcA,\mcB)$ is a complete cotorsion pair cut along $\mcS_i$ for every $i \in I$.
\end{enumerate}
\end{remark}

\begin{example}
Over any ring $R$, $(\mcGP(R),\mcGP(R)^{\perp_1})$ is a complete cotorsion pair cut along $\mcGP(R)^\wedge$. This clearly follows by Corollary \ref{coro:cuts_from_Frobenius_and_GP-admissoble_pairs} and by the union property. 
\end{example}


\subsection*{Maximal cotorsion cuts} 

From the union property it is natural to think of the possibility of finding the largest cotorsion cut for a pair $(\mcA,\mcB)$. Indeed, assuming that $\mcA$ is closed under direct summands and that $0 \in \mcB$ (and so ${\rm lCuts}(\mcA,\mcB) \neq \emptyset$), it is possible to define the biggest cut for $(\mcA,\mcB)$ as a certain union.

\begin{definition}\label{def:maximal_cut}
Let $\mcA$ and $\mcB$ be two classes of objects in $\mcC$ such that $\mcA$ is closed under direct summands and $0 \in \mcB$. The \textbf{maximal left cotorsion cut of $\bm{(\mcA,\mcB)}$} is the union
\[
\mbS_l(\mcA,\mcB) := \bigcup \{ \mcS \text{ : } \mcS \in {\rm lCuts}(\mcA,\mcB) \}.
\]
The \textbf{maximal right cotorsion cut} and the \textbf{maximal cotorsion cut of $\bm{(\mcA,\mcB)}$} are defined similarly, and will be denoted by $\mbS_r(\mcA,\mcB)$ and $\mbS(\mcA,\mcB)$.  
\end{definition}

The maximal left cotorsion cut of $(\mcA,\mcB)$ has the interesting property that, under some mild conditions, it can cover the class of all objects satisfying \lccpthree. Let us denote the latter class by $\mcE_l(\mcA,\mcB)$, that is, $\mcE_l(\mcA,\mcB)$ is the class of all objects $C \in \mcC$ for which there exists a short exact sequence $0 \to B \to A \to C \to 0$ with $A \in \mcA$ and $B \in \mcB$. We can note that $\mbS_l(\mcA,\mcB) \subseteq \mcE_l(\mcA,\mcB)$, although the equality does not hold in general. Below we show that if $\mcE_l(\mcA,\mcB)$ is a left cotorsion cut of $(\mcA,\mcB)$, then it has to be the maximal one.

\begin{theorem}\label{theo:maximal_S}
The following conditions are equivalent for any two classes $\mcA$ and $\mcB$ of objects in $\mcC$, where $\mcA$ is closed under direct summands and $0 \in \mcB$:
\begin{enumerate}
\item[(a)] $\Ext^1_{\mcC}(\mcA,\mcB) = 0$.

\item[(b)] $\mbS_l(\mcA,\mcB)  = \mcE_l(\mcA,\mcB)$.

\item[(c)] $\mcE_l(\mcA,\mcB) \in {\rm lCuts}(\mcA,\mcB)$.

\item[(d)] $\mcA = {}^{\perp_1}\mcB \cap \mcE_l(\mcA,\mcB)$. 
\end{enumerate}
\end{theorem}

\begin{proof}
Note first that $\mcA \subseteq \mcE_l(\mcA,\mcB)$ since $0 \in \mcB$. Then, $\mcA \cap \mcE_l(\mcA,\mcB) = \mcA$, and so the implication (c) $\Rightarrow$ (d) is clear. On the other hand, the implications (b) $\Rightarrow$ (c) and (d) $\Rightarrow$ (a) are trivial. Thus, we only focus on proving that (a) $\Rightarrow$ (b). 

Suppose that $\Ext^1_{\mcC}(\mcA,\mcB) = 0$. Note that the containment $\mbS_l(\mcA,\mcB) \subseteq \mcE_l(\mcA,\mcB)$ is clear. Now let $X \in \mcE_l(\mcA,\mcB)$. We prove that $(\mcA,\mcB)$ is a complete left cotorsion pair cut along $\mcS := \mbS_l(\mcA,\mcB) \cup \{ X \}$. For this, we only need to prove \lccptwo. Consider the following two cases:
\begin{enumerate}
\item $X \in \mcA$: Since $\Ext^1_{\mcC}(\mcA,\mcB) = 0$ we have that $X \in {}^{\perp_1}\mcB$. Thus, 
\[
\mcA \cap \{ X \} = \{ X \} = {}^{\perp_1}\mcB \cap \{ X \}.
\] 

\item $X \not\in \mcA$: Since $X \in \mcE_l(\mcA,\mcB)$, there exists a non-split short exact sequence $0 \to B \to A \to X \to 0$ with $A \in \mcA$ and $B \in \mcB$. It follows that $X \not\in {}^{\perp_1}\mcB$, and so $\mcA \cap \{ X \} = \emptyset = {}^{\perp_1}\mcB \cap \{ X \}$. 
\end{enumerate}
In both cases, we get the equality $\mcA \cap \{ X \} = {}^{\perp_1}\mcB \cap \{ X \}$. Therefore, $\mbS_l(\mcA,\mcB) \cup \{ X \} \in {\rm lCuts}(\mcA,\mcB)$, and so $\mbS_l(\mcA,\mcB) \cup \{ X \} = \mbS_l(\mcA,\mcB)$ by maximality, proving that $\mbS_l(\mcA,\mcB) \supseteq \mcE_l(\mcA,\mcB)$. 
\end{proof}

A similar equivalence holds for the class $\mcE_r(\mcA,\mcB)$ of all objects satisfying \rccpthree, and for $\mcE(\mcA,\mcB) := \mcE_l(\mcA,\mcB) \cap \mcE_r(\mcA,\mcB)$.

\begin{remark}
As we mentioned earlier, $\mbS_l(\mcA,\mcB) \subseteq \mcE_l(\mcA,\mcB)$. In some cases this containment is strict and nontrivial, that is, we can find classes $\mcA$ and $\mcB$ such that $\{ 0 \} \subsetneq \mbS_l(\mcA,\mcB) \subsetneq \mcE_l(\mcA,\mcB)$. Indeed, let us consider $\mcA = \mcGP(R)$ and $\mcB = \Mod(R)$. One can note that $(\mcGP(R),\Mod(R))$ is a complete left cotorsion pair cut along $\mcP(R)^\wedge$, and so $\mcP(R)^\wedge \subseteq \mbS_l(\mcGP(R),\Mod(R))$, that is, $\mbS_l(\mcGP(R),\Mod(R)) \neq \{ 0 \}$. On the other hand, by Theorem \ref{theo:maximal_S} we have that 
\[
\mbS_l(\mcGP(R),\Mod(R)) = \mcE_l(\mcGP(R),\Mod(R)) \Longleftrightarrow \Ext^1_R(\mcGP(R),\Mod(R)) = 0,
\]  
and there are rings over which the latter condition does not hold (see Example \ref{ex:GP2} (2)). 
\end{remark}


\subsection*{Compatibility between cotorsion cuts} 

So far the methods we have showed to obtain new cotorsion cuts are restricted to a fixed pair $(\mcA,\mcB)$ of classes of objects of $\mcC$. In some cases it is possible to get new pairs along new cuts. More specifically, we show in Proposition \ref{prop:unions_cuts_and_pairs} below an extension of the union property of Proposition \ref{prop:properties_of_cuts} (2), in the sense that it is possible to take the union of two different complete cut cotorsion pairs along the union of their cuts, provided that certain compatibility condition between the given pairs is satisfied.

\begin{definition}\label{def:compatible_pairs}
Let $\mcA_1$, $\mcA_2$, $\mcB_1$ and $\mcB_2$ be classes of objects in $\mcC$, where $\mcA_1$ and $\mcA_2$ are closed under direct summands and $0 \in \mcB_1 \cap \mcB_2$, and let $\mcS_1 \in {\rm lCuts}(\mcA_1,\mcB_1)$ and $\mcS_2 \in {\rm lCuts}(\mcA_2,\mcB_2)$ be cotorsion cuts in $\mcC$. We shall say that $(\mcA_1,\mcB_1)$ and $(\mcA_2,\mcB_2)$ are \textbf{compatible} if the following two conditions hold:
\begin{enumerate}
\item $\Ext^1_{\mcC}(\mcA_1,\mcB_2) = 0$ and $\Ext^1_{\mcC}(\mcA_2,\mcB_1) = 0$.

\item $\mcA_1 \cap \mcS_2 = \mcA_2 \cap \mcS_1$. 
\end{enumerate}
\textbf{Compatible complete right cut cotorsion pairs} and \textbf{compatible complete cut cotorsion pairs} are defined similarly. 
\end{definition}

Considering relations between complete cut cotorsion pairs will be important in Section \ref{sec:versus} to establish correspondences between these pairs and the relative versions of the notions of Frobenius pairs and Auslander-Buchweitz contexts developed later in Section \ref{sec:cut_Frobenius_and_cut_AB}.

\begin{proposition}\label{prop:unions_cuts_and_pairs}
Let $\mcA_1$, $\mcA_2$, $\mcB_1$, $\mcB_2$, $\mcS_1$ and $\mcS_2$ be as in Definition \ref{def:compatible_pairs}. If $(\mcA_1,\mcB_1)$ and $(\mcA_2,\mcB_2)$ are compatible, then $\mcS_1 \cup \mcS_2 \in {\rm lCuts}(\mcA_1 \cup \mcA_2, \mcB_1 \cup \mcB_2)$.
\end{proposition}

\begin{proof}
It is clear that conditions \lccpone \ and \lccpthree \ hold for $(\mcA_1 \cup \mcA_2, \mcB_1 \cup \mcB_2)$ and $\mcS_1 \cup \mcS_2$. Then, it suffices to show that the following equality holds: 
\[
(\mcA_1 \cup \mcA_2) \cap (\mcS_1 \cup \mcS_2) = {}^{\perp_1}(\mcB_1 \cup \mcB_2) \cap (\mcS_1 \cup \mcS_2).
\]
\begin{enumerate}
\item[($\supseteq$)] Let $C \in {}^{\perp_1}(\mcB_1 \cup \mcB_2) \cap (\mcS_1 \cup \mcS_2)$. Note that $C \in {}^{\perp_1}\mcB_1 \cap {}^{\perp_1}\mcB_2$ and that $C \in \mcS_i$ for some $i = 1,2$. Then, $C \in {}^{\perp_1}\mcB_i \cap \mcS_i = \mcA_i \cap \mcS_i \subseteq (\mcA_1 \cup \mcA_2) \cap (\mcS_1 \cup \mcS_2)$. 

\item[($\subseteq$)] Let $C \in (\mcA_1 \cup \mcA_2) \cap (\mcS_1 \cup \mcS_2)$. Then, $C \in \mcA_i$ for some $i = 1,2$, and $C \in \mcS_j$ for some $j = 1,2$. 
\begin{itemize}
\item If $i = j$, then $C \in \mcA_i \cap \mcS_i = {}^{\perp_1}\mcB_i \cap \mcS_i$. By condition (1) in Definition \ref{def:compatible_pairs}, we have that $C \in {}^{\perp_1}\mcB_1 \cap {}^{\perp_1}\mcB_2 = {}^{\perp_1}(\mcB_1 \cup \mcB_2)$. It follows that $C \in {}^{\perp_1}(\mcB_1 \cup \mcB_2) \cap (\mcS_1 \cup \mcS_2)$.

\item If $i \neq j$, then $C \in \mcA_i \cap \mcS_j$ implies that $C \in \mcA_j \cap \mcS_i$ by condition (2) in Definition \ref{def:compatible_pairs}. Thus, $C \in \mcA_i \cap \mcS_i = {}^{\perp_1}\mcB_i \cap \mcS_i$ for $i = 1, 2$, and so the containment $(\subseteq)$ follows in this case as well. 
\end{itemize}
\end{enumerate} 
\end{proof}

\begin{example}
The complete cotorsion pairs $(\mcGP(R),\mcP(R)^\wedge)$ and $(\mcP(R),\mcGP(R)^{\perp_1})$ cut along $\mcGP(R)^\wedge$ and $\mcP(R)^\wedge$, respectively, are compatible. Indeed, conditions (1) and (2) of Definition \ref{def:compatible_pairs} are clear by \cite[Proposition 10.2.3]{EJ1}, and the dual of (2), that is, the equality $\mcGP(R)^{\perp_1} \cap \mcGP(R)^\wedge = \mcP(R)^\wedge$ follows by \cite[Theorem 2.8]{BMPS}. 
\end{example}


\subsection*{More induced cotorsion cuts and examples}

Previously we showed how to construct new cotorsion cuts from a given one, or from a family of cotorsion cuts (as in Propositions \ref{prop:properties_of_cuts} and \ref{prop:unions_cuts_and_pairs}). In the last part of this section, we give some sufficient conditions on three classes $\mcA, \mcB, \mcS \subseteq \mcC$, without needing that $\mcS \in {\rm Cuts}(\mcA,\mcB)$, which imply that $(\mcA,\mcB)$ is a cotorsion pair cut along $\mcA^\wedge \cap \mcS$. The classes $\mcE_l(\mcA,\mcB)$, $\mcE_r(\mcA,\mcB)$ and $\mcE(\mcA,\mcB)$ will be useful to prove the following result.

\begin{proposition}\label{prop:a_new_cotorsion_cut}
Let $\mcS$, $\mcA$ and $\mcB$ be classes of objects in $\mcC$, and let $\omega := \mcA \cap \mcB$, such that the following conditions are satisfied:
\begin{enumerate}
\item $\mcA$ is closed under extensions and direct summands;

\item $\mcB$ is closed under direct summands;

\item $\omega \cap \mcS$ is a relative cogenerator in $\mcA$;

\item $(\omega \cap \mcS)^\wedge \subseteq \mcB$; 

\item $\Ext^1_{\mcC}(\mcA \cap \mcS,\mcB) = 0$ and $\Ext^1_{\mcC}(\mcA,\mcB \cap \mcS) = 0$.
\end{enumerate}
Then, $\mcA^\wedge \cap \mcS \in {\rm Cuts}(\mcA,\mcB)$. 
\end{proposition}

\begin{proof}
First, we show the following containment:
\begin{align}
\mcA^{\wedge} \subseteq \mcE_r(\mcA, (\omega \cap \mcS)^{\wedge}). \label{eq0}
\end{align}
Indeed, let $M \in \mcA^{\wedge}$. We proceed by induction on $n := \resdim_{\mcA}(M)$. 
\begin{itemize}
\item \underline{Initial step}: If $n = 0$, the containment \eqref{eq0} follows by condition (3). 

\item \underline{Induction step}: Suppose that $n \geq 1$ and that the containment \eqref{eq0} holds for any object in $\mcA^\wedge$ with $\mcA$-resolution dimension at most $n-1$. Now consider a short exact sequence
\begin{align}
0 \to L \to A \to M \to 0, \label{eq1}
\end{align}
with $A \in \mcA$ and $\resdim_{\mcA}(L) = n-1$. By the induction hypothesis, there is a short exact sequence
\begin{align}
0 \to L \to K \to A' \to 0, \label{eq2}
\end{align}
with $A' \in \mcA$ and $K \in (\omega \cap \mcS)^\wedge$. Taking the pushout of $K \leftarrow L \to A$ from \eqref{eq1} and \eqref{eq2} yields the following solid diagram:
\begin{equation}\label{eq3}
\parbox{1.75in}{
\begin{tikzpicture}[description/.style={fill=white,inner sep=2pt}] 
\matrix (m) [ampersand replacement=\&, matrix of math nodes, row sep=2.5em, column sep=2.5em, text height=1.25ex, text depth=0.25ex] 
{ 
L \& A \& M \\
K \& E \& M \\
A' \& A' \& {} \\
}; 
\path[->] 
(m-1-1)-- node[pos=0.5] {\footnotesize$\mbox{\bf po}$} (m-2-2) 
; 
\path[>->]
(m-1-1) edge (m-2-1) (m-1-2) edge (m-2-2)
(m-2-1) edge (m-2-2) (m-1-1) edge (m-1-2)
(m-3-1) edge (m-3-2)
;
\path[->>]
(m-2-1) edge (m-3-1) (m-2-2) edge (m-3-2)
(m-2-2) edge (m-2-3) (m-1-2) edge (m-1-3)
;
\path[-,font=\scriptsize]
(m-1-3) edge [double, thick, double distance=2pt] (m-2-3)
(m-3-1) edge [double, thick, double distance=2pt] (m-3-2)
;
\end{tikzpicture} 
}
\end{equation}  
Note that $E \in \mcA$ by condition (1), and so by condition (3) there exists a short exact sequence $0 \to E \to W \to A'' \to 0$ with $W \in \omega \cap \mcS$ and $A'' \in \mcA$. Now take the pushout of $W \leftarrow E \to M$ to obtain the following solid diagram: 
\begin{equation}\label{eq4} 
\parbox{1.75in}{
\begin{tikzpicture}[description/.style={fill=white,inner sep=2pt}] 
\matrix (m) [ampersand replacement=\&, matrix of math nodes, row sep=2.5em, column sep=2.5em, text height=1.25ex, text depth=0.25ex] 
{ 
K \& E \& M \\
K \& W \& F \\
{} \& A'' \& A'' \\
};  
\path[->] 
(m-1-2)-- node[pos=0.5] {\footnotesize$\mbox{\bf po}$} (m-2-3) 
; 
\path[>->]
(m-1-1) edge (m-2-1) (m-1-2) edge (m-2-2)
(m-2-1) edge (m-2-2) (m-1-1) edge (m-1-2)
(m-1-3) edge (m-2-3)
;
\path[->>]
(m-2-3) edge (m-3-3) (m-2-2) edge (m-2-3)
(m-2-2) edge (m-3-2) (m-1-2) edge (m-1-3)
;
\path[-,font=\scriptsize]
(m-1-1) edge [double, thick, double distance=2pt] (m-2-1)
(m-3-2) edge [double, thick, double distance=2pt] (m-3-3)
;
\end{tikzpicture} 
}
\end{equation} 
We have that $F \in (\omega \cap \mcS)^\wedge$, and so the right-hand column of \eqref{eq4} implies that $M \in \mcE_r(\mcA,(\omega \cap \mcS)^\wedge)$. 
\end{itemize}

The containment \eqref{eq0} then holds true. Moreover, from the central row in \eqref{eq3} and condition (4) we have that for every $M \in \mcA^\wedge$ with $\resdim_{\mcA}(M) \geq 1$ there is a short exact sequence
\begin{align}
0 \to B \to A \to M \to 0 \label{eq5}
\end{align}
with $A \in \mcA$ and $B \in \mcB$. Moreover, for the case $n = 0$ we can simply take $A = M$ and $B = 0$ in \eqref{eq5} since $0 \in \mcB$, as $\mcB$ is closed under direct summands. In other words, we also have that the following containment holds:
\begin{align}
\mcA^\wedge & \subseteq \mcE_l(\mcA,\mcB). \label{eq6}
\end{align}
Hence, from \eqref{eq0}, \eqref{eq6} and condition (4) we conclude that $\mcA^\wedge \subseteq \mcE(\mcA,\mcB)$. It is clear now that the pair $(\mcA,\mcB)$ satisfies conditions \lccpthree \ and \rccpthree \ with respect to the class $\mcA^\wedge \cap \mcS$, and we also know from the hypotheses the validity of \lccpone \ and \rccpone. Finally, by condition (5) we have that $\mcA \cap (\mcA^\wedge \cap \mcS) \subseteq {}^{\perp_1}\mcB \cap (\mcA^\wedge \cap \mcS)$ and $\mcB \cap (\mcA^\wedge \cap \mcS) \subseteq \mcA^{\perp_1} \cap (\mcA^\wedge \cap \mcS)$, and the converse containments follow by \lccpthree \ and \rccpthree. Therefore, $(\mcA,\mcB)$ is a complete cotorsion pair cut along $\mcA^\wedge \cap \mcS$. 
\end{proof}

Let us apply the previous result to obtain another example of a complete cut cotorsion pair from Gorenstein objects relative to a GP-admissible pair.

\begin{proposition}\label{prop:GPXY_cotorsion_pair}
Let $(\mcX,\mcY)$ be a GP-admissible pair in $\mcC$ and $\omega := \mcX \cap \mcY$, such that $\mcY^\wedge$, $\omega$ and $\mcX \cap \mcY^\wedge$ are closed under direct summands. Then, $(\mcGP_{(\mcX,\mcY)},\mcY^\wedge)$ is a complete cotorsion pair cut along $\mcX^\wedge$. Moreover, 
\begin{align}
\mcGP_{(\mcX,\mcY)} \cap \mcY^\wedge & = \omega = \mcX \cap \mcY^\wedge = \mcY \cap \mcGP_{(\mcX,\mcY)}. \label{GPXY}
\end{align}
\end{proposition}

\begin{proof}
We have by hypothesis that $\mcY^\wedge$ is closed under direct summands. On the other hand, by \cite[Corollary 3.33]{BMS} we also have that $\mcGP_{(\mcX,\mcY)}$ is closed under extensions and direct summands. We then have that conditions (1) and (2) in Proposition \ref{prop:a_new_cotorsion_cut} are valid.   

Now let us show that $\mcGP_{(\mcX,\mcY)} \cap \mcY^\wedge \cap \mcX^\wedge$ is a relative cogenerator in $\mcGP_{(\mcX,\mcY)}$. By \cite[Theorem 3.34 (c) and (d)]{BMS} we know that the equality \eqref{GPXY} holds. Then, 
\[
\mcGP_{(\mcX,\mcY)} \cap \mcY^\wedge \cap \mcX^\wedge = \omega \cap \mcX^\wedge = \omega,
\]
and so condition (3) in Proposition \ref{prop:a_new_cotorsion_cut} follows as well by \cite[Corollary 3.25 (a)]{BMS}, while condition (4) is clear. 

Finally, the orthogonality relations in condition (5) of Proposition \ref{prop:a_new_cotorsion_cut} are a consequence of \cite[Corollary 3.15]{BMS}.
\end{proof}

\begin{corollary}\label{coro:GPXY_cotorsion_pair}
Let $(\mcX,\mcY)$ be a hereditary complete cotorsion pair in $\mcC$ and $\omega := \mcX \cap \mcY$. Then, $(\mcGP_{(\mcX,\omega)},\omega^\wedge)$ and $(\mcGP_{(\mcX,\omega)},\mcY)$ are complete cotorsion pairs cut along $\mcX^\wedge$. Moreover, 
\begin{align}
\mcGP_{(\mcX,\omega)} \cap \omega^\wedge & = \omega = \mcX \cap \omega^\wedge. \label{GPXW}
\end{align}
\end{corollary}

\begin{proof}
Recall from Example \ref{ex:GP_admissible_pairs} (1) that $(\mcX,\omega)$ is a GP-admissible pair. Moreover, it is clear that $\mcX$ and $\omega$ are closed under direct summands, while the same holds for $\omega^\wedge$ by \cite[Theorem 2.11 and Proposition 2.13]{BMPS}. Then, $\mcX \cap \omega^\wedge$ is also closed under direct summands, and Proposition \ref{prop:GPXY_cotorsion_pair} implies that $(\mcGP_{(\mcX,\omega)},\omega^\wedge)$ is a complete cotorsion pair cut along $\mcX^\wedge$ and the equality \eqref{GPXW}. 

For the second pair $(\mcGP_{(\mcX,\omega)},\mcY)$, the classes $\mcA := \mcGP_{(\mcX,\omega)}$, $\mcB := \mcY$ and $\mcS := \mcX^\wedge$ satisfy the conditions in Proposition \ref{prop:a_new_cotorsion_cut}. Indeed, the first four conditions are clear. For (5), it suffices to note that $\mcY \cap \mcX^\wedge = \omega^\wedge$ and $\mcGP_{(\mcX,\omega)} \cap \mcX^\wedge = \mcX$, which follow from the assumptions and \cite[Theorem 2.8]{BMPS}. Finally, for the equality \eqref{GPXW} we have that $\mcGP_{(\mcX,\omega)} \cap \omega^\wedge = \mcGP_{(\mcX,\omega)} \cap \omega^\wedge \cap \mcX^\wedge = \mcX \cap \omega^\wedge = \mcX \cap \mcY \cap \mcX^\wedge = \omega \cap \mcX^\wedge = \omega$.
\end{proof}

\begin{example}\label{ex:GPXY_cotorsion_pair}
For any ring $R$, $(\mcDP(R),\mcF(R)^\wedge)$ is a complete cotorsion pair cut along $\mcP(R)^\wedge$ in $\Mod(R)$. Indeed, we know from Example \ref{ex:GP_admissible_pairs} (2) that the pair $(\mcP(R),\mcF(R))$ is GP-admissible, and $\mcGP_{(\mcP(R),\mcF(R))}$ is precisely the class $\mcDP(R)$ of Ding projective $R$-modules. Moreover, it is clear that $\mcP(R) \cap \mcF(R) = \mcP(R)$, $\mcF(R)^\wedge$ and $\mcP(R) \cap \mcF(R)^\wedge = \mcP(R)$ are closed under direct summands. Then by Proposition \ref{prop:GPXY_cotorsion_pair} we get the desired result. 

Note also that this example cannot by obtained by using Corollary \ref{coro:GPXY_cotorsion_pair} since in general there is no hereditary complete cotorsion pair $(\mcX,\mcY)$ in $\Mod(R)$ such that $\mcX = \mcP(R)$ and $\mcX \cap \mcY = \mcF(R)$. This is possible for example for the trivial cotorsion pair $(\mcP(R),\Mod(R))$ over a left perfect ring $R$. 
\end{example}

\begin{remark}
Several of the examples of relative Gorenstein pairs are cut along the class $\mcGP_{(\mcX,\mcY)}$, but this is not always the case. For instance, the pair $(\mcGP_{(\mcX,\omega)},\mcY)$ from Corollary \ref{coro:GPXY_cotorsion_pair} is another example of a complete cut cotorsion pair that cannot be extended to a bigger class. In this case, one can note that $(\mcGP_{(\mcX,\omega)},\mcY)$ is a complete cotorsion pair cut along $\mcGP_{(\mcX,\omega)}^\wedge$ if, and only if, $\mcGP_{(\mcX,\omega)} = \mcX$. 
\end{remark}

Intersections of the form $\omega \cap \mcS$ considered in Proposition \ref{prop:a_new_cotorsion_cut} are not a mere technicality to obtain new cut cotorsion pairs. They are in fact an important component of the notions of cut Frobenius pairs and cut Auslander-Buchweitz contexts, which will be presented and studied in detail in the next section.


\section{\textbf{Cut Frobenius pairs and cut Auslander-Buchweitz contexts}}\label{sec:cut_Frobenius_and_cut_AB}

This section is devoted to present and study the concepts of cut Frobenius pairs and cut Auslander-Buchweitz contexts. As mentioned in the introduction, one of the main purposes of the present article is to describe an interplay between these two notions and the concept of complete cut cotorsion pairs studied in the previous section. This interplay will be the main topic of the next section. For the moment, we can note the following relation between certain Gorenstein complete cut cotorsion pairs and left Frobenius pairs.

\begin{proposition}\label{proposition:Gorenstein_Frobenius_pair_vs_Gorenstein_cotorsion_cut}
Let $(\mcX,\mcY)$ be a GP-admissible pair in $\mcC$, with $\omega := \mcX \cap \mcY$ closed under direct summands. Then, the following conditions are equivalent:
\begin{enumerate}
\item[(a)] $(\mcGP_{(\mcX,\mcY)},\mcY)$ is a left Frobenius pair.

\item[(b)] $\mcY \subseteq \mcGP_{(\mcX,\mcY)}$.

\item[(c)] $(\mcGP_{(\mcX,\mcY)},\mcY^\wedge)$ is a complete cotorsion pair cut along $\mcGP_{(\mcX,\mcY)}^\wedge$ with $\mcY \subseteq \mcGP_{(\mcX,\mcY)}$. 
\end{enumerate}
Moreover, if any of the previous conditions is satisfied, then $\mcY = \omega$.
\end{proposition}

\begin{proof}
Note that the implications (a) $\Rightarrow$ (b) and (c) $\Rightarrow$ (b) are trivial, while (a) $\Rightarrow$ (c) follows by \cite[Theorem 3.6]{BMPS}. Then, we only focus on proving (b) $\Rightarrow$ (a). So let us assume that the containment $\mcY \subseteq \mcGP_{(\mcX,\mcY)}$ is satisfied. By \cite[Corollary 3.33]{BMS} we know that $\mcGP_{(\mcX,\mcY)}$ is left thick, that is, condition \lfpone \ holds. By hypothesis, we also have that $\omega$ is closed under direct summands. On the other hand, using \cite[Theorem 3.32 and Corollary 3.25 (b)]{BMS} and the containment $\mcY \subseteq \mcGP_{(\mcX,\mcY)}$, we have that $\omega = \mcY \cap \mcGP_{(\mcX,\mcY)} = \mcY$, and then \lfptwo \ follows. Finally, the fact that $\mcY = \omega$ is a $\mcGP_{(\mcX,\mcY)}$-injective relative cogenerator in $\mcGP_{(\mcX,\mcY)}$ follows by \cite[Corollary 3.25 (a)]{BMS}. 
\end{proof}

\begin{remark}
Note that, for any hereditary complete cotorsion pair $(\mcX,\mcY)$ in $\mcC$, with $\omega := \mcX \cap \mcY$, one has that $(\mcX,\omega)$ is a GP-admissible pair with $\omega \subseteq \mcGP_{(\mcX,\omega)}$. Then, by Proposition \ref{proposition:Gorenstein_Frobenius_pair_vs_Gorenstein_cotorsion_cut}, $(\mcGP_{(\mcX,\omega)},\omega^\wedge)$ is a complete cotorsion pair cut along $\mcGP_{(\mcX,\omega)}^\wedge$. In particular, we have another way to obtain the first pair appearing in Corollary \ref{coro:GPXY_cotorsion_pair}, since $\mcX^\wedge \subseteq \mcGP_{(\mcX,\omega)}^\wedge$. 
\end{remark}


\subsection*{Some technical lemmas}

Before giving the definition of cut Frobenius pairs, we need to prove some preliminary results. The idea is to find conditions under which $\omega^\wedge$ is closed under extensions. The Induction Principle will be a frequently used argument in the following lemmas.

\begin{lemma}\label{lem:technical_lemma_1}
Let $\omega$ and $\mcS$ be classes of objects in $\mcC$ such that $\omega$ is closed under extensions and $\omega \cap \mcS$ is a relative generator in $\omega$. Let $C \in \mcC$ be an object for which there exists an exact sequence
\begin{align}
0 & \to E_n \xrightarrow{f_n} W_{n-1} \to \cdots \to W_1 \xrightarrow{f_1} W_0 \xrightarrow{f_0} C \to 0 \label{eqn:resolution_omega}
\end{align}
for some $n \geq 1$, with $E_{j+1} := \Ker(f_j)$ and $W_j \in \omega$ for every $0 \leq j \leq n-1$. Then, there exist short exact sequences
\begin{align}
0 & \to G_j \to X_{j+1} \to E_{j+1} \to 0, \label{eqn:technical1} \\
0 & \to X_{j+1} \to F_j \to X_j \to 0, \label{eqn:technical2} 
\end{align}
where $X_0 := C$, $F_j \in \omega \cap \mcS$ and $G_j \in \omega$ for every $0 \leq j \leq n-1$. 
\end{lemma}

\begin{proof}
Let us prove this result by induction on $j$. 
\begin{itemize}
\item \underline{Initial step}: For the case $j = 0$, since $\omega \cap \mcS$ is a relative generator in $\omega$, there is a short exact sequence $0 \to G_0 \to F_0 \to W_0 \to 0$ with $G_0 \in \omega$ and $F_0 \in \omega \cap \mcS$. Taking the pullback of $E_1 \to W_0 \leftarrow F_0$ yields the following solid diagram:
\begin{equation}\label{diagram1}
\parbox{1.75in}{
\begin{tikzpicture}[description/.style={fill=white,inner sep=2pt}] 
\matrix (m) [ampersand replacement=\&, matrix of math nodes, row sep=2.5em, column sep=2.5em, text height=1.25ex, text depth=0.25ex] 
{ 
G_{0} \& G_{0} \& {} \\
X_{1} \& F_{0}\& C \\
E_{1} \& W_{0} \& C \\
};  
\path[->] 
(m-2-1)-- node[pos=0.5] {\footnotesize$\mbox{\bf pb}$} (m-3-2) 
; 
\path[>->]
(m-1-1) edge (m-2-1) (m-1-2) edge (m-2-2)
(m-2-1) edge (m-2-2) (m-1-1) edge (m-1-2)
(m-3-1) edge (m-3-2)
;
\path[->>]
(m-2-3) edge (m-3-3) (m-2-2) edge (m-2-3)
(m-2-2) edge (m-3-2) (m-3-2) edge (m-3-3)
(m-2-1) edge (m-3-1)
;
\path[-,font=\scriptsize]
(m-1-1) edge [double, thick, double distance=2pt] (m-1-2)
(m-2-3) edge [double, thick, double distance=2pt] (m-3-3)
;
\end{tikzpicture} 
}
\end{equation} 
The left-hand column and the central row are precisely the sequences \eqref{eqn:technical1} and \eqref{eqn:technical2}. 

\item \underline{Induction step}: Now suppose that for $1 \leq j \leq n-2$ there are short exact sequences $0 \to G_j \to X_{j+1} \to E_{j+1} \to 0$ and $0 \to X_{j+1} \to F_j \to X_j \to 0$, with $F_j \in \omega \cap \mcS$ and $G_j \in \omega$. Consider also the $(j+1)$-th splicer from the resolution \eqref{eqn:resolution_omega}, namely $0 \to E_{j+2} \to W_{j+1} \to E_{j+1} \to 0$. Taking the pullback of $W_{j+1} \to E_{j+1} \leftarrow X_{j+1}$ yields the following solid diagram:
\begin{equation}\label{diagram2}
\parbox{1.75in}{
\begin{tikzpicture}[description/.style={fill=white,inner sep=2pt}] 
\matrix (m) [ampersand replacement=\&, matrix of math nodes, row sep=2.5em, column sep=2.5em, text height=1.25ex, text depth=0.25ex] 
{ 
{} \& G_j \& G_j \\
E_{j+2} \& F'_{j+1}\& X_{j+1}\\
E_{j+2} \& W_{j+1} \& E_{j+1} \\
};  
\path[>->]
(m-1-3) edge (m-2-3) (m-1-2) edge (m-2-2)
(m-2-1) edge (m-2-2) 
(m-3-1) edge (m-3-2)
;
\path[->>]
(m-2-3) edge (m-3-3) (m-2-2) edge (m-2-3)
(m-2-2) edge (m-3-2) (m-3-2) edge (m-3-3)
;
\path[->] 
(m-2-2)-- node[pos=0.5] {\footnotesize$\mbox{\bf pb}$} (m-3-3)
; 
\path[-,font=\scriptsize]
(m-1-2) edge [double, thick, double distance=2pt] (m-1-3)
(m-2-1) edge [double, thick, double distance=2pt] (m-3-1)
;
\end{tikzpicture} 
}
\end{equation}
Since $\omega$ is closed under extensions, $F'_{j+1} \in \omega$. By using again that $\omega \cap \mcS$ is a generator in $\omega$, we have an exact sequence $0 \to G_{j+1} \to F_{j+1} \to F'_{j+1} \to 0$ with $F_{j+1} \in \omega \cap \mcS$ and $G_{j+1} \in \omega$. Now we take the pullback of the cospan $E_{j+2} \to F'_{j+1} \leftarrow F_{j+1}$ in order to obtain the following solid diagram:
\begin{equation}\label{diagram3}
\parbox{1.75in}{
\begin{tikzpicture}[description/.style={fill=white,inner sep=2pt}] 
\matrix (m) [ampersand replacement=\&, matrix of math nodes, row sep=2.5em, column sep=2.5em, text height=1.25ex, text depth=0.25ex] 
{ 
G_{j+1} \& G_{j+1} \& {} \\
X_{j+2} \& F_{j+1}\& X_{j+1} \\
E_{j+2} \& F'_{j+1}\& X_{j+1} \\
};  
\path[->] 
(m-2-1)-- node[pos=0.5] {\footnotesize$\mbox{\bf pb}$} (m-3-2)
; 
\path[>->]
(m-1-1) edge (m-2-1) (m-1-2) edge (m-2-2)
(m-2-1) edge (m-2-2) 
(m-3-1) edge (m-3-2)
;
\path[->>]
(m-2-3) edge (m-3-3) (m-2-2) edge (m-2-3)
(m-2-2) edge (m-3-2) (m-3-2) edge (m-3-3) (m-2-1) edge (m-3-1)
;
\path[-,font=\scriptsize]
(m-1-1) edge [double, thick, double distance=2pt] (m-1-2)
(m-2-3) edge [double, thick, double distance=2pt] (m-3-3)
;
\end{tikzpicture} 
}
\end{equation}
The left-hand column and the central row give the desired sequences \eqref{eqn:technical1} and \eqref{eqn:technical2} for $j+1$. 
\end{itemize}
\end{proof}

\begin{lemma}\label{lem:technical_lemma_2}
Let $\omega \subseteq \mcC$ be a class of objects in $\mcC$ closed under extensions. If 
\begin{align}
0 & \to W \to B \to C \to 0 \label{eqn:technical3}
\end{align}
is a short exact sequence with $W \in \omega$ and $C \in \omega^\wedge$, then $B \in \omega^\wedge$ and $\resdim_\omega(B) \leq \resdim_\omega(C)$. 
\end{lemma}

\begin{proof}
For $\resdim_{\omega}(C) = 0$, the result follows since $\omega$ is closed under extensions. So we may assume that $\resdim_{\omega}(C) \geq 1$. Then, there exists an exact sequence $0 \to W' \to W_0 \to C \to 0$ with $W_0 \in \omega$ and $\resdim_{\omega}(W') = \resdim_{\omega}(C) - 1$. Now let us take the pullback of $B \to C \leftarrow W_0$ in order to obtain the following solid diagram:
\begin{equation}\label{diagram4}
\parbox{1.75in}{
\begin{tikzpicture}[description/.style={fill=white,inner sep=2pt}] 
\matrix (m) [ampersand replacement=\&, matrix of math nodes, row sep=2.5em, column sep=2.5em, text height=1.25ex, text depth=0.25ex] 
{ 
{} \& W' \& W' \\
W \& E \& W_0 \\
W \& B \& C \\
};  
\path[>->]
(m-1-3) edge (m-2-3) (m-1-2) edge (m-2-2)
(m-2-1) edge (m-2-2) 
(m-3-1) edge (m-3-2)
;
\path[->>]
(m-2-3) edge (m-3-3) (m-2-2) edge (m-2-3)
(m-2-2) edge (m-3-2) (m-3-2) edge (m-3-3)
;
\path[->] 
(m-2-2)-- node[pos=0.5] {\footnotesize$\mbox{\bf pb}$} (m-3-3)
; 
\path[-,font=\scriptsize]
(m-1-2) edge [double, thick, double distance=2pt] (m-1-3)
(m-2-1) edge [double, thick, double distance=2pt] (m-3-1)
;
\end{tikzpicture} 
}
\end{equation}
Note that $E \in \omega$. Therefore, the central column of \eqref{diagram4} gives an $\omega$-resolution of $B$ with $\resdim_{\omega}(B) \leq \resdim_{\omega}(W') + 1 = \resdim_{\omega}(C)$. 
\end{proof}

We are now ready to prove the main technical lemma of this section, where we give sufficient conditions so that the class $\omega^\wedge$ is closed under extensions and direct summands. It is not enough to assume that $\omega$ is closed under extensions and direct summands. In addition, we need an auxiliary class $\mcS$ so that $\omega \cap \mcS$ is a $\omega$-projective relative generator in $\omega$. Such class $\mcS$ will play the role of a suitable cut to propose a relative version for the concept of left Frobenius pair in Definition \ref{def:relative_Frobenius_pair} below.

\begin{lemma}\label{lem:main_lemma}
Let $\omega$ and $\mcS$ be classes of objects of $\mcC$ such that $\omega$ is closed under extensions and $\omega \cap \mcS$ is an $\omega$-projective relative generator in $\omega$. Then, the following assertions hold true:
\begin{enumerate}
\item $\omega \cap \mcS$ is an $\omega^\wedge$-projective relative generator in $\omega^\wedge$. Moreover, for any $C \in \omega^\wedge$ with $\resdim_{\omega}(C) \geq 1$, there exists an exact sequence $0 \to K \to F \to C \to 0$ such that $F \in \omega \cap \mcS$ and $\resdim_{\omega}(K) = \resdim_{\omega}(C) - 1$. 

\item $\omega^\wedge$ is closed under extensions.

\item If $\omega$ is closed under direct summands, then so is $\omega^\wedge$. 

\item If $\omega$ is closed under isomorphisms and $\mcS$ is closed under epi-kernels and mono-cokernels, then $\omega^\wedge \cap \mcS = (\omega \cap \mcS)^\wedge$.
\end{enumerate}
\end{lemma}

\begin{proof} \
\begin{enumerate}
\item First, we show that $\omega \cap \mcS$ is a relative generator in $\omega^\wedge$. We use induction on $n := \resdim_{\omega}(C)$ for $C \in \omega^\wedge$. 
\begin{itemize}
\item \underline{Initial step}: This is clear. 

\item \underline{Induction step}: For the case $n \geq 1$, we have an exact sequence
\[
0 \to W_n \to W_{n-1} \to \cdots \to W_1 \to W_0 \to C \to 0
\]
with $W_k \in \omega$ for every $0 \leq k \leq n$. By Lemma \ref{lem:technical_lemma_1} and its notation therein, we have that $X_n \in \omega$ since $G_{n-1}, E_n := W_n \in \omega$ and $\omega$ is closed under extensions. Glueing together the splicer sequences \eqref{eqn:technical2} in Lemma \ref{lem:technical_lemma_1} gives rise to the exact sequence
\[
0 \to X_n \to F_{n-1} \to \cdots \to F_1 \to F_0 \to C \to 0
\]
with $F_k \in \omega \cap \mcS$ for every $0 \leq k \leq n-1$. Thus, for the short exact sequence $0 \to X_1 \to F_0 \to C \to 0$ we have $F_0 \in \omega \cap \mcS$ and $X_1 \in \omega^\wedge$, with $\resdim_{\omega}(X_1) = n - 1$; that is, $\omega \cap \mcS$ is a relative generator in $\omega^\wedge$. 
\end{itemize}

Now in order to show that $\omega \cap \mcS$ is $\omega^\wedge$-projective, that is $\pd_{\omega^\wedge}(\omega \cap \mcS) = 0$, we use \cite[dual of Lemma 2.13 (a) and Remark 1.2 (a)]{MS} as follows:
\[
\pd_{\omega^\wedge}(\omega \cap \mcS) = \id_{\omega \cap \mcS}(\omega^\wedge) = \id_{\omega \cap \mcS}(\omega) = \pd_{\omega}(\omega \cap \mcS) = 0.
\]

\item Suppose we are given a short exact sequence $0 \to A \to B \to C \to 0$ with $A, C \in \omega^\wedge$. Let us use induction on $n := \resdim_{\omega}(A)$ to show that $B \in \omega^\wedge$. 
\begin{itemize}
\item \underline{Initial step}: If $\resdim_{\omega}(A) = 0$, the result follows by Lemma \ref{lem:technical_lemma_2}.

\item \underline{Induction step}: We may assume that $\resdim_{\omega}(A) \geq 1$. Suppose also that for any short exact sequence $0 \to A' \to B' \to C' \to 0$ with $C' \in \omega^\wedge$ and $\resdim_{\omega}(A') \leq n-1$, one has that $B' \in \omega^\wedge$. By part (1) there is a short exact sequence $0 \to C' \to P \to C \to 0$ with $P \in \omega \cap \mcS$ and $C' \in \omega^\wedge$. We take the pullback of $B \to C \leftarrow P$ to obtain the following solid diagram:
\begin{equation}\label{diagram5}
\parbox{1.75in}{
\begin{tikzpicture}[description/.style={fill=white,inner sep=2pt}] 
\matrix (m) [ampersand replacement=\&, matrix of math nodes, row sep=2.5em, column sep=2.5em, text height=1.25ex, text depth=0.25ex] 
{ 
{} \& C' \& C' \\
A \& E \& P \\
A \& B \& C \\
};  
\path[>->]
(m-1-3) edge (m-2-3) (m-1-2) edge (m-2-2)
(m-2-1) edge (m-2-2) 
(m-3-1) edge (m-3-2)
;
\path[->>]
(m-2-3) edge (m-3-3) (m-2-2) edge (m-2-3)
(m-2-2) edge (m-3-2) (m-3-2) edge (m-3-3)
;
\path[->] 
(m-2-2)-- node[pos=0.5] {\footnotesize$\mbox{\bf pb}$} (m-3-3)
; 
\path[-,font=\scriptsize]
(m-1-2) edge [double, thick, double distance=2pt] (m-1-3)
(m-2-1) edge [double, thick, double distance=2pt] (m-3-1)
;
\end{tikzpicture} 
}
\end{equation}
Since $\pd_{\omega^\wedge}(\omega \cap \mcS) = 0$, the central row in \eqref{diagram5} splits, and then $E \simeq A \oplus P$. Now by part (1) again, consider a short exact sequence 
\begin{align}
0 & \to A' \to W_0 \to A \to 0 \label{eqn:res_for_A}
\end{align}
with $W_0 \in \omega \cap \mcS$ and $\resdim_{\omega}(A') = \resdim_{\omega}(A) - 1$. By taking the direct sum of the identity complex on $P$ and the sequence \eqref{eqn:res_for_A}, we get the short exact sequence
\[
0 \to A' \to W_0 \oplus P \to A \oplus P \to 0,
\] 
with $W_0 \oplus P \in \omega$ since $\omega$ is closed under extensions. Now take the pullback of $C' \to A \oplus P \leftarrow W_0 \oplus P$ to obtain the following solid diagram:
\begin{equation}\label{diagram6}
\parbox{1.75in}{
\begin{tikzpicture}[description/.style={fill=white,inner sep=2pt}] 
\matrix (m) [ampersand replacement=\&, matrix of math nodes, row sep=2.5em, column sep=2.5em, text height=1.25ex, text depth=0.25ex] 
{ 
A' \& A' \& {} \\
B' \& W_{0} \oplus P \& B \\
C' \& A \oplus P \& B \\
};  
\path[>->]
(m-1-1) edge (m-2-1) (m-1-2) edge (m-2-2)
(m-2-1) edge (m-2-2) 
(m-3-1) edge (m-3-2)
;
\path[->>]
(m-2-3) edge (m-3-3) (m-2-2) edge (m-2-3)
(m-2-2) edge (m-3-2) (m-3-2) edge (m-3-3) (m-2-1) edge (m-3-1)
;
\path[->] 
(m-2-1)-- node[pos=0.5] {\footnotesize$\mbox{\bf pb}$} (m-3-2)
;
\path[-,font=\scriptsize]
(m-1-1) edge [double, thick, double distance=2pt] (m-1-2)
(m-2-3) edge [double, thick, double distance=2pt] (m-3-3)
;
\end{tikzpicture} 
} 
\end{equation}
By the induction hypothesis, we have that $B' \in \omega^\wedge$. Therefore, the central row of \eqref{diagram6} is precisely a $\omega$-resolution of $B$, and so $B \in \omega^\wedge$.
\end{itemize}

\item Given an object $C = C_1 \oplus C_2 \in \omega^\wedge$ with $n := \resdim_{\omega}(C)$, the proof that $C_1, C_2 \in \omega^\wedge$ follows by induction on $n$ and using an argument similar to the one appearing in \cite[Proof of Proposition 5.3]{BGP}.

\item The containment $\omega^\wedge \cap \mcS \supseteq (\omega \cap \mcS)^\wedge$ is clear since $\mcS$ is closed under mono-cokernels. Now suppose we are given an object $S \in \omega^\wedge \cap \mcS$, that is, an $S \in \mcS$ with a finite $\omega$-resolution
\[
0 \to W_n \to W_{n-1} \to \cdots \to W_1 \to W_0 \to S \to 0.
\]
If $n = 0$, the result follows since $\omega$ is closed under isomorphisms. So we may assume that $n \geq 1$. Since $\mcS$ is closed under epi-kernels,  from Lemma \ref{lem:technical_lemma_1} there are short exact sequences $0 \to X_{j+1} \to F_j \to X_j \to 0$ with $X_j \in \mcS$ and $F_j \in \omega \cap \mcS$ for all $0 \leq j \leq n-1$. Moreover, $X_n \in \omega \cap \mcS$ since $G_{n-1}, E_n := W_n \in \omega$ and $\omega$ is closed under extensions. Then, 
\[
0 \to X_n \to F_{n-1} \to \cdots \to F_1 \to F_0 \to S \to 0
\]
is a $(\omega \cap \mcS)$-resolution of $S$. 
\end{enumerate}
\end{proof}

\subsection*{Cut Frobenius pairs}

We are now ready to present the concept of Frobenius pairs cut along subcategories. It is important that the reader keeps in mind the concept of Frobenius pairs from the preliminaries (Definition \ref{def:Frobenius_pair}).

\begin{definition}\label{def:relative_Frobenius_pair}
Let $\mcX$, $\omega$ and $\mcS$ be classes of objects in $\mcC$. We say that $(\mcX,\omega)$ is a \textbf{left Frobenius pair cut along $\bm{\mcS}$} if the following conditions are satisfied: 
\begin{enumerate}
\item[\lcfpone] $\mcX$ is left thick. 

\item[\lcfptwo] $(\mcX \cap \mcS,\omega \cap \mcS)$ is a left Frobenius pair in $\mcC$.

\item[\lcfpthree] $\omega \cap \mcS$ is an $\omega$-projective relative generator in $\omega$.

\item[\lcfpfour] $\omega$ is closed under extensions and direct summands. 
\end{enumerate}
Dually, we have the notion of \textbf{right Frobenius pairs $\bm{(\nu,\mcY)}$ cut along $\bm{\mcS}$}. 
\end{definition}

\begin{remark}\label{rem:properties_of_cut_Frobenius_pairs}
For any left Frobenius pair $(\mcX,\omega)$ cut along $\mcS$, we can note the following by Lemma \ref{lem:main_lemma} and \cite[Proposition 2.7 (2)]{BMPS}:
\begin{enumerate}
\item $\omega \cap \mcS$ is closed under extensions and finite coproducts, and it is an $\omega^\wedge$-projective relative generator in $\omega^\wedge$.

\item $\omega^\wedge$ is closed under extensions and direct summands. 

\item $\omega$ is closed under isomorphisms, and so $\omega^\wedge \cap \mcS = (\omega \cap \mcS)^\wedge$ whenever $\mcS$ is closed under epi-kernels and mono-cokernels. 
\end{enumerate} 
\end{remark}

Of course any left Frobenius pair $(\mcX,\omega)$, with $\omega$ closed under finite coproduts, is a left Frobenius pair cut along $\mcC$, but not every relative left Frobenius pair is an absolute left Frobenius pair, as we show in Example \ref{ex:relative_Frobenius_but_not_absolute} below.

\begin{proposition}\label{prop:induced_cut_cotorsion_pair}
Let $(\mcX,\mcY)$ be a GP-admissible pair in $\mcC$ such that: 
\begin{enumerate}
\item $\mcX$ is closed under direct summands;

\item $\mcY$ is closed under extensions and direct summands; and 

\item $\mcX \cap \mcY$ is a relative generator in $\mcY$. 
\end{enumerate}
Then, $(\mcC,\mcY^\wedge)$ is a left Frobenius pair cut along $\mcGP_{(\mcX,\mcY)}$. Moreover, the following equality holds true: 
\begin{align}
\mcGP_{(\mcX,\mcY)} \cap \mcY^\wedge & = \mcX \cap \mcY = \mcX \cap \mcY^\wedge. \label{eqn:relative_not_absolute_Frobenius}
\end{align}
\end{proposition}

\begin{proof} 
Let us verify each condition in Definition \ref{def:relative_Frobenius_pair}, although we need to do this in a specific order. 

Condition \lcfpone \ is clearly satisfied. On the other hand, since $\mcY$ is closed under extensions and direct summands (by (2)), and $\mcX \cap \mcY$ is an $\mcY$-projective relative generator in $\mcY$ (by (3) and \GPaone \ in Definition \ref{def:GP-admissible}), we have from Lemma \ref{lem:main_lemma} (2) and (3) that $\mcY^\wedge$ is closed under extensions and direct summands, that is, condition \lcfpfour \ holds. Moreover, since also $\mcX \cap \mcY$ and $\mcX \cap \mcY^\wedge$ are closed under direct summands, we obtain the equality \eqref{eqn:relative_not_absolute_Frobenius} from \cite[Theorem 3.34 (d)]{BMS}. This, along with Lemma \ref{lem:main_lemma} (1), implies that $\mcY^\wedge \cap \mcGP_{(\mcX,\mcY)}$ is a $\mcY^\wedge$-projective relative generator in $\mcY^\wedge$, that is, we have condition \lcfpthree. Finally, condition \lcfptwo, that is, that $(\mcGP_{(\mcX,\mcY)}, \mcY^\wedge \cap \mcGP_{(\mcX,\mcY)}) = (\mcGP_{(\mcX,\mcY)},\mcX \cap \mcY)$ is a left Frobenius pair in $\mcC$, follows by \cite[Corollaries 3.25 (a) and 3.33]{BMS}.
\end{proof}

\begin{example}\label{ex:relative_Frobenius_but_not_absolute}
For any ring $R$, if $\mcY$ is a class in $\Mod(R)$ closed under finite coproducts and containing $\mcP(R)$, then $(\Mod(R),\mcY^\wedge)$ is a left Frobenius pair cut along $\mcGP_{(\mcP(R),\mcY)}$, since $(\mcP(R),\mcY)$ is a GP-admissible pair satisfying the conditions of the previous proposition. In particular, we have that $(\Mod(R),\mcP(R)^\wedge)$ and $(\Mod(R),\mcF(R)^\wedge)$ are left Frobenius pairs cut along $\mcGP(R)$ and $\mcDP(R)$, respectively. 

In a more general sense, consider a hereditary complete cotorsion pair $(\mcX,\mcY)$ in $\mcC$. Then, $(\mcC,\mcY^\wedge)$ is a left Frobenius pair cut along $\mcX$. Note in this case that $\mcX = \mcGP_{(\mcX,\mcY)}$ by \cite[Corollary 3.17]{BMS}. On the other hand, for the class $\omega := \mcX \cap \mcY$, we have that $(\mcC,\omega^\wedge)$ is a left Frobenius pair cut along $\mcGP_{(\mcX,\omega)}$.
\end{example}

Below we establish necessary and sufficient conditions under which the relative Frobenius pair from the previous example is left Frobenius in $\Mod(R)$. This will be a consequence of the following general result.

\begin{proposition}
Let $(\mcX,\omega)$ be a pair of classes of objects in $\mcC$ such that $\omega^\wedge \subseteq \mcX$, $\mcX$ is left thick, and $\omega$ is closed under direct summands and a relative cogenerator in $\mcX$. Then, the following statements are equivalent:
\begin{enumerate}
\item[(a)] $(\mcX,\omega^\wedge)$ is a left Frobenius pair in $\mcC$.

\item[(b)] $\omega = \mcX^\perp \cap \mcX$.
\end{enumerate}
Moreover, if any of the above equivalent conditions holds, then $\omega = \omega^\wedge$. 
\end{proposition}

\begin{proof}
The containment $\mcX^\perp \cap \mcX \subseteq \omega$ always holds since $\omega$ is closed under direct summands and a relative cogenerator in $\mcX$. 

Suppose first that $(\mcX,\omega^\wedge)$ is a left Frobenius pair in $\mcC$. Then, $0 = \id_{\mcX}(\omega^\wedge) = \id_{\mcX}(\omega)$ and so the other containment $\omega \subseteq \mcX^\perp \cap \mcX$ holds as well. 

Now suppose that $\omega = \mcX^\perp \cap \mcX$. This implies that $(\mcX,\omega)$ is a left Frobenius pair. By \cite[Proposition 2.13]{BMPS}, we have that $\omega^\wedge = \mcX^\perp \cap \mcX$, and so $\omega^\wedge$ is closed under direct summands. Finally, it is clear that $\omega^\wedge \subseteq \mcX$ is an $\mcX$-injective relative cogenerator in $\mcX$. Hence, $(\mcX,\omega^\wedge)$ is a left Frobenius pair in $\mcC$. 
\end{proof}

A more particular version of the previous result is the following.

\begin{corollary}
The following assertions are equivalent for any class $\mcY$ of $R$-modules:
\begin{enumerate}
\item[(a)] $(\Mod(R),\mcY^\wedge)$ is a left Frobenius pair in $\Mod(R)$.

\item[(b)] $\mcY^\wedge = \mcI(R)$. 
\end{enumerate}
In case any of the above conditions holds true and $\mcP(R) \subseteq \mcY$, then $R$ is a quasi-Frobenius ring.
\end{corollary}

\begin{proof}
The implication (b) $\Rightarrow$ (a) is trivial. Now for (a) $\Rightarrow$ (b), since $\id(\mcY^\wedge) = 0$ we have that $\mcY^\wedge \subseteq \mcI(R)$. Now if $E$ is an injective $R$-module, since $\mcY^\wedge$ is a relative cogenerator in $\Mod(R)$ we can find a split exact sequence $0 \to E \to W \to M \to 0$ where $W \in \mcY^\wedge$, and so $E \in \mcY^\wedge$. Hence, $\mcY^\wedge = \mcI(R)$. 
\end{proof}

From the previous two results, we can note that the left Frobenius pairs $(\mcX,\omega)$ that induce a left Frobenius pair of the form $(\mcX,\omega^\wedge)$ are scarce. As a matter of fact, what we expect from a left Frobenius pair $(\mcX,\omega)$ is that $(\mcX,\omega^\wedge)$ is a complete cut cotorsion pair. This is precisely the case of Gorenstein left Frobenius pairs $(\mcGP_{(\mcX,\mcY)},\mcY)$ in Proposition \ref{proposition:Gorenstein_Frobenius_pair_vs_Gorenstein_cotorsion_cut}. We shall explore this in more detail in Section \ref{sec:versus}. 

Let us now present one more example of cut Frobenius pairs in the context of quiver representations.

\begin{example}
Several facts in this example are extracted from \cite[Example 5.3]{ZX}. Let $\Lambda$ be the quotient path $k$-algebra given by the quiver
\[
\xymatrix{
1 \ar@/^/[r]^{\alpha} & 2\ar@/^/[l]^{\beta} & 3 \ar[l]^{\gamma}
}
\]
with relations $\alpha \beta = 0 = \beta \alpha$. In the category $\fgMod(\Lambda)$ of finitely generated (left) $\Lambda$-modules, the indecomposable projective $\Lambda$-modules are:
\begin{align*}
{\scriptsize \begin{array}{c} 1 \\ 2 \end{array}}, & & {\scriptsize \begin{array}{c} 2 \\ 1 \end{array}} & & \text{and} & & {\scriptsize \begin{array}{l} 3 \\ 2 \\ 1. \end{array}}
\end{align*}
On the other hand, the indecomposable injective $\Lambda$-modules are:
\begin{align*}
{\scriptsize \begin{array}{c} 3 \\ 2 \\ 1 \end{array}}, & & {\scriptsize \begin{array}{ccc} 1 & & 3 \\ & 2 \end{array}} & & \text{and} & & {\scriptsize \begin{array}{c} 3. \end{array}}
\end{align*}
It follows that the Auslander-Reiten quiver of $\fgMod(\Lambda)$ is given by:
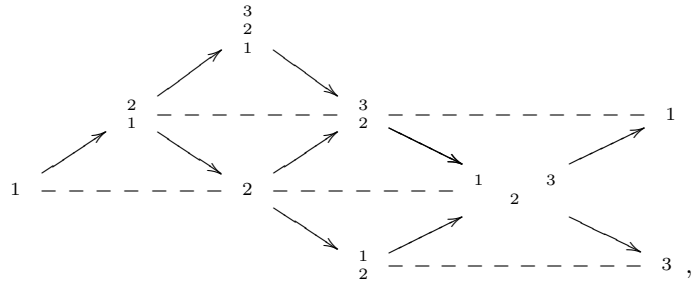
\[
\xymatrix@R=0.3cm{
 & & {\begin{tiny}
\begin{array}{c}
3 \\
2\\
1
\end{array}
\end{tiny}} \ar[rd] & & & \\
 & {\begin{tiny}
\begin{array}{c}
2\\
1
\end{array}
\end{tiny}} \ar[ru] \ar[rd] \ar@{--}[rr] & & {\begin{tiny}
\begin{array}{c}
3 \\
2
\end{array}
\end{tiny}} \ar[rd] \ar[rd] \ar@{--}[rr] & & {\scriptsize \begin{array}{c} 1 \end{array}} \\
{\scriptsize \begin{array}{c} 1 \end{array}} \ar[ru] \ar@{--}[rr] & & {\scriptsize \begin{array}{c} 2 \end{array}} \ar[ru] \ar[rd] \ar@{--}[rr] & & {\begin{tiny}
\begin{array}{ccc}
1& &3\\
&2&
\end{array}
\end{tiny}} \ar[ru] \ar[rd] & \\
 & & & {\begin{tiny}
\begin{array}{c}
1\\
2
\end{array}
\end{tiny}} \ar[ru] \ar@{--}[rr] & & {\scriptsize \begin{array}{c} 3 \end{array}},
}
\]
where the two vertices $1$ represent the same simple module.

Now let $\mcX := {\rm add} ( \begin{tiny} \begin{array}{c} 1 \end{array} \end{tiny} \oplus \begin{tiny} \begin{array}{c} 2 \\ 1 \end{array}\end{tiny} \oplus \begin{tiny} \begin{array}{c} 2 \end{array} \end{tiny} \oplus \begin{tiny}\begin{array}{c} 1 \\ 2 \end{array}\end{tiny} )$ be the class of direct summands of finite coproducts of copies of the $\Lambda$-module $\begin{tiny} \begin{array}{c} 1 \end{array} \end{tiny} \oplus \begin{tiny} \begin{array}{c} 2 \\ 1 \end{array}\end{tiny} \oplus \begin{tiny} \begin{array}{c} 2 \end{array} \end{tiny} \oplus \begin{tiny}\begin{array}{c} 1 \\ 2 \end{array}\end{tiny}$. Then, $\mcX$ is closed under extensions and a Frobenius subcategory of $\fgMod(A)$. Moreover, the class of the projective-injective objects in $\mcX$ is $\mcP(\mcX) = {\rm add}(\begin{tiny}\begin{array}{c} 1 \\ 2 \end{array}\end{tiny} \oplus \begin{tiny}\begin{array}{c} 2 \\ 1 \end{array} \end{tiny})$. Indeed, by using Auslander-Reiten theory, it can be shown that $\Ext^i_\Lambda(\mcX,\mcP(\mcX)) = 0$ for every $i \geq 1$. 

We assert that $(\fgMod(\Lambda), \mcP(\mcX))$ is a left Frobenius pair cut along $\mcX$. Note first that $\mcX$ is not a resolving class in $\fgMod(\Lambda)$ since it does not contain all indecomposable projective $\Lambda$-modules. On the one hand, it is easy to see that conditions \lcfpone, \lcfpthree \ and \lcfpfour \ in Definition \ref{def:relative_Frobenius_pair} hold true as $\mcP(\mcX) \subseteq \mcP(\Lambda)$, and $\mcX$ is closed under extensions and direct summands. So, we only need to prove that $(\mcX,\mcP(\mcX))$ is a left Frobenius pair in $\fgMod(\Lambda)$. First, we show that $\mcX$ is left thick. Since $\mcX$ is closed under extensions and direct summands, it is only left to verify that it is closed under epi-kernels. So suppose we are given a short exact sequence $0 \to X \to Y \to Z \to 0$ in $\fgMod(\Lambda)$ with $Y, Z \in \mcX$. Using that $\mcX$ has enough projective objects, we construct the following solid diagram by taking the pullback of $Y \to Z \leftarrow P$:
\begin{equation}\label{1ec}
\parbox{1.5in}{
\begin{tikzpicture}[description/.style={fill=white,inner sep=2pt}] 
\matrix (m) [ampersand replacement=\&, matrix of math nodes, row sep=2.5em, column sep=2.5em, text height=1.25ex, text depth=0.25ex] 
{ 
{} \& Z' \& Z'\\
X \& L \& P\\
X \& Y \& Z\\
}; 
\path[->] 
(m-2-2)-- node[pos=0.5] {\footnotesize$\mbox{\bf pb}$} (m-3-3)
; 
\path[-,font=\scriptsize]
(m-1-2) edge [double, thick, double distance=2pt] (m-1-3)
(m-2-1) edge [double, thick, double distance=2pt] (m-3-1)
;
\path[>->]
(m-1-2) edge (m-2-2)
(m-1-3) edge (m-2-3)
(m-2-1) edge (m-2-2)
(m-3-1) edge (m-3-2)
;
\path[->>]
(m-2-2) edge (m-2-3)
(m-2-2) edge (m-3-2)
(m-3-2) edge (m-3-3)
(m-2-3) edge (m-3-3)
;
\end{tikzpicture}
}
\end{equation}
with $P\in \mcP(\mcX)$ and $Z' \in \mcX$. Note that $L \in \mcX$ since $\mcX$ is closed under extensions, and $Y$ and $Z'$ belong to $\mcX$. Moreover, since $\mcP(\mcX) \subseteq \mcP(\Lambda)$ the central row in (\ref{1ec}) splits, and then $L \simeq X \oplus P$. Thus, $X \in \mcX$ as $\mcX$ is closed under direct summands. The remaining conditions for being a left Frobenius pair follow from the fact that $\mcX$ is a Frobenius subcategory of $\fgMod(\Lambda)$. 
\end{example}


\subsection*{Cut Auslander-Buchweitz contexts}

In this part we introduce the concept of cut Auslander-Buchweitz contexts and present some examples related to hereditary complete cotorsion pairs. Let us recall the following definition from \cite[Definition 5.1]{BMPS}.

\begin{definition}\label{def:AB_context}
A pair $(\mcA,\mcB)$ of classes of objects in $\mcC$ with $\omega := \mcA \cap \mcB$ is a \textbf{left weak Auslander-Buchweitz context} (or a \textbf{left weak AB context}, for short) if the following three conditions are satisfied: 
\begin{enumerate}
\item[\lABone] $(\mcA,\omega)$ is a left Frobenius pair in $\mcC$.

\item[\lABtwo] $\mcB$ is a right thick class. 

\item[\lABthree] $\mcB \subseteq \mcA^\wedge$.
\end{enumerate}
The notion of \textbf{right weak AB context} is dual. 
\end{definition}

Let us now propose the following generalization of the previous definition.

\begin{definition}\label{def:cut_AB_context}
Let $\mcA$, $\mcB$ and $\mcS$ be classes of objects in $\mcC$, and let $\omega := \mcA \cap \mcB$. We say that $(\mcA,\mcB)$ is a \textbf{left weak AB context cut along $\bm{\mcS}$} if the following three conditions are satisfied:
\begin{enumerate}
\item[\lcABone] $(\mcA,\omega)$ is a left Frobenius pair cut along $\mcS$.

\item[\lcABtwo] $\mcB \cap \mcS$ is a right thick class. 

\item[\lcABthree] $\mcB \cap \mcS \subseteq (\mcA \cap \mcS)^\wedge$.
\end{enumerate} 
Dually, we have the concept of \textbf{right weak AB context cut along $\bm{\mcS}$}.
\end{definition}

\begin{remark}\label{rem:vs_first_approach}
We can have a first approach to the relation between cut AB contexts and relative cotorsion pairs, which will be analyzed in more detail in Section \ref{sec:versus}. For any left weak AB context $(\mcA,\mcB)$ cut along $\mcS$, one has that $(\mcA \cap \mcS,\mcB \cap \mcS)$ is a left weak AB context in $\mcC$. Then, $(\mcA \cap \mcS, \mcB \cap \mcS)$ is a relative $\mathsf{Thick}(\mcA \cap \mcS)$-cotorsion pair with $\id_{\mcA \cap \mcS}(\mcB \cap \mcS) = 0$ and $(\mcA \cap \mcB \cap \mcS)^\wedge = \mcB \cap \mcS$. See \cite[Definition 3.4 and Proposition 5.5]{BMPS} for details. 
\end{remark}

\begin{example}\label{ex:From_GP_to_AB}
Let $(\mcX,\mcY)$ be a GP-admissible pair in $\mcC$ with $\omega := \mcX \cap \mcY$, such that $\mcX$ is closed under epi-kernels and direct summands, and such that $\mcY$ is right thick. 
\begin{enumerate}
\item $(\mcX,\mcY)$ is a left weak AB context cut along $\mcX^\wedge$: It is easy to check that $(\mcX,\mcY)$ satisfies conditions \lcABone \ and \lcABthree \ in Definition \ref{def:cut_AB_context}. Moreover, the class $\mcY \cap \mcX^\wedge$ is right thick since $\mcX^\wedge$ is thick by \cite[Theorem 2.11]{BMPS} and $\mcY$ is right thick by assumption. Thus, $(\mcX,\mcY)$ satisfies also \lcABtwo.

\item $(\mcGP_{(\mcX,\mcY)},\mcY)$ is a left weak AB context cut along $\mcX^\wedge$: Let us first see that the pair $(\mcGP_{(\mcX,\mcY)},\mcGP_{(\mcX,\mcY)} \cap \mcY)$ is left Frobenius cut along $\mcX^\wedge$. By \cite[Corollary 3.33]{BMS} we have that $\mcGP_{(\mcX,\mcY)}$ is left thick. Moreover, $\mcGP_{(\mcX,\mcY)} \cap \mcY = \omega$ (by \cite[Corollary 3.25 (b)]{BMS}) and $\mcGP_{(\mcX,\mcY)} \cap \mcX^\wedge = \mcX$. Indeed, the first equality  and the containment $\mcGP_{(\mcX,\mcY)} \cap \mcX^\wedge \supseteq \mcX$  are clear. Now let $C \in \mcGP_{(\mcX,\mcY)} \cap \mcX^\wedge$. By \cite[Theorem 2.8]{BMPS} we can find a short exact sequence $0 \to K \to X \to C \to 0$ with $X \in \mcX$ and $K \in \omega^\wedge$. This sequence splits since $C \in \mcGP_{(\mcX,\mcY)}$, and so $C \in \mcX$. We then have that $(\mcX, \omega) = (\mcGP_{(\mcX,\mcY)} \cap \mcX^\wedge, (\mcGP_{(\mcX,\mcY)} \cap \mcY) \cap \mcX^\wedge)$ is a left Frobenius pair. From the previous equalities we can also note that $(\mcGP_{(\mcX,\mcY)} \cap \mcY) \cap \mcX^\wedge$ is a $(\mcGP_{(\mcX,\mcY)} \cap \mcY)$-projective relative generator in $\mcGP_{(\mcX,\mcY)} \cap \mcY$, and that $\mcGP_{(\mcX,\mcY)} \cap \mcY$ is closed under extensions and direct summands. Hence, we have that the pair $(\mcGP_{(\mcX,\mcY)}, \mcGP_{(\mcX,\mcY)} \cap \mcY)$ satisfies \lcABone. Finally, conditions \lcABtwo \  and \ \lcABthree \ are clear. 

\item $(\mcGP_{(\mcX,\omega)},\mcY)$ is a left weak AB context cut along $\mcX^\wedge$: We first check that the pair  $(\mcGP_{(\mcX,\omega)},\mcGP_{(\mcX,\omega)} \cap \mcY)$ is left Frobenius cut along $\mcX^\wedge$. Note that $\pd_{\omega}(\mcX) = 0$, $\omega$ is closed under finite coproducts, and $\omega$ is a relative cogenerator in $\mcX$. Then, it follows by \cite[Theorems 3.30 and 3.32]{BMS} that $\mcGP_{(\mcX,\omega)}$ is left thick. Using again \cite[Theorem 2.8]{BMPS}, we have that $\mcGP_{(\mcX,\omega)} \cap \mcX^\wedge = \mcX$ and $(\mcGP_{(\mcX,\omega)} \cap \mcY) \cap \mcX^\wedge = \omega$. Thus, $(\mcGP_{(\mcX,\omega)} \cap \mcX^\wedge, (\mcGP_{(\mcX,\omega)} \cap \mcY) \cap \mcX^\wedge) = (\mcX,\omega)$ is a left Frobenius pair in $\mcC$. The rest of the proof follows easily. 
\end{enumerate}
\end{example}

\begin{remark}
Note that hereditary complete cotorsion pairs provide with a wide range of GP-admissible pairs $(\mcX,\mcY)$ satisfying the assumptions in the previous example. Let us now exhibit a GP-admissible pair $(\mcX,\mcY)$, with $\mcX$ closed under epi-kernels and direct summands, and with $\mcY$ right thick, such that $(\mcX,\mcY)$ is not a hereditary complete cotorsion pair. This is the case of the classes $\mcX = \mcGP_{\text{AC}}(R)$ of Gorenstein AC-projective $R$-modules, and $\mcY = \mcL(R)^\wedge$ of $R$-modules of finite level dimension, provided that $R$ is not an AC-Gorenstein ring (see Gillespie's \cite[Theorem 6.2]{GillespieAC}). Indeed, by \cite[Lemma 8.6 and Proposition 8.10]{BGH} we know that $(\mcGP_{\text{AC}}(R),(\mcGP_{\text{AC}}(R))^{\perp_1})$ is a hereditary complete cotorsion pair in $\Mod(R)$. Moreover, it is clear that $\pd_{\mcL(R)^\wedge}(\mcGP_{\text{AC}}(R)) = 0$ and that $\mcL(R)^\wedge$ is closed under finite coproducts. On the other hand, it is straightforward to check that $\mcGP_{\text{AC}}(R) \cap \mcL(R)^\wedge = \mcP(R)$. Thus, in particular, we have from the previous example that $(\mcGP_{\text{AC}}(R),\mcL(R)^\wedge)$ is a left weak AB-context cut along $\mcGP_{\text{AC}}(R)^\wedge$. 
\end{remark}

In Example \ref{ex:From_GP_to_AB} we obtained the left weak AB-context $(\mcGP_{(\mcX,\omega)}, \mcY)$ cut along $\mcX^\wedge$, from a GP-admissible pair $(\mcX,\mcY)$ with $\mcX$ left thick and $\mcY$ right thick. With a couple of extra assumptions on $\mcX$ and $\mcY$, we are able to characterize $(\mcGP_{(\mcX,\omega)}, \mcY)$ as a left weak AB-context that is \emph{absolute} (that is, cut along the whole category $\mcC$). Let us specify this in the following result.

\begin{theorem}\label{theo:GP_admissible_weak_AB}
Let $(\mcX,\mcY)$ be a GP-admissible pair in $\mcC$ such that $\omega := \mcX \cap \mcY$ is closed under direct summands, ${}^{\perp_1}\mcY \subseteq \mcX$ and $\mcY$ is right thick. Then, the following conditions are equivalent:
\begin{enumerate}
\item[(a)] $(\mcGP_{(\mcX,\omega)},\mcY)$ is a left weak AB-context in $\mcC$. 

\item[(b)] $\mcY \subseteq \mcX^\wedge$ and $\mcX$ is left thick. 

\item[(c)] $(\mcX,\mcY)$ is a left weak AB-context in $\mcC$. 
\end{enumerate}
Moreover, if any of  the previous conditions holds, then $\mcX = \mcGP_{(\mcX,\omega)}$. 
\end{theorem}

\begin{proof}
Let us show first that (a) $\Rightarrow$ (b). Since $(\mcGP_{(\mcX,\omega)},\mcY)$ is a left weak AB-context in $\mcC$, we have by \cite[Proposition 5.5]{BMPS} that $\id_{\mcGP_{(\mcX,\omega)}}(\mcY) = 0$. This fact, along with the assumption (a), implies that $\mcX \subseteq \mcGP_{(\mcX,\omega)} \subseteq {}^{\perp_1}\mcY \subseteq \mcX$, and thus $\mcX = \mcGP_{(\mcX,\omega)}$, which is left thick by \cite[Corollary 3.33]{BMS} since $(\mcX,\omega)$ is a GP-admissible pair by \cite[Theorem 3.34 (b)]{BMS}. Then, $\mcY \subseteq \mcGP^\wedge_{(\mcX,\omega)} = \mcX^\wedge$ by \lABthree. 

The implication (b) $\Rightarrow$ (c) is clear. Finally, let us show (c) $\Rightarrow$ (a). So suppose that $(\mcX,\mcY)$ is a left weak AB-context. Thus, we have that $\mcY$ is right thick and that $\mcY \subseteq \mcX^\wedge \subseteq \mcGP^\wedge_{(\mcX,\omega)}$. It is only left to show that $(\mcGP_{(\mcX,\omega)}, \mcGP_{(\mcX,\omega)} \cap \mcY)$ is a left Frobenius pair in $\mcC$. We know that $\mcGP_{(\mcX,\omega)}$ is left thick by \cite[Corollary 3.33]{BMS}. For conditions \lfptwo \ and \lfpthree, it suffices to show that $\mcGP_{(\mcX,\omega)} \cap \mcY = \omega$. Indeed, using \cite[Theorem 2.8]{BMPS}, $\Ext^1_{\mcC}(\mcGP_{(\mcX,\omega)},\omega^\wedge) = 0$ (see \cite[Corollary 3.25 (a)]{BMS}) and the fact that $\mcX$ is closed under direct summands, we can note that $\mcGP_{(\mcX,\omega)} \cap \mcX^\wedge = \mcX$. This equality along with the assumption $\mcY \subseteq \mcX^\wedge$ yields $\omega \subseteq \mcGP_{(\mcX,\omega)} \cap \mcY \subseteq \mcX \cap \mcY = \omega$.
\end{proof}

The previous theorem basically asserts that $(\mcX,\mcX \cap \mcY)$-Gorenstein projective objects in the sense of Xu \cite{Xu} are trivial in the case they are part of a left weak AB-context. In other words, given a hereditary complete cotorsion pair $(\mcX,\mcY)$, it is not useful to apply the theory of \emph{absolute} AB-contexts \cite{BMPS} to the objects in $\mcGP_{(\mcX,\mcX \cap \mcY)}$, in the sense that any result obtained this way is simply a property for the classes $\mcX$, $\mcY$ and $\mcX \cap \mcY$. This leads to the need of a relativization for the notion of AB context to subcategories of $\mcC$. One interesting aspect about the relativization proposed in Definition \ref{def:cut_AB_context} is that the correspondence proved in \cite{BMPS} for the absolute case, between AB contexts, Frobenius pairs and relative cotorsion pairs, is still going to be valid. We shall prove this assertion in a series of results which are part of Section \ref{sec:versus}. For the moment, we give a small preamble to this by constructing below in Proposition \ref{prop:GP_cotorsion_cuts} three examples of complete cut cotorsion pairs involving the classes $\mcX$, $\mcGP_{(\mcX,\mcY)}$ and $\mcGP_{(\mcX,\omega)}$, which were considered in Example \ref{ex:From_GP_to_AB} in the construction of relative weak AB-contexts.

In \cite[Theorem 3.6]{BMPS}, the authors proved that if $(\mcX,\omega)$ is a left Frobenius pair, then 
$(\mcX,\omega^\wedge)$ is an  $\mcX^\wedge$-cotorsion pair. If in addition $\omega$ is an $\mcX$-projective relative generator in $\mcX$, then $(\omega,\mcX^\wedge)$ is also an $\mcX^\wedge$-cotorsion pair (see \cite[Theorem 3.7]{BMPS}). The following establishes similar results in the setting of complete cut cotorsion pairs.

\begin{proposition}\label{prop:cotorsion_cut_AB_theory}
Let $\mcX$ and $\omega$ be classes of objects in $\mcC$. The following holds:
\begin{enumerate}
\item If $\omega$ is closed under extensions and direct summands, and $\id_{\omega}(\omega) = 0$, then $\omega^\wedge$ is also closed under extensions and direct summands. 
\end{enumerate}
If $\mcX$ is closed under extensions and direct summands, then the following also hold:
\begin{enumerate}\setcounter{enumi}{1}
\item If $\omega$ is an $\mcX$-injective relative cogenerator in $\mcX$, then $\mcX^\wedge \in {\rm lCuts}(\mcX,\omega^\wedge)$. If in addition, $\omega$ is closed under extensions and direct summands, then $\mcX^\wedge \in {\rm rCuts}(\mcX,\omega^\wedge)$ (and so $\mcX^\wedge \in {\rm Cuts}(\mcX,\omega^\wedge)$).

\item If $\omega$ is an $\mcX$-projective relative generator in $\mcX$, then $\mcX^\wedge$ is closed under extensions and direct summands. If in addition $\omega$ is closed under direct summands, then $\mcX^\wedge \in {\rm Cuts}(\omega,\mcX^\wedge)$. 
\end{enumerate}
\end{proposition}

\begin{proof} \
\begin{enumerate}
\item Follows by taking $\mcS := \mcC$ in Lemma \ref{lem:main_lemma} (2) and (3). 

\item From \cite[Theorem 2.8]{BMPS} we know that for any $C \in \mcX^\wedge$ there are exact sequences $\eta \colon 0 \to Y \to X \to C \to 0$ and $\epsilon \colon 0 \to C \to Y' \to X' \to 0$ with $X, X' \in \mcX$ and $Y, Y' \in \omega^\wedge$. We show that $\mcX  = {}^{\perp_1}(\omega^\wedge) \cap \mcX^\wedge$. Since $\omega$ is $\mcX$-injective, we have from \cite[dual of Lemma 2.13 (a)]{MS} that $\Ext^1_{\mcC}(\mcX,\omega^\wedge) = 0$  and so the containment $\subseteq$ holds true. Now, let $C \in {}^{\perp_1}(\omega^\wedge) \cap \mcX^\wedge$ and consider $\eta$ as above. Since $C \in {}^{\perp_1}(\omega^\wedge)$ we get that $\eta$ splits and so $C$ is a direct summand of $X \in \mcX$. Hence, $C \in \mcX$ and the equality holds true. We thus have $\mcX^\wedge \in {\rm lCuts}(\mcX,\omega^\wedge)$. For the assertion $\mcX^\wedge \in {\rm rCuts}(\mcX,\omega^\wedge)$, suppose that $\omega$ is closed under extensions and direct summands. From part (1), $\omega^\wedge$ is closed under direct summands. Then, it remains to show the equality $\omega^\wedge = \mcX^{\perp_1} \cap \mcX^\wedge$. The containment $\subseteq$ follows from \cite[dual of Lemma 2.13]{MS} again. Now, let $C \in \mcX^{\perp_1} \cap \mcX^\wedge$ and consider $\epsilon$ as above. Notice that $\epsilon$ splits since $C \in \mcX^{\perp_1}$. Then, $C$ is a direct summand of an object in $\omega^\wedge$. Therefore, $C \in \omega^\wedge$. 

\item The first part follows by Lemma \ref{lem:main_lemma} (2) and (3). For the second part, we have now that $\mcX^\wedge$ and $\omega$ are closed under direct summands. Now let us show the equality $\omega = {}^{\perp_1}(\mcX^\wedge) \cap \mcX^\wedge$. The inclusion $(\subseteq)$ is clear since $\omega$ is $\mcX$-projective. Now let $C \in {}^{\perp_1}(\mcX^\wedge) \cap \mcX^\wedge$. Then there is a short exact sequence $\xi \colon 0 \to K \to X \to C \to 0$ with $X \in \mcX$ and $K \in \mcX^\wedge$. Since $C \in {}^{\perp_1}(\mcX^\wedge)$, we have that $\xi$ splits and so $C$ is a direct summand of $X$. It follows that $C \in \mcX$. Then, since $\omega$ is a generator in $\mcX$, there exists a short exact sequence $\eta \colon 0 \to X' \to W \to C \to 0$ with $W \in \omega$ and $X' \in \mcX$. Again, using the fact that $C \in {}^{\perp_1}(\mcX^\wedge)$, we have that $\eta$ splits, and so $C$ is a direct summand of $W$. Therefore, $C \in \omega$. The other equality $\mcX^\wedge = \omega^{\perp_1} \cap \mcX^\wedge$ follows from \cite[dual of Lemma 2.13 (a)]{MS}. 

It is clear that $(\omega,\mcX^\wedge)$ satisfies \rccpthree. On the other hand, to show \lccpthree \ let $C \in \mcX^\wedge$ and set $n := \resdim_{\mcX}(C)$. Then, we have a short exact sequence $0 \to K \xrightarrow{\alpha} X \to C \to 0$ where $X \in \mcX$ and $\resdim_{\mcX}(K) \leq n-1$. Since $\omega$ is a relative generator in $\mcX$, there exists also a short exact sequence $0 \to X' \to W \xrightarrow{p} X \to 0$ with $W \in \omega$ and $X' \in \mcX$. Taking the pullback of $\alpha$ and $p$ gives rise to the following solid diagram:
\[
\begin{tikzpicture}[description/.style={fill=white,inner sep=2pt}]
\matrix (m) [matrix of math nodes, row sep=2.5em, column sep=2.5em, text height=1.25ex, text depth=0.25ex]
{ 
X' & X' & {}\\
W' & W & C\\
K & X & C\\
};
\path[>->]
(m-1-1) edge (m-2-1) (m-2-1) edge (m-2-2)
(m-1-2) edge (m-2-2) (m-3-1) edge (m-3-2)
;
\path[->>]
(m-2-1) edge (m-3-1) (m-2-2) edge (m-3-2)
(m-2-2) edge (m-2-3) (m-3-2) edge (m-3-3)
;
\path[->]
(m-2-1)-- node[pos=0.5] {\footnotesize$\mbox{\bf pb}$} (m-3-2)
;
\path[-,font=\scriptsize]
(m-1-1) edge [double, thick, double distance=2pt] (m-1-2)
(m-2-3) edge [double, thick, double distance=2pt] (m-3-3)
;
\end{tikzpicture}
\]
Using the fact that $\mcX^\wedge$ is closed under extensions, the central row is the desired short exact sequence. 
\end{enumerate}
\end{proof}

\begin{proposition}\label{prop:GP_cotorsion_cuts}
Let $(\mcX,\mcY)$ be a GP-admissible pair in $\mcC$ such that $\mcX$ and $\omega := \mcX \cap \mcY$ are closed under direct summands. Then, the following assertions hold true:
\begin{enumerate}
\item $(\mcX,\omega^\wedge)$ is a complete cotorsion pair cut along $\mcX^\wedge$.

\item $(\mcGP_{(\mcX,\mcY)},\omega^\wedge)$ is a complete cotorsion pair cut along $\mcGP^\wedge_{(\mcX,\mcY)}$.

\item $(\mcGP_{(\mcX,\omega)},\omega^\wedge)$ is a complete cotorsion pair cut along $\mcGP^\wedge_{(\mcX,\omega)}$.
\end{enumerate}
\end{proposition}

\begin{proof}
We first note that the class $\omega$ is closed under extensions by \cite[Corollary 3.25 (a)]{BMS}. Then part (1) follows by Proposition \ref{prop:cotorsion_cut_AB_theory} (2). For part (2), by \cite[Theorem 3.34 (a) and (c)]{BMS} we have that $(\mcGP_{(\mcX,\mcY)},\mcY)$ is a GP-admissible pair with $\mcGP_{(\mcX,\mcY)} \cap \mcY = \omega$. Moreover, by \cite[Corollary 3.33]{BMS} we also know that $\mcGP_{(\mcX,\mcY)}$ is closed under direct summands. Thus, the results follows as an application of part (1). Finally, part (3) is in turn an application of part (2) by considering the GP-admissible pair $(\mcX,\omega)$. 
\end{proof}


\section{\textbf{Correspondences between complete cut cotorsion pairs, cut Frobenius pairs and cut Auslander-Buchweitz contexts}}\label{sec:versus}

In this section we study the interplay between cut Frobenius pairs and cut AB contexts. This interaction will depend on two equivalence relations: one defined on $\mfF_{\mcS}$, the class of left Frobenius pairs cut along $\mcS$, and the other one defined on the class $\mfC_{\mcS}$ of left weak AB-contexts cut along $\mcS$. Using these relations, we shall prove that there exists a one-to-one correspondence between the corresponding quotient classes. We shall also show that there exists a one-to-one correspondence between cut AB contexts and complete cut cotorsion pairs. For this, we consider an equivalence relation defined on the class $\mfP_{\mcS}$ of complete cotorsion pairs $(\mcF,\mcG)$ cut along the smallest thick subcategory containing $\mcF$ and which satisfy a certain relation with $\mcS$. 

Let us commence proving the following lemma, which is a relativization of \cite[Theorem 3.6]{BMPS}.

\begin{lemma}\label{lem:relative_equality}
Let $\mcX$, $\omega$ and $\mcS$ be classes of objects in $\mcC$ satisfying the following conditions:
\begin{enumerate}
\item $\omega$ is closed under extensions and isomorphisms; 

\item $\omega \cap \mcS$ is closed under direct summands and a relative generator in $\omega$;

\item $\omega \cap \mcS \subseteq \mcX \cap \mcS$;

\item $\mcX \cap \mcS$ is closed under epi-kernels; and 

\item $\id_{\mcX \cap \mcS}(\omega \cap \mcS) = 0$.
\end{enumerate} 
Then, $\omega \cap \mcS = \mcX \cap \omega^\wedge \cap \mcS$. In particular, this equality holds if $(\mcX,\omega)$ is a left Frobenius pair cut along $\mcS$. 
\end{lemma}

\begin{proof}
The containment $\omega \cap \mcS \subseteq \mcX \cap \omega^\wedge \cap \mcS$ is clear. For the converse, let $M \in \mcX \cap \omega^\wedge \cap \mcS$ and $0 \to W_n \to W_{n-1} \to \cdots \to W_1 \to W_0 \to M \to 0$ be an $\omega$-resolution of $M$. If $n = 0$, then $M \in \omega \cap \mcS$ since $\omega$ is closed under isomorphisms. So we can assume that $n \geq 1$. Now, since $\mcX \cap \mcS$ is closed under epi-kernels, we have by Lemma \ref{lem:technical_lemma_1} and its notation therein that there are exact sequences 
\[
\eta_j \colon 0 \to X_{j+1} \to F_j \to X_j \to 0
\] 
where $X_0 = M$ and $X_j \in \mcX \cap \mcS$ for all $1 \leq j \leq n-1$. Notice that $X_n \in \omega \cap \mcS$ since $G_{n-1}, E_n := W_n \in \omega$. Thus, $\eta_{n-1}$ splits and $X_{n-1} \in \omega \cap \mcS$ since $\id_{\mcX \cap \mcS}(\omega \cap \mcS) = 0$ and $\omega \cap \mcS$ is closed under direct summands. Using again the preceding argument, we get that $X_j \in \omega \cap \mcS$ for all $1 \leq j \leq n$. Therefore, $\eta_0$ splits and so $M \in \omega \cap \mcS$. 
\end{proof}

\begin{example}
We know from \cite[Proposition 6.1]{BMPS} that $(\mcX,\omega) := (\Mod(R),\mcP(R))$ is a left Frobenius pair cut along $\mcGP(R)$, for any ring $R$. Notice that in this case the equality $\mcX \cap \omega^\wedge = \omega$ does not necessarily hold true, while $\mcX \cap \omega^\wedge \cap \mcS = \omega \cap \mcS$ does by the previous lemma. 
\end{example}


\subsection*{Cut Frobenius pairs vs. cut AB-contexts}

We are now ready to give a precise definition for the equivalence relations mentioned at the beginning of this section.

\begin{definition}\label{def:Frobenius_and_AB_relations}
Let $\mcS$ be a class of objects in $\mcC$. For $(\mcX,\omega), (\mcX',\omega') \in \mfF_{\mcS}$ and $(\mcA,\mcB)$, $(\mcA',\mcB') \in \mfC_{\mcS}$, we shall say that:
\begin{enumerate}
\item $(\mcX,\omega)$ is \textbf{related} to $(\mcX',\omega')$ in $\mfF_{\mcS}$, denoted $(\mcX,\omega) \sim (\mcX',\omega')$, if $\mcX \cap \mcS = \mcX' \cap \mcS$ and $\omega \cap \mcS = \omega' \cap \mcS$;

\item $(\mcA,\mcB)$ is \textbf{related} to $(\mcA',\mcB')$ in $\mfC_{\mcS}$, denoted $(\mcA,\mcB) \sim (\mcA',\mcB')$, if $\mcA \cap \mcS = \mcA' \cap \mcS$ and $\mcA \cap \mcB \cap \mcS = \mcA' \cap \mcB' \cap \mcS$. 
\end{enumerate} 
\end{definition}

Notice that (1) and (2) in the previous definition are equivalence relations. We denote by $[\mcX,\omega]_{\mfF_{\mcS}}$ the equivalence class of $(\mcX,\omega)$ in $\mfF_{\mcS} / \sim$. Similarly, $[\mcA,\mcB]_{\mfC_{\mcS}}$ denotes the equivalence class of $(\mcA,\mcB)$ in $\mfC_{\mcS} / \sim$.

\begin{example}
Let $(\mcX,\mcY)$ be a GP-admissible pair with $\mcX$ closed under epi-kernels and direct summands, and such that $\mcY$ is right thick. We know from Example \ref{ex:From_GP_to_AB} that $(\mcX,\mcY)$ and $(\mcGP_{(\mcX,\omega)},\mcY)$ are left weak AB-contexts cut along $\mcX^\wedge$, that is, $(\mcX,\mcY), (\mcGP_{(\mcX,\omega)},\mcY) \in \mfC_{\mcX^\wedge}$. Also, one can verify using \cite[Theorem 2.8]{BMPS} and \cite[Theorem 3.4 (c)]{BMS} that $\mcX = \mcGP_{(\mcX,\omega)} \cap \mcX^\wedge$ and $\mcGP_{(\mcX,\omega)} \cap \mcY \cap \mcX^\wedge = \omega$. It follows that $(\mcX,\mcY)$ and $(\mcGP_{(\mcX,\omega)},\mcY)$ are related in $\mfC_{\mcX^\wedge}$. 
\end{example}

The following result is a generalization of \cite[part 1. of Theorem 5.4]{BMPS}, one of the main results in this reference, where the authors establish a one-to-one correspondence between left Frobenius pairs and left weak AB-contexts.

\begin{theorem}[first correspondence theorem]\label{theo:correspondence_1}
Let $\mcS$ be a class of objects in $\mcC$, which is closed under epi-kernels and mono-cokernels. Then, there is a one-to-one correspondence 
\begin{align*}
\Phi_{\mcS} \colon & \mfF_{\mcS} / \sim \mbox{} \to \mfC_{\mcS} / \sim & & \text{given by} & [\mcX,\omega]_{\mfF_{\mcS}} & \mapsto [\mcX,\omega^\wedge]_{\mfC_{\mcS}},
\end{align*}
with inverse
\begin{align*}
\Psi_{\mcS} \colon & \mfC_{\mcS} / \sim \mbox{} \to \mfF_{\mcS} / \sim & & \text{given by} & [\mcA,\mcB]_{\mfC_{\mcS}} & \mapsto [\mcA,\mcA \cap \mcB]_{\mfF_{\mcS}}.
\end{align*}
\end{theorem}

\begin{proof}
First, we show that the mappings $\Phi_{\mcS}$ and $\Psi_{\mcS}$ are well defined. On the one hand, for $\Psi_{\mcS}$, we have that $(\mcA,\mcA \cap \mcB) \in \mfF_{\mcS}$ for every $(\mcA,\mcB) \in \mfC_{\mcS}$, by definition of cut left weak AB-context. Also, it is clear that $\Psi_{\mcS}([\mcA,\mcB]_{\mfC_{\mcS}})$ does not depend on the chosen representative $(\mcA,\mcB) \in \mfC_{\mcS}$. On the other hand, $\Phi_{\mcS}$ does not depend on representatives either by Lemma \ref{lem:relative_equality}. Now we prove that if $(\mcX,\omega) \in \mfF_{\mcS}$ then $(\mcX,\omega^\wedge) \in \mfC_{\mcS}$ by checking conditions \lcABone, \lcABtwo \ and \lcABthree \ in Definition \ref{def:cut_AB_context}:
\begin{itemize}
\item $(\mcX,\mcX \cap \omega^\wedge)$ is a left Frobenius pair cut along $\mcS$:

Clearly, $\mcX$ is left thick by \lcfpone \ in Definition \ref{def:relative_Frobenius_pair}. Now by Lemma \ref{lem:relative_equality}, we have that $(\mcX \cap \mcS, \mcX \cap \omega^\wedge \cap \mcS) = (\mcX \cap \mcS,\omega \cap \mcS)$, which is a left Frobenius pair in $\mcC$. 

In order to show that $\omega \cap \mcS$ is an $(\mcX \cap \omega^\wedge)$-projective relative generator in $\mcX \cap \omega^\wedge$, note first that $\pd_{\omega^\wedge}(\omega \cap \mcS) = 0$ by Lemma \ref{lem:main_lemma} (1). Now let $M \in \mcX \cap \omega^\wedge$. Using again Lemma \ref{lem:main_lemma} (1), there exists a short exact sequence $
0 \to M' \to P \to M \to 0$ with $P \in \omega \cap \mcS$ and $M' \in \omega^\wedge$. Since $\mcX$ is left thick and $\omega \cap \mcS \subseteq \mcX \cap \mcS$, we get that $M' \in \mcX \cap \omega^\wedge$.

Finally, we note that $\mcX \cap \omega^\wedge$ is closed under extensions and direct summands. Indeed, since $(\mcX,\omega) \in \mfF_{\mcS}$, we have by Remark \ref{rem:properties_of_cut_Frobenius_pairs} (3) that the statements of Lemma \ref{lem:main_lemma} hold true. In particular, $\omega^\wedge$ is closed under extensions and direct summands, and so the same holds for $\mcX \cap \omega^\wedge$ since $\mcX$ is left thick. 

\item $\omega^\wedge \cap \mcS$ is right thick and $\omega^\wedge \cap \mcS \subseteq (\mcX \cap \mcS)^\wedge$:

First, note that since $(\mcX \cap \mcS, \omega \cap \mcS)$ is a left Frobenius pair in $\mcC$, we have by \cite[Theorem 5.4]{BMPS} that $(\omega \cap \mcS)^\wedge$ is right thick and that $\omega \cap \mcS \subseteq \mcX \cap \mcS$. The rest follows by Lemma \ref{lem:main_lemma} (4). 
\end{itemize}

Let us conclude this proof showing that $\Psi_{\mcS}$ and $\Phi_{\mcS}$ are inverse to each other. For any $(\mcX,\omega) \in \mfF_{\mcS}$ we have that
\[
\Psi_{\mcS} \circ \Phi_{\mcS}([\mcX,\omega]_{\mfF_{\mcS}}) = \Psi_{\mcS}([\mcX,\omega^\wedge]_{\mfC_{\mcS}}) = [\mcX, \mcX \cap \omega^\wedge]_{\mfF_{\mcS}} = [\mcX,\omega]_{\mfF_{\mcS}},
\]
where the relation $(\mcX,\mcX \cap \omega^\wedge) \sim (\mcX,\omega)$ holds true by Lemma \ref{lem:relative_equality}. The equality $\Phi_{\mcS} \circ \Psi_{\mcS}([\mcA,\mcB]_{\mfC_{\mcS}}) = [\mcA,\mcB]_{\mfC_{\mcS}}$ follows similarly. 
\end{proof}


\subsection*{Cut AB-contexts vs. complete cut cotorsion pairs}

In order to present the second correspondence mentioned earlier, let us point out the following facts about complete cut cotorsion pairs.

\begin{lemma}\label{lem:Frobenius_from_cut_cotorsion}
Let $(\mcF,\mcG)$ be a complete cotorsion pair cut along $\Thick(\mcF)$ with $\id_{\mcF}(\mcG) = 0$. Then, the following statements hold true:
\begin{enumerate}
\item $(\mcF,\mcF \cap \mcG)$ is a left Frobenius pair in $\mcC$.

\item $\mcF \cap \mcG = \mcF \cap \mcF^{\perp_1} = \mcF \cap (\mcF \cap \mcG)^\wedge$.

\item $(\mcF \cap \mcG)^\wedge = \mcF^{\perp} \cap \mcF^\wedge$.

\item $\mcF^\wedge = \Thick(\mcF)$. 
\end{enumerate}
\end{lemma}

\begin{proof}
Let us first check conditions \lfpone, \lfptwo \ and \lfpthree \ in Definition \ref{def:Frobenius_pair} for the pair $(\mcF,\mcF \cap \mcG)$. Since $(\mcF,\mcG)$ is a complete cotorsion pair cut along $\Thick(\mcF)$, we have the equality $\mcF = {}^{\perp_1}\mcG \cap \Thick(\mcF)$. Then, $\mcF$ is closed under extensions and direct summands. The fact that $\mcF$ is closed under epi-kernels follows from the equality $\id_{\mcF}(\mcG) = 0$. We thus have that $\mcF$ is left thick. Note from \rccptwo \ in Definition \ref{def:cut_cotorsion_pair} and $\mcF \cap \Thick(\mcF) = \mcF$ that $\mcF \cap \mcG = \mcF \cap \mcF^{\perp_1}$, which is closed under direct summands. So it is only left to show that $\mcF \cap \mcG$ is an $\mcF$-injective relative cogenerator in $\mcF$. We have that $\id_{\mcF}(\mcF \cap \mcG) = 0$ directly from the assumption $\id_{\mcF}(\mcG) = 0$. Now let $F \in \mcF$. Since $(\mcF,\mcG)$ is a complete cotorsion pair cut along $\Thick(\mcF)$, there exists a short exact sequence $0 \to F \to G \to F' \to 0$ with $G \in \mcG$ and $F' \in \mcF$. Moreover, $G \in \mcF \cap \mcG$ since $\mcF$ is closed under extensions. Hence, part (1) follows. 

We previously proved that $\mcF \cap \mcG = \mcF \cap \mcF^{\perp_1}$. The rest of the equalities appearing in (2), (3) and (4) follow from part (1) and \cite[Theorems 2.11 and 3.6]{BMPS}.
\end{proof}

\begin{lemma}\label{lem:correspondence_2}
Let $\mcS$ be a thick subcategory of $\mcC$ and $(\mcF,\mcG)$ be a complete cotorsion pair cut along $\Thick(\mcF)$, with $\id_{\mcF}(\mcG) = 0$. If $\mcF \cap \mcG \cap \mcS$ is both a relative generator and cogenerator in $\mcF \cap \mcG$, then $(\mcF, \mcF \cap \mcG)$ is a left Frobenius pair cut along $\mcS$. 
\end{lemma}

\begin{proof}
We need to verify conditions \lcfpone, \lcfptwo, \lcfpthree \ and \lcfpfour \ in Definition \ref{def:relative_Frobenius_pair} for the classes $\mcF$, $\mcF \cap \mcG$ and $\mcS$. First note that \lcfpone \ and \lcfpfour \ hold by the previous lemma. Now, $\id_{\mcF}(\mcG) = 0$ implies that $\id_{\mcF \cap \mcG \cap \mcS}(\mcF \cap \mcG) = 0$ and $\id_{\mcF \cap \mcG}(\mcF \cap \mcG \cap \mcS) = 0$. In particular, \lcfpthree \ holds. It remains to check \lcfptwo, that is, that $(\mcF \cap \mcS, \mcF \cap \mcG \cap \mcS)$ is a left Frobenius pair in $\mcC$. Note first that $\mcF \cap \mcS$ is left thick, since $\mcS$ is thick by hypothesis and $\mcF$ is left thick by Lemma \ref{lem:Frobenius_from_cut_cotorsion}. Furthermore, by the previous lemma we have that $\mcF \cap \mcG \cap \mcS = \mcF \cap \mcF^{\perp_1} \cap \mcS$, and then $\mcF \cap \mcG \cap \mcS$ is closed under direct summands. Finally, since $\id_{\mcF \cap \mcG \cap \mcS}(\mcF \cap \mcG) = 0$, it suffices to show that $\mcF \cap \mcG \cap \mcS$ is a relative cogenerator in $\mcF \cap \mcS$. So for any $F \in \mcF \cap \mcS$, by the relative right completeness of $(\mcF,\mcG)$ there is a short exact sequence $0 \to F \to G \to F' \to 0$ with $G \in \mcF \cap \mcG$ and $F' \in \mcF$. Now by using that $\mcF \cap \mcG \cap \mcS$ is a relative cogenerator in $\mcF \cap \mcG$, we can construct the following solid diagram  
\begin{equation}\label{4-1dig}
\parbox{1.5in}{
\begin{tikzpicture}[description/.style={fill=white,inner sep=2pt}] 
\matrix (m) [ampersand replacement=\&, matrix of math nodes, row sep=2.5em, column sep=2.5em, text height=1.25ex, text depth=0.25ex] 
{ 
F \& G \& F' \\
F \& L \& K \\
{} \& G' \& G' \\
}; 
\path[->] 
(m-1-2)-- node[pos=0.5] {\footnotesize$\mbox{\bf po}$} (m-2-3)
; 
\path[-,font=\scriptsize]
(m-1-1) edge [double, thick, double distance=2pt] (m-2-1)
(m-3-2) edge [double, thick, double distance=2pt] (m-3-3)
;
\path[>->]
(m-1-1) edge (m-1-2)
(m-2-1) edge (m-2-2)
(m-1-2) edge (m-2-2)
(m-1-3) edge (m-2-3)
;
\path[->>]
(m-1-2) edge (m-1-3)
(m-2-2) edge (m-2-3)
(m-2-2) edge (m-3-2)
(m-2-3) edge (m-3-3)
;
\end{tikzpicture}
}
\end{equation}
with $L \in \mcF \cap \mcG \cap \mcS$ and $G' \in \mcF \cap \mcG$. Note also that $K \in \mcF \cap \mcS$ since $\mcF$ is closed under extensions and $\mcS$ is thick. Thus, the central row in \eqref{4-1dig} is the desired sequence for $F$. Therefore, the result follows. 
\end{proof}

For a fixed class of objects $\mcS \subseteq \mcC$, let $\mfP_{\mcS}$ denote the class of pairs $(\mcF,\mcG)$ of classes of objects in $\mcC$ such that $\Thick(\mcF) \in {\rm Cuts}(\mcF,\mcG)$, $\id_{\mcF}(\mcG) = 0$ and $\mcF \cap \mcG \cap \mcS$ is both a relative generator and cogenerator in $\mcF \cap \mcG$.

\begin{definition}
Let $(\mcF,\mcG), (\mcF',\mcG') \in \mfP_{\mcS}$. We shall say that $(\mcF,\mcG)$ is \textbf{related} to $(\mcF',\mcG')$ in $\mfP_{\mcS}$, denoted $(\mcF,\mcG) \sim (\mcF',\mcG')$, if $\mcF \cap \mcS = \mcF' \cap \mcS$ and $\mcF \cap \mcG \cap \mcS = \mcF' \cap \mcG' \cap \mcS$. 
\end{definition}

Note that $\sim$ is an equivalence relation in $\mfP_{\mcS}$. In what follows, let us denote by $[\mcF,\mcG]_{\mfP_{\mcS}}$ the equivalence class of the representative $(\mcF,\mcG) \in \mfP_{\mcS}$.

\begin{example} 
Let $(\mcX,\omega)$ be a left Frobenius pair in $\mcC$. We know from Proposition \ref{prop:cuts_from_Frobenius_pairs} that $(\mcX,\omega^\wedge)$ and $(\omega,\mcX^{\perp_1})$ are complete cotorsion pairs cut along $\omega^\wedge$. These pairs are related in $\mfP_{\omega^\wedge}$ by \cite[Proposition 2.7]{BMPS}. In particular, $(\mcGP(R),\mcP(R)^\wedge) \sim (\mcP(R),\mcGP(R)^{\perp_1})$ in $\mfP_{\mcP(R)^\wedge}$. 
\end{example}

We are now ready to show the correspondence between the quotient classes $\mfP_{\mcS} / \sim$ and $\mfC_{\mcS} / \sim$ in the following result, which generalizes \cite[part 2. of Theorem 5.4]{BMPS}.

\begin{theorem}[second correspondence theorem]\label{theo:correspondence_2}
Let $\mcS \subseteq \mcC$ be a thick subcategory. Then, there is a one-to-one correspondence 
\begin{align*}
\Lambda_{\mcS} \colon & \mfP_{\mcS} / \sim \mbox{} \to \mfC_{\mcS} / \sim & & \text{given by} & [\mcF,\mcG]_{\mfP_{\mcS}} & \mapsto [\mcF,(\mcF \cap \mcG)^\wedge]_{\mfC_{\mcS}},
\end{align*}
with inverse
\begin{align*}
\Upsilon_{\mcS} \colon & \mfC_{\mcS} / \sim \mbox{} \to \mfP_{\mcS} / \sim & & \text{given by} & [\mcA,\mcB]_{\mfC_{\mcS}} & \mapsto [\mcA \cap \mcS,\mcB \cap \mcS]_{\mfP_{\mcS}}.
\end{align*}
\end{theorem}

\begin{proof}
First, we check that the mappings $\Lambda_{\mcS}$ and $\Upsilon_{\mcS}$ are well defined. On the one hand, let $(\mcF,\mcG), (\mcF',\mcG') \in \mfP_{\mcS}$ such that $(\mcF,\mcG) \sim (\mcF',\mcG')$. By Theorem \ref{theo:correspondence_1} and Lemma \ref{lem:correspondence_2}, we get that $(\mcF,(\mcF \cap \mcG)^\wedge)$ is a left weak AB-context cut along $\mcS$. Moreover, by Lemma \ref{lem:Frobenius_from_cut_cotorsion} we have the following equalities:
\begin{align*}
\mcF \cap \mcG \cap \mcS & = \mcF \cap (\mcF \cap \mcG)^\wedge \cap \mcS & & \text{and} & \mcF' \cap \mcG' \cap \mcS & = \mcF' \cap (\mcF' \cap \mcG')^\wedge \cap \mcS.
\end{align*}
Thus, $(\mcF,(\mcF \cap \mcG)^\wedge) \sim (\mcF', (\mcF' \cap \mcG')^\wedge)$, and hence $\Lambda_{\mcS}$ does not depend on representatives. On the other hand, consider now $(\mcA,\mcB) \in \mfC_{\mcS}$. By Remark \ref{rem:vs_first_approach} we have that $(\mcA \cap \mcS, \mcB \cap \mcS)$ is a complete cotorsion pair cut along $\Thick(\mcA \cap \mcS)$ with $\id_{\mcA \cap \mcS}(\mcB \cap \mcS) = 0$. The fact that $\Upsilon_{\mcS}$ does not depend on representatives is clear. 

Finally, we show that $\Lambda_{\mcS}$ and $\Upsilon_{\mcS}$ are inverse to each other. For any $(\mcF,\mcG) \in \mfP_{\mcS}$, we have:
\[
\Upsilon_{\mcS} \circ \Lambda_{\mcS}([\mcF,\mcG]_{\mfP_{\mcS}}) = \Upsilon_{\mcS}([\mcF,(\mcF \cap \mcG)^\wedge]_{\mfC_{\mcS}}) = [\mcF \cap \mcS,(\mcF \cap \mcG)^\wedge \cap \mcS]_{\mfP_{\mcS}} = [\mcF,\mcG]_{\mfP_{\mcS}},
\]
where $(\mcF \cap \mcS, (\mcF \cap \mcG)^\wedge \cap \mcS) \sim (\mcF,\mcG)$ by Lemma \ref{lem:Frobenius_from_cut_cotorsion}. Now let $(\mcA,\mcB)$ be a left weak AB-context cut along $\mcS$. Then, 
\[
\Lambda_{\mcS} \circ \Upsilon_{\mcS}([\mcA,\mcB]_{\mfC_{\mcS}}) = \Lambda_{\mcS}([\mcA \cap \mcS,\mcB \cap \mcS]_{\mfP_{\mcS}}) = [\mcA \cap \mcS,(\mcA \cap \mcB \cap \mcS)^\wedge]_{\mfC_{\mcS}} = [\mcA,\mcB]_{\mfC_{\mcS}},
\]
where $(\mcA \cap \mcS,(\mcA \cap \mcB \cap \mcS)^\wedge) \sim (\mcA \cap \mcS,\mcB \cap \mcS) \sim (\mcA,\mcB)$ by Remark \ref{rem:vs_first_approach}. Therefore, the result follows. 
\end{proof}


\section{\textbf{Applications and more examples}}\label{sec:applications}

In this last section we present more detailed examples and applications of the theory of complete cut cotorsion pairs. Our examples will be presented in settings more general than $\Mod(R)$, namely, in categories of chain complexes, quasi-coherent sheaves, and modules over extriangulated categories. We also explore some relations with the Finitistic Dimension Conjecture and Serre subcategories.


\subsection*{Induced cotorsion cuts in chain complexes}

In the category $\Ch(\mcC)$ of chain complexes in an abelian category $\mcC$, we shall write the extension bifunctors $\Ext^i_{\Ch(\mcC)}$ as $\Ext^i_{\Ch}$, for simplicity. 

In this section, we induce complete cut cotorsion pairs in chain complexes in the following sense: we consider complete cotorsion pairs $(\mathcal{A,B})$ in $\mcC$ cut along $\mcS \subseteq \mcC$, and study under which conditions it is possible to get complete cut cotorsion pairs in $\Ch(\mcC)$ involving the classes $\widetilde{\mcA}$ of $\mcA$-complexes and $\widetilde{\mcB}$ of $\mcB$-complexes. Our motivation comes from Gillespie's result \cite[Proposition 3.6]{GillespieFlat}, which asserts that if $(\mcA,\mcB)$ is a complete cotorsion pair in $\mcC$ with enough $\mcA$-objects and enough $\mcB$-objects, then $(\tilde{\mcA},{\rm dg}\tilde{\mcB})$ and $({\rm dg}\tilde{\mcA},\tilde{\mcB})$ are complete cotorsion pairs in $\Ch(\mcC)$. 

In the case where $(\mcA,\mcB)$ is a complete cotorsion pair cut along $\mcS \subseteq \mcC$, we shall determine a subcategory of $\Ch(\mcC)$ along which the classes $\tilde{\mcA}$, $\Ch_{\rm acy}(\widetilde{\mcA};\mcB)$, $\tilde{\mcB}$ and $\Ch_{\rm acy}(\mcA;\widetilde{\mcB})$ (see \cite[Definition 6.4]{HMP}) form a complete cut cotorsion pair. It is known from Yang and Liu's \cite[Lemma 3.4]{YangLiu} that if $(\mcA,\mcB)$ is a hereditary complete cotorsion pair in $\Ch(R)$ (with $R$ an associative ring with unit) then every exact complex admits an special $\tilde{\mcA}$-precover and a special $\tilde{\mcB}$-preenvelope. This suggests that the class $\tilde{\mcS}$ should be considered as a possible cotorsion cut. Below we impose some conditions on $\mcA$, $\mcB$ and $\mcS$ so that $(\tilde{\mcA},\Ch_{\rm acy}(\widetilde{\mcA};\mcB))$ is a complete cotorsion pair cut along $\tilde{\mcS}$. 

Let us recall and specify some of the notation previously displayed. Let $\mcX$, $\mcA$ and $\mcB$ be classes of objects in $\mcC$:
\begin{itemize}
\item $\Ch(\mcX)$ denotes the class of complexes $X_\bullet \in \Ch(\mcC)$ such that $X_m \in \mcX$ for every $m \in \mathbb{Z}$.

\item $\widetilde{\mcX}$ denotes the class of exact complexes $X_\bullet \in \Ch(\mcC)$ such that $Z_m(X_\bullet) \in \mcX$ for every $m \in \mathbb{Z}$.  

\item $\Ch_{\rm acy}(\mathcal{A};\widetilde{\mathcal{B}})$ denotes the class of complexes $X_\bullet \in \Ch(\mcA)$ such that the internal hom $\mathcal{H}{om}(X_\bullet,B_\bullet)$\footnote{The reader can recall the definition of $\mathcal{H}{om}(X_\bullet,B_\bullet)$ from \cite[Section 2.1]{GR}.} is an exact complex of abelian groups whenever $B_\bullet \in \widetilde{\mcB}$. Dually, $\Ch_{\rm acy}(\widetilde{\mathcal{A}};\mathcal{B})$ denotes the class of complexes $Y_\bullet \in \Ch(\mcB)$ such that $\mathcal{H}{om}(A_\bullet,Y_\bullet)$ is exact for every $A_\bullet \in \widetilde{\mcA}$. 
\end{itemize}

\begin{proposition}\label{cutcomplexes}
Let $\mcA$ and $\mcB$ be two classes of objects in $\mcC$ closed under extensions and such that $\Ext_{\mcC}^i(\mcA,\mcB) = 0$ for every $i = 1, 2$. If $\mcA$ is closed under direct summands and $\mcS \subseteq \mcC$ is a class of objects in $\mcC$ such that $\mcS \in {\rm lCuts}(\mcA,\mcB)$, then $\widetilde{\mcS} \in {\rm lCuts}(\widetilde{\mcA},\Ch_{\rm acy}(\widetilde{\mcA};\mcB))$.
\end{proposition}

\begin{proof} 
It is clear that $\widetilde{\mcA}$ is closed under diret summands. Moreover, following the proof given in  \cite[Remark 3.10]{HMP}, we obtain for every complex $S_\bullet \in \widetilde{\mcS}$ an exact sequence in $\Ch(\mcC)$ of the form $\eta \colon 0 \to B_{\bullet} \to A_{\bullet} \to S_\bullet \to 0$ where  $A_{\bullet} \in \widetilde{\mcA}$ and $B_{\bullet} \in \widetilde{\mcB} \subseteq \Ch_{\rm acy}(\widetilde{\mcA};\mcB)$ (this inclusion follows by \cite[Lemma 6.6]{HMP}). So, it suffices to show that $\tilde{\mcA} \cap \widetilde{\mcS} = {}^{\perp_1}(\Ch_{\rm acy}(\widetilde{\mcA};\mcB)) \cap \widetilde{\mcS}$. On the one hand, considering $\eta$ as above, with $S_{\bullet} \in {}^{\perp_1}(\Ch_{\rm acy}(\widetilde{\mcA};\mcB)) \cap \tilde{\mcS}$, we have that $\eta$ splits and then $S_{\bullet} \in \widetilde{\mcA}\cap \widetilde{\mcS}$. Thus, the containment $\supseteq$ holds true. For the remaining containment $\subseteq$, since $\Ext_{\mcC}^1(\mcA,\mcB) = 0$, it follows from \cite[Lemma 6.7 (1)]{HMP} that $\Ext_{\Ch}^1(\widetilde{\mcA},\Ch_{\rm acy}(\widetilde{\mcA};\mcB)) = 0$. Therefore, $\tilde{\mcA} \cap \widetilde{\mcS} \subseteq {}^{\perp_1}(\Ch_{\rm acy}(\widetilde{\mcA};\mcB)) \cap \widetilde{\mcS}$.
\end{proof}

In the next result we give some conditions so that the converse of Proposition \ref{cutcomplexes} holds. In what follows, given an object $C \in \mcC$, we denote by $D^1(C)$ the chain complex with $C$ in the first and zeroth positions, and with $0$ elsewhere, where the only nonzero differential is the identity on $C$. The complex $D^0(C)$ is defined similarly.

\begin{proposition}\label{cutcomplexes2}
Let $\mcA$, $\mcB$ and $\mcS$ be classes of objects in $\mcC$, such that $\mcA$ is closed under extensions, $\widetilde{\mcA}$ is closed under direct summands, $0 \in \mcS$ and $\widetilde{\mcS} \in {\rm lCuts}(\widetilde{\mcA},\Ch_{\rm acy}(\widetilde{\mcA};\mcB))$. Then, $\mcS \in {\rm lCuts}(\mcA,\mcB)$ provided that $\Ext^1_{\mcC}(\mcA,\mcB) = 0$ or $D^1(B) \in \Ch_{\rm acy}(\widetilde{\mcA};\mcB)$ for every $B \in \mcB$. 
\end{proposition}

\begin{proof} 
Let us assume that $\Ext^1_{\mcC}(\mcA,\mcB) = 0$ to show that $\mcS \in {\rm lCuts}(\mcA,\mcB)$. The closure for $\mcA$ under direct summands is easy to show. Now, we prove \lccpthree \ in Definition \ref{def:cut_cotorsion_pair}. Let $S\in \mcS$. Notice that $D^0(S) \in \widetilde{\mcS}$ since $0 \in \mcS$. Now using the assumption $\widetilde{\mcS} \in {\rm lCuts}(\widetilde{\mcA},\Ch_{\rm acy}(\widetilde{\mcA};\mcB))$, we have that there exists a short exact sequence $0 \to B_{\bullet} \to A_{\bullet} \to D^0(S) \to 0$ in  $\Ch(\mcC)$ with $A_{\bullet} \in \widetilde{\mcA}$ and $B_{\bullet} \in \Ch_{\rm acy}(\widetilde{\mcA};\mcB)$. Then, \lccpthree \ follows by considering $0$-th entries in this sequence, since $\mcA$ is closed under extensions.

Finally, we see that $\mcA \cap \mcS = {}^{\perp_1} \mcB \cap \mcS$. On the one hand, let $S \in {}^{\perp_1}\mcB \cap \mcS$. From the preceding paragraph, there exists a short exact sequence of the form $0 \to B_0 \to A_0 \to S \to 0$, which splits since $S \in {}^{\perp_1}\mcB \cap \mcS$. Then, $S \in \mcA$ since $\mcA$ is closed under direct summands. On the other hand, for the remaining containment $\subseteq$, we use that $\Ext^1_{\mcC}(\mcA,\mcB) = 0$.

Now we consider the other condition $D^1(B) \in \Ch_{\rm acy}(\widetilde{\mcA};\mcB)$ for every $B \in \mcB$. The closure under direct summands for $\mcA$, condition \lccpthree \ and the containment $\mcA \cap \mcS \supseteq {}^{\perp_1}\mcB \cap \mcS$ follow as in the previous proof. So we only verify that $\mcA \cap \mcS \subseteq {}^{\perp_1}\mcB \cap \mcS$. Let $A \in \mcA \cap \mcS$ and $B \in \mcB$. By \cite[Lemma 3.1]{GillespieFlat} and Proposition \ref{prop:characterization_ccp}, we get that $\Ext_{\mcC}^1(A,B) \cong \Ext_{\Ch}^1(D^0(A),D^1(B)) = 0$ since $D^0(A) \in \widetilde{\mcA} \cap \widetilde{\mcS}$ and $D^1(B) \in \Ch_{\rm acy}(\widetilde{\mcA};\mcB)$.
\end{proof}


\subsection*{Cut cotorsion pairs in the category of quasi-coherent sheaves}

For the notions about sheaves and schemes appearing in this section, we recommend the reader to check Hartshorne's \cite{Hartshorne}. In what follows, $X$ will be a scheme with structure sheaf $\mathcal{O}_X$, and $\Qcoh(X)$ will denote the category of quasi-coherent sheaves on $X$. For simplicity, we shall refer to the objects in $\Qcoh(X)$ simply as ``sheaves''. It is a well-known fact that $\Qcoh(X)$ is a Grothendieck category which in general does not have enough projective objects. The latter makes us think that it is not likely to obtain complete cut cotorsion pairs in $\Qcoh(X)$ involving the class of Gorenstein projective sheaves. This suggests to consider the class $\mfGF(X)$ of Gorenstein flat sheaves on $X$ as a more reliable source to obtain complete cut cotorsion pairs. Indeed, in \cite[Theorem 2.2]{CET} Christensen, Estrada and Thompson proved that $(\mfGF(X),(\mfGF(X))^{\perp_1})$ is a hereditary complete cotorsion pair in $\Qcoh(X)$, provided that $X$ is a semi-separated noetherian scheme. In general, the orthogonal class $(\mfGF(X))^{\perp_1}$ does not always have an explicit description in terms of simpler sheaves, but studying the cotorsion of $\mathcal{A} := \mfGF(X)$ along certain subcategories $\mathcal{S} \subseteq \Qcoh(X)$ could overcome this limitation. The two questions that arise at this point are: (1) what can we expect for a suitable ``local'' orthogonal complement $\mathcal{B}$ of $\mfGF(X)$?, and (2) which cotorsion cut $\mathcal{S}$ do we need to choose for $(\mathcal{A,B})$?. First, we known from \cite[Proposition 6.17]{BMPS} and \cite[Corollary 4.12]{SS} that $(\mathcal{GF}(R),\mathcal{F}(R) \cap (\mathcal{F}(R))^{\perp_1})$ is a left Frobenius pair in $\Mod(R)$, for any ring $R$, where $\mathcal{GF}(R)$ denotes the class of Gorenstein flat $R$-modules. It follows by \cite[Theorem 5.4]{BMPS} that $(\mathcal{GF}(R),(\mathcal{F}(R) \cap (\mathcal{F}(R))^{\perp_1})^\wedge)$ is a $\Thick(\mathcal{GF}(R))$-cotorsion pair in $\Mod(R)$. On the other hand, in case $R$ is commutative, we can regard $\Mod(R)$ as the category $\mathfrak{Qcoh}(\text{Spec}(R))$. The previous, along with the correspondences in Theorem \ref{theo:correspondence_1} and \ref{theo:correspondence_2}, suggests that we should take $\mathcal{B} := (\mfF(X) \cap \mfC(X))^\wedge$ (where $\mfF(X)$ and $\mfC(X) = \mfF(X)^{\perp_1}$ denote the classes of flat and cotorsion  sheaves on $X$, respectively) and $\mathcal{S}$ as a subcategory of $\Qcoh(X)$ equivalent to $\Mod(\mathcal{O}_X(U))$, for some affine open set $U \subseteq X$. One can for instance determine an equivalence between $\mathcal{S}$ and $\Mod(\mathcal{O}_X(U))$ by using the inverse and direct image functors $i^\ast \colon \Qcoh(X) \longrightarrow \Qcoh(U)$ and $i_\ast \colon  \Qcoh(U) \longrightarrow \Qcoh(X)$ induced by the inclusion $i \colon U \to X$.\footnote{Notice that these functors are well defined as they preserve quasi-coherence. See \cite[Proposition II 5.8]{Hartshorne} and \cite[Theorem 1.17]{Iitaka}.} Thus, we shall be working then with the following classes of sheaves:
\begin{itemize}
\item $\mathcal{A} := \mfGF(X)$ for Gorenstein flat sheaves.\footnote{Recall from \cite[Definition 1.2]{CET} that a sheaf $\mathscr{M}$ is \emph{Gorenstein flat} if $\mathscr{M} = Z_0(\mathscr{F}_\bullet)$, where $\mathscr{F}_\bullet$ is an exact complex of flat sheaves on $X$ such that $\mathscr{F}_\bullet \otimes \mathscr{I}$ is an exact complex of $\mathcal{O}_X$-modules for every injective sheaf $\mathscr{I}$.}

\item $\mathcal{B} := (\mfF(X) \cap \mfC(X))^\wedge$. 

\item $\mathcal{S} := i_\ast(\Qcoh(U))$, where $U$ is a affine open subset of $X$. That is, $\mathcal{S}$ is the class of sheaves on $X$ isomorphic to sheaves of the form $i_\ast(\mathscr{N})$, where $\mathscr{N}$ is a sheaf on $U$. 
\end{itemize}

We shall give sufficient conditions on $X$ so that $(\mcA,\mcB)$ is a complete right cotorsion pair cut along $\mcS$. More specifically, the goal of this section is to show the following result.

\begin{theorem}
Let $X$ be a semi-separated noetherian scheme, and $U \subseteq X$ be an affine open subset such that every $\mathcal{O}_X(U)$-module has finite Gorenstein flat dimension. Then, 
\[
(\mfGF(X),(\mfF(X) \cap \mfC(X))^\wedge)
\]
is a complete right cotorsion pair cut along $i_\ast(\Qcoh(U))$.
\end{theorem}

In order to prove this theorem, according to the dual of Proposition \ref{prop:characterization_ccp}, we need to show the following:
\begin{enumerate}
\item $(\mfF(X) \cap \mfC(X))^\wedge$ is closed under direct summands.

\item $\Ext^1_X(\mfGF(X),(\mfF(X) \cap \mfC(X))^\wedge \cap i_\ast(\Qcoh(U))) = 0$.\footnote{By $\Ext^1_X(-,-)$ we mean the extension bifunctor $\Ext^1_{\Qcoh(X)}(-,-)$.} 

\item For every $\mathscr{G} \in i_\ast(\Qcoh(U))$ there exists a short exact sequence 
\[
0 \to \mathscr{G} \to \mathscr{B} \to \mathscr{A} \to 0
\]
with $\mathscr{A} \in \mfGF(X)$ and $\mathscr{B} \in (\mfF(X) \cap \mfC(X))^\wedge$. 
\end{enumerate}

Property (1) is easy to note on any scheme $X$ from our results. First, it is clear that $\mfF(X) \cap \mfC(X)$ is self-orthogonal and closed under extensions and direct summands. Then by Lemma \ref{lem:main_lemma} (3) we have that $(\mfF(X) \cap \mfC(X))^\wedge$ is closed under direct summands. For (2), we need $X$ to be a semi-separated scheme.\footnote{Recall that $X$ is \emph{semi-separated} if it has an open cover by affine open sets with affine intersections. See for instance Neeman's \cite{Neeman}.}

\begin{proposition}\label{prop:property2}
Let $X$ be a quasi-compact\footnote{Recall from \cite[Definition 3.16 (b)]{GW} that a scheme $(X,\mathcal{O}_X)$ is quasi-compact if the underlying topological space $X$ is quasi-compact, that is, if any open covering of $X$ has a finite subcovering.} and semi-separated scheme, and $U \subseteq X$ an affine open. Then, $\Ext^i_X(\mathscr{A},\mathscr{B}) = 0$ for every $\mathscr{A} \in \mfGF(X)$, $\mathscr{B} \in (\mfF(X) \cap \mfC(X))^\wedge \cap i_\ast(\Qcoh(U))$ and $i \geq 1$.
\end{proposition}

\begin{proof}
First, let us write $\mathscr{B} \simeq i_\ast(\mathscr{N})$ where $\mathscr{N} \in \Qcoh(U)$. Since $X$ is semi-separated, it is known by Gillespie's \cite[Lemma 6.5]{GillespieKaplansky} that there is a natural adjunction $\Ext^i_X(\mathscr{A},\mathscr{B}) = \Ext^i_X(\mathscr{A},i_\ast(\mathscr{N})) \cong \Ext^i_U(i^\ast(\mathscr{A}),\mathscr{N})$. On the other hand, since $U$ is an open affine, by a well-known result of Grothendieck (see for instance \cite[Corollary II 5.5]{Hartshorne}) the categories $\Qcoh(U)$ and $\Mod(\mathcal{O}_X(U))$ are equivalent via the mapping $\mathscr{N} \mapsto \mathscr{N}(U)$, and so we can write the previous natural isomorphism as $\Ext^i_X(\mathscr{A},\mathscr{B}) \cong \Ext^i_{\mathcal{O}_X(U)}(i^\ast(\mathscr{A})(U),\mathscr{N}(U))$. The result will follow after showing that $i^\ast(\mathscr{A})(U)$ is a Gorenstein flat $\mathcal{O}_X(U)$-module and that $\mathscr{N}(U) \in (\mathcal{F}(\mathcal{O}_X(U)) \cap \mathcal{C}(\mathcal{O}_X(U)))^\wedge$. 
\begin{itemize}
\item $i^\ast(\mathscr{A})(U) \in \mathcal{GF}(\mathcal{O}_X(U))$: First, we know that $\mathscr{A} = Z_0(\mathscr{F}_\bullet)$ for an exact complex $\mathscr{F}_\bullet$ of flat sheaves on $X$ such that $\mathscr{F}_\bullet \otimes \mathscr{I}$ is exact, for every injective sheaf $\mathscr{I} \in \Qcoh(X)$. By the implication (i) $\Rightarrow$ (ii) in \cite[Proposition 2.10]{CEI} we have that $\mathscr{F}_\bullet(U)$ is an exact complex of flat $\mathcal{O}_X(U)$-modules, such that $\mathscr{F}_\bullet(U) \otimes_{\mathcal{O}_X(U)} I$ is an exact complex of abelian groups for every injective $\mathcal{O}_X(U)$-module $I$, that is, $Z_0(\mathscr{F}_\bullet(U))$ is a Gorenstein flat $\mathcal{O}_X(U)$-module. On the other hand, the functor $i^\ast$ is the restriction on $U$, and so $i^\ast(\mathscr{A})(U) = Z_0(\mathscr{F}_\bullet)|_U(U) = Z_0(\mathscr{F}_\bullet)(U) = Z_0(\mathscr{F}_\bullet(U)) \in \mathcal{GF}(\mathcal{O}_X(U))$.

\item $\mathscr{N}(U) \in (\mcF(\mathcal{O}_X(U)) \cap \mathcal{C}(\mathcal{O}_X(U)))^\wedge$: We proof this claim by induction on the flat-cotorsion resolution dimension of $\mathscr{B}$. So suppose first that $\mathscr{B} = i_\ast(\mathscr{N}) \in \mfF(X) \cap \mfC(X)$. To see that $\mathscr{N}(U)$ is a flat $\mathcal{O}_X(U)$-module, we verify that the functor $\mathscr{N}(U) \otimes_{\mathcal{O}_X(U)} -$ is exact. Since $i_\ast(\mathscr{N})  \otimes -$ is an exact functor and $U$ is an affine open set, we have that 
\[
(i_\ast(\mathscr{N}) \otimes -)(U) = i_\ast(\mathscr{N})(U) \otimes_{\mathcal{O}_X(U)} - = \mathscr{N}(U) \otimes_{\mathcal{O}_X(U)} -
\]
is exact by \cite[Proof of Proposition 3.3]{EEO}. We now show that $\mathscr{N}(U)$ is also a cotorsion $\mathcal{O}_X(U)$-module. For let $F \in \mcF(\mathcal{O}_X(U))$ and consider the sheaf on $U$, $\tilde{F} \in \Qcoh(U)$, associated to $F$ (see \cite[Section II 5.]{Hartshorne}). Note that $\tilde{F} \simeq i^\ast(i_\ast(\tilde{F}))$, and so
\begin{align*}
\Ext^1_{\mathcal{O}_X(U)}(F,\mathscr{N}(U)) & \cong \Ext^1_U(\tilde{F},\mathscr{N}) \cong \Ext^1_U(i^\ast(i_\ast(\tilde{F})),\mathscr{N}) \\
& \cong \Ext^1_X(i_\ast(\tilde{F}),i_\ast(\mathscr{N}))
\end{align*}
Now since $F$ is a flat $\mathcal{O}_X(U)$-module, one can note that $i_\ast(\tilde{F})$ is a flat sheaf on $X$. Also, $i_\ast(\mathscr{N})$ is a cotorsion sheaf on $X$ by assumption. Hence, we obtain $\Ext^1_X(i_\ast(\tilde{F}),i_\ast(\mathscr{N})) = 0$ and so $\mathscr{N}(U)$ is a cotorsion $\mathcal{O}_X(U)$-module. 

So far we have shown that $\mathscr{N}(U) \in \mcF(\mathcal{O}_X(U)) \cap \mcC(\mathcal{O}_X(U))$ if $i_\ast(\mathscr{N}) \in \mfF(X) \cap \mfC(X)$. The more general case where $i_\ast(\mathscr{N})$ has positive flat-cotorsion dimension will follow by using Auslander-Buchweitz Approximation Theory. Specifically, we shall use the equalities
\begin{align}
(\mfF(X) \cap \mfC(X))^\wedge & = \mfF(X)^\perp \cap \mfF(X)^\wedge, \label{eqn:ABscheme} \\
(\mcF(\mathcal{O}_X(U)) \cap \mcC(\mathcal{O}_X(U)))^\wedge & = \mcF(\mathcal{O}_X(U))^\perp \cap \mcF(\mathcal{O}_X(U))^\wedge. \label{eqn:ABring}
\end{align}
In order to prove \eqref{eqn:ABscheme}, note first that since $X$ is a quasi-compact and semi-separated scheme, we have by \cite[Corollary 4.2]{EE} that $(\mfF(X),\mfC(X))$ is a complete cotorsion pair in $\Qcoh(X)$. On the other hand, by \cite[Lemma A.1]{EfimovPositselski} we know that the category $\Qcoh(X)$ has a flat generator. Then, by \cite[Lemma 4.25]{SaorinStovicek} we finally have that $(\mfF(X),\mfC(X))$ is also a hereditary cotorsion pair in $\Qcoh(X)$. It follows that the classes $\mfF(X)$ and $\mfF(X) \cap \mfC(X)$ satisfy the conditions of \cite[Proposition 2.13]{BMPS}, and so \eqref{eqn:ABscheme} holds. The equality \eqref{eqn:ABring} is simply the affine case of \eqref{eqn:ABscheme}. 

Let us now show that $\mathscr{N}(U) \in \mcF(\mathcal{O}_X(U))^\perp \cap \mcF(\mathcal{O}_X(U))^\wedge$. We already know from previous arguments that $\mathscr{N}(U) \in \mcF(\mathcal{O}_X(U))^\perp$. Now let us check $\mathscr{N}(U) \in \mcF(\mathcal{O}_X(U))^\wedge$. Since $i_\ast(\mathscr{N}) \in \mfF(X)^\wedge$, there is an exact sequence $0 \to \mathscr{F}_m \to \mathscr{F}_{m-1} \to \cdots \to \mathscr{F}_1 \to \mathscr{F}_0 \to i_\ast(\mathscr{N})  \to 0$, for some $m > 0$, where $\mathscr{F}_k$ is a flat sheaf on $X$ for every $0 \leq k \leq m$. Apply now the exact functor $i^\ast$ (see \cite[Section 6.3]{GillespieKaplansky}) to obtain the following exact sequence in $\Qcoh(U)$:
\[
\parbox{4.15in}{
\begin{tikzpicture}[description/.style={fill=white,inner sep=2pt}] 
\matrix (m) [ampersand replacement=\&, matrix of math nodes, row sep=2em, column sep=1em, text height=1.25ex, text depth=0.25ex] 
{ 
0 \& i^\ast(\mathscr{F}_m) \& i^\ast(\mathscr{F}_{m-1}) \& \cdots \& i^\ast(\mathscr{F}_1) \& i^\ast(\mathscr{F}_0) \& i^\ast(i_\ast(\mathscr{N})) \& 0 \\
0 \& \mathscr{F}_m|_U \& \mathscr{F}_{m-1}|_U \& \cdots \& \mathscr{F}_1|_U \& \mathscr{F}_0|_U \& \mathscr{N} \& 0 \\
}; 
\path[->] 
(m-1-1) edge (m-1-2) (m-1-2) edge (m-1-3) (m-1-3) edge (m-1-4) (m-1-4) edge (m-1-5) (m-1-5) edge (m-1-6) (m-1-6) edge (m-1-7) (m-1-7) edge (m-1-8) 
(m-2-1) edge (m-2-2) (m-2-2) edge (m-2-3) (m-2-3) edge (m-2-4) (m-2-4) edge (m-2-5) (m-2-5) edge (m-2-6) (m-2-6) edge (m-2-7) (m-2-7) edge (m-2-8)
; 
\path[-,font=\scriptsize]
(m-1-2) edge [double, thick, double distance=2pt] (m-2-2)
(m-1-3) edge [double, thick, double distance=2pt] (m-2-3)
(m-1-5) edge [double, thick, double distance=2pt] (m-2-5)
(m-1-6) edge [double, thick, double distance=2pt] (m-2-6)
(m-1-7) edge [double, thick, double distance=2pt] (m-2-7)
;
\end{tikzpicture}
}
\]
Since the previous sequence is formed by quasi-coherent sheaves on $U$, it remains exact after applying the functor of global sections $\Gamma(U,-)$ (see \cite[Proposition II 5.6]{Hartshorne}):
\[
\parbox{4.15in}{
\footnotesize
\begin{tikzpicture}[description/.style={fill=white,inner sep=2pt}] 
\matrix (m) [ampersand replacement=\&, matrix of math nodes, row sep=2.5em, column sep=1em, text height=1.25ex, text depth=0.25ex] 
{ 
0 \& \Gamma(U,\mathscr{F}_m|_U) \& \Gamma(U,\mathscr{F}_{m-1}|_U) \& \cdots \& \Gamma(U,\mathscr{F}_1|_U) \& \Gamma(U,\mathscr{F}_0|_U) \& \Gamma(U,\mathscr{N}) \& 0 \\
0 \& \mathscr{F}_m(U) \& \mathscr{F}_{m-1}(U) \& \cdots \& \mathscr{F}_1(U) \& \mathscr{F}_0(U) \& \mathscr{N}(U) \& 0 \\
}; 
\path[->] 
(m-1-1) edge (m-1-2) (m-1-2) edge (m-1-3) (m-1-3) edge (m-1-4) (m-1-4) edge (m-1-5) (m-1-5) edge (m-1-6) (m-1-6) edge (m-1-7) (m-1-7) edge (m-1-8) 
(m-2-1) edge (m-2-2) (m-2-2) edge (m-2-3) (m-2-3) edge (m-2-4) (m-2-4) edge (m-2-5) (m-2-5) edge (m-2-6) (m-2-6) edge (m-2-7) (m-2-7) edge (m-2-8)
; 
\path[-,font=\scriptsize]
(m-1-2) edge [double, thick, double distance=2pt] (m-2-2)
(m-1-3) edge [double, thick, double distance=2pt] (m-2-3)
(m-1-5) edge [double, thick, double distance=2pt] (m-2-5)
(m-1-6) edge [double, thick, double distance=2pt] (m-2-6)
(m-1-7) edge [double, thick, double distance=2pt] (m-2-7)
;
\end{tikzpicture}
}
\]
Here, each $\mathscr{F}_k(U)$ is a flat $\mathcal{O}_X(U)$-module by the case $m = 0$ settled previously. Then, $\mathscr{N}(U) \in \mathcal{F}(\mathcal{O}_X(U))^\wedge$. Hence, from \eqref{eqn:ABring} we can conclude that $\mathscr{N}(U) \in (\mcF(\mathcal{O}_X(U)) \cap \mcC(\mathcal{O}_X(U)))^\wedge$. 
\end{itemize}

Therefore, since $i^\ast(\mathscr{A})(U) \in \mathcal{GF}(\mathcal{O}_X(U))$ and $\mathscr{N}(U) \in (\mcF(\mathcal{O}_X(U)) \cap \mcC(\mathcal{O}_X(U)))^\wedge$, we conclude that $\Ext^i_X(\mathscr{A},\mathscr{B}) \cong \Ext^i_{\mathcal{O}_X(U)}(i^\ast(\mathscr{A})(U),\mathscr{N}(U)) = 0$.
\end{proof}

\begin{proposition}
Let $X$ be a noetherian\footnote{Recall that $X$ is noetherian if it has a finite covering by affine open sets $\text{Spec}(A_i)$, where each $A_i$ is a (commutative) noetherian ring.} semi-separated scheme, and $U \subseteq X$ be an open affine subset such that every $\mathcal{O}_X(U)$-module has finite Gorenstein flat dimension. Then, for every $\mathscr{S} \in i_\ast(\Qcoh(U))$ there exists a short exact sequence
\[
0 \to \mathscr{S} \to \mathscr{B} \to \mathscr{A} \to 0
\]
with $\mathscr{A} \in \mfGF(X)$ and $\mathscr{B} \in (\mfF(X) \cap \mfC(X))^\wedge$.
\end{proposition}

\begin{proof}
Let us write $\mathscr{S} \simeq i_\ast(\mathscr{N})$ with $\mathscr{N} \in \Qcoh(U)$. First, note by \cite[Proposition II 5.4]{Hartshorne} that we can write $\mathscr{N} \simeq \tilde{N}$ for some $\mathcal{O}_X(U)$-module $N \in \Mod(\mathcal{O}_X(U))$. Since $N \in \Mod(\mathcal{O}_X(U))$ has finite Gorenstein flat dimension by assumption, and $(\mathcal{GF}(\mathcal{O}_X(U)),\mathcal{F}(\mathcal{O}_X(U)) \cap \mathcal{C}(\mathcal{O}_X(U)))$ is a left Frobenius pair in $\Mod(\mathcal{O}_X(U))$, there exists a  short exact sequence $0 \to N \to B \to A \to 0$ with $A \in \mathcal{GF}(\mathcal{O}_X(U))$ and $B \in (\mathcal{F}(\mathcal{O}_X(U)) \cap \mathcal{C}(\mathcal{O}_X(U)))^\wedge$. The previous induces by \cite[Proposition II 5.2]{Hartshorne} an exact sequence $0 \to \tilde{N} \to \tilde{B} \to \tilde{A} \to 0$ of associated sheaves on $U$. Now since the functor $i_\ast$ is exact by \cite[Lemma 6.5]{GillespieKaplansky}, the previous sequence induces in turn a short exact sequence $0 \to \mathscr{S} \to i_\ast(\tilde{B}) \to i_\ast(\tilde{A}) \to 0$ of sheaves on $X$. The result will follow after showing that $i_\ast(\tilde{A}) \in \mfGF(X)$ and $i_\ast(\tilde{B}) \in (\mfF(X) \cap \mfC(X))^\wedge$:
\begin{itemize}
\item $i_\ast(\tilde{A})$ is a Gorenstein flat sheaf on $X$: Since $A$ is a Gorenstein flat $\mathcal{O}_X(U)$-module, we have that $A = Z_0(F_\bullet)$ for some exact complex $F_\bullet$ of flat $\mathcal{O}_X(U)$-modules such that $F_\bullet \otimes_{\mathcal{O}_X(U)} I$ is exact for every injective $I \in \Mod(\mathcal{O}_X(U))$. Using the assumption that $X$ is a noetherian semi-separated scheme, we can apply \cite[Lemma 4.8]{MurfetSalarian} to deduce that $i_\ast(\tilde{F_\bullet})$ is an exact complex of flat sheaves on $X$ such that $i_\ast(\tilde{F_\bullet}) \otimes \mathscr{I}$ is exact for every injective sheaf $\mathscr{I} \in \Qcoh(X)$. Hence, $i_\ast(\tilde{A}) = Z_0(i_\ast(\tilde{F_\bullet}))$ is a Gorenstein flat sheaf on $X$.  

\item $i_\ast(\tilde{B})$ has finite flat-cotorsion dimension: Since $B \in (\mathcal{F}(\mathcal{O}_X(U)) \cap \mathcal{C}(\mathcal{O}_X(U)))^\wedge$, we have that there is a flat-cotorsion resolution of $B$,
\[
0 \to F_m \to F_{m-1} \to \cdots \to F_1 \to F_0 \to B \to 0,
\]
in $\Mod(\mathcal{O}_X(U))$, which induces an exact sequence
\begin{align}\label{eqn:flat-cot-res}
0 \to i_\ast(\tilde{F_m}) \to i_\ast(\tilde{F_{m-1}}) \to \cdots \to i_\ast(\tilde{F_1}) \to i_\ast(\tilde{F_0}) \to i_\ast(\tilde{B}) \to 0
\end{align}
where each $i_\ast(\tilde{F}_k)$ is a flat sheaf on $X$ by our previous comments in the proof of Proposition \ref{prop:property2}. Moreover, since $U$ is affine, for every flat sheaf $\mathscr{F}$ on $X$ and each $0 \leq k \leq m$ we have that
\[
\Ext^i_X(\mathscr{F},i_\ast(\tilde{F_k})) \cong \Ext^i_U(i^\ast(\mathscr{F}),\tilde{F_k}) \cong \Ext^1_{\mathcal{O}_X(U)}(\mathscr{F}(U),F_k) = 0.
\] 
Hence, \eqref{eqn:flat-cot-res} is a flat-cotorsion resolution of $i_\ast(\tilde{B})$.
\end{itemize}
\end{proof}


\subsection*{Cut cotorsion pairs and the Finitistic Dimension Conjecture}

Among the homological conjectures studied nowadays, the \textit{Finitistic Dimension Conjecture} has a remarkable importance in representation theory of algebras, as it implies the validity of other well-known conjectures, such as the \textit{Nunke Condition} and the (\textit{generalized}) \textit{Nakayama Conjecture}. The Finitistic Dimension Conjecture was stated by H. Bass in 1960 \cite{Bass}, and it says that the \textit{small finitistic dimension} of an Artin algebra is always finite. This problem still remains open, but has been proved in several cases (see for instance \cite{GKK,GZH,IT}).

In the next lines, we give some examples of cut cotorsion pairs and complete cut cotorsion pairs that arise when studying the finiteness of the big and small finitistic dimensions of a ring. Moreover, in the last part of this section, we shall provide a characterization of the Finitistic Dimension Conjecture in terms of complete cut cotorsion pairs.

In what follows, we let $\mcC$ be an abelian category with enough projective and injective objects. The \emph{finitistic dimension of $\mcC$} is defined as
\[
\text{Findim}(\mcC) := \pd(\mcP^\wedge).
\]
Note that $\mcP^\wedge$ is a resolving class, and since $\mcC$ has enough projectives, one has that $(\mcP^\wedge)^{\perp_1} = (\mcP^\wedge)^{\perp}$. Hence, by setting the classes 
\[
\mcG := (\mcP^\wedge)^{\perp_1} \text{ \ and \ } \mcF := {}^{\perp_1}\mcG = {}^{\perp_1}((\mcP^\wedge)^{\perp_1})
\] 
one forms a hereditary cotorsion pair $(\mcF,\mcG)$, which turns out to be useful for computing the finitistic dimension of $\mcC$, as we show in the following result.

\begin{proposition}\label{prop:findim}
${\rm Findim}(\mcC) = \coresdim_{\mcG}(\mcC) = \pd(\mcF)$.
\end{proposition}

\begin{proof}
For any $M \in \mcC$ we have that $\coresdim_{\mcG}(M) = \coresdim_{(\mcP^\wedge)^\perp}(M)$. Now by the dual of \cite[Lemma 2.11]{BMS}, we get that $\coresdim_{(\mcP^\wedge)^\perp}(M) = \id_{\mcP^\wedge}(M)$, which yields
\[
\coresdim_{(\mcP^\wedge)^\perp}(\mathcal{C}) = \id_{\mcP^\wedge}(\mathcal{C}) = \pd(\mcP^\wedge) =: \text{Findim}(\mcC).
\]
On the other hand, using again \cite[Lemma 2.11]{BMS}, we get that
\[
\pd(\mcF) = \pd({}^\perp\mcG) = \id_{{}^\perp\mcG}(\mcC) = \coresdim_{({}^\perp\mcG)^\perp}(\mcC) = \coresdim_{\mcG}(\mcC).
\]
Hence, the result follows. 
\end{proof}

In the next result we aim to characterize the finiteness of $\text{Findim}(\mcC)$ by means of the existence of a certain cut cotorsion pair.

\begin{theorem}\label{theo:findim}
The following conditions are equivalent for any $n \geq 0$ and $\mcG := (\mcP^\wedge)^{\perp_1}$:
\begin{enumerate}
\item[(a)] ${\rm Findim}(\mcC) \leq n$.

\item[(b)] $(\mcP^\wedge_n,\mcG)$ is a left cotorsion pair cut along $\mcP^\wedge_n \cup {}^{\perp_1}\mcG$.

\item[(c)] $\mcP^\wedge_n = {}^{\perp_1}\mcG$. 
\end{enumerate}
Moreover, ${\rm Findim}(\mcC) = \coresdim_{\mcG}(\mcC) = \pd({}^{\perp_1}\mcG)$.
\end{theorem}

\begin{proof} 
The implication (c) $\Rightarrow$ (b) is immediate, while for (a) $\Rightarrow$ (c) we have by Proposition \ref{prop:findim} that $\pd({}^{\perp_1}\mcG) = \text{Findim}(\mcC) \leq n$, and so ${}^{\perp_1}\mcG \subseteq \mcP^\wedge_n \subseteq \mcP^\wedge \subseteq {}^{\perp_1}\mcG$. Finally, for (b) $\Rightarrow$ (a), we can note that $\mcP^\wedge_n = {}^{\perp_1}\mcG$. We then have by Proposition \ref{prop:findim} that $\text{Findim}(\mcC) = \pd({}^{\perp_1}\mcG) = \pd(\mcP^\wedge_n) \leq n$.
\end{proof}

In the particular case where $\mcC$ is the category $\Mod(R)$ of modules over a ring $R$, let $\text{Findim}(R)$ denote the finitistic dimension of $\Mod(R)$. Using \cite[Theorems 2.2 and 3.2]{Estrada}, we can add to the equivalence in Theorem \ref{theo:findim} an additional condition, and also improve condition (b).  

Note that in Theorem \ref{theo:findim} we only need conditions \lccpone \ and \lccptwo \ to characterize the finiteness of ${\rm Findim}(\mcC)$, that is, in some cases left completeness is not required for objects along the cut. This is the case shown in the following example.

\begin{example}\label{ex:FPn}
In \cite{BGP}, the authors introduced the notion of objects of finite type in Grothen-dieck categories, as a generalization for finitely $n$-presented modules in the sense of \cite[Section 1]{BP}. An object $F$ in a Grothendieck category $\mcC$ is said to be \textbf{of type $\bm{\text{FP}_n}$} if the functor $\Ext^k_{\mcC}(F,-)$ preserves direct limits for every $0 \leq k \leq n-1$ (see \cite[Definition 2.1]{BGP}). Recall also that $\mcC$ is \textbf{locally finitely presented} if it has a generating family of finitely presented objects. 

Let $\mathcal{FP}_n$ denote the class of objects of type $\text{FP}_n$ in a locally finitely presented Grothen-dieck category $\mcC$, with $n \geq 2$. Consider also the class $\mathcal{FP}_n^{\perp_1}$ of \textbf{$\bm{\text{FP}_n}$-injective objects}. By \cite[Part 4 of Proposition 2.8]{BGP}, we know that $\mathcal{FP}_n$ is closed under direct summands. Moreover, by \cite[Proposition 3.8]{BGP} the equality $\mathcal{FP}_{n} \cap \mathcal{FP}_{n-1} = {}^{\perp_1}(\mathcal{FP}_n^{\perp_1}) \cap \mathcal{FP}_{n-1}$ holds true. Hence, $(\mathcal{FP}_n, \mathcal{FP}_n^{\perp_1})$ and $\mathcal{FP}_{n-1}$ satisfy \lccpone \ and \lccptwo. Also, it is clear that \rccpone \ and \rccptwo hold for $(\mathcal{FP}_n, \mathcal{FP}_n^{\perp_1})$ and $\mathcal{FP}_{n-1}$. 

In the particular case where $\mcC$ coincides with $\Mod(R)$, $\Ch(R)$ or $\Qcoh(X)$ (with $X$ a semi-separated scheme), then $(\mathcal{FP}_1,\mathcal{FP}_1^{\perp_1})$ is a cotorsion pair cut along the class of finitely generated objects. See \cite[Remark 3.9 and Proposition B.2]{BGP} for details. 
\end{example}

\begin{proposition}\label{prop:findimR}
Let $R$ be an arbitrary ring. The following are equivalent for the class $\mcG := (\mcP(R)^\wedge)^{\perp_1}$ and any integer $n \geq 0$:
\begin{enumerate}
\item[(a)] ${\rm Findim}(R) \leq n$.

\item[(b)] $\mcP(R)^\wedge_n \cup {}^{\perp_1}\mcG \in {\rm lCuts}(\mcP(R)^\wedge_n,\mcG)$.

\item[(c)] $\mcP(R)^\wedge_n = {}^{\perp_1}\mcG$.

\item[(d)] There exists $\mcS \in {\rm rCuts}(\mcP(R),\mcG^\vee_n)$ such that $R^{(R)} \in \mcS$. 
\end{enumerate}
Moreover, ${\rm Findim}(R) = \coresdim_{\mcG}(R^{(R)})$.
\end{proposition}

\begin{proof}
We already have from Theorem \ref{theo:findim} the implications (a) $\Leftrightarrow$ (c) and (b) $\Rightarrow$ (a) and (c). 
\begin{itemize}
\item (c) $\Rightarrow$ (b):  From (c) we have that $(\mcP(R)_n^{\wedge},\mcG) = ({}^{\perp_1}\mcG, \mcG)$ is a hereditary cotorsion pair in $\Mod(R)$, which is also complete by \cite[Theorem 7.4.6]{EJ1}. In particular, $(\mcP(R)_n^{\wedge},\mcG)$ is a complete cotorsion pair cut along any class of objects in $\Mod(R)$.

\item (c) $\Rightarrow$ (a): It follows from \cite[Theorem 3.2]{Estrada}.

\item (a) $\Rightarrow$ (d): From \cite[Theorem 3.2]{Estrada} and Proposition \ref{prop:findim}, we get that $\mcG_{n}^{\vee} = \Mod(R)$. Then, $(\mcP(R), \mcG_{n}^{\vee}) = (\mcP(R),\Mod(R))$  is clearly a complete right cotorsion pair cut along $\mcS := \Mod(R)$.

\item (d) $\Rightarrow$ (a): Condition (d) yields $\mcG_{n}^{\vee} \cap \mcS = \mcP(R)^{\perp_1} \cap \mcS = \mcS$. Thus, $R^{(R)} \in \mcS \subseteq \mcG_n^{\vee}$. By \cite[Dual of Proposition 1.11 and Theorem 3.2]{Estrada}, the latter is equivalent to saying that $\text{Findim}(R) \leq n$.
\end{itemize}
\end{proof}

Condition (d) in the previous theorem can be simplified for certain rings. Specifically, using the proof of (a) $\Leftrightarrow$ (d), along with \cite[Corollary 3.3]{Estrada}, we have the following result.

\begin{corollary}
Let $R$ be a left perfect and right coherent ring and $\mcG:=(\mathcal{P}(R)^{\wedge})^{\perp_{1}}$. Then,
${\rm Findim}(R) \leq n$ if, and only if, there exists a class $\mcS \subseteq \Mod(R)$ such that $\mcS \in {\rm rCuts}(\mcP(R), \mcG_n^{\vee})$ with $R \in \mcS$. 
\end{corollary}

In the rest of this section we apply Theorem \ref{theo:findim} to establish a relation between cotorsion cuts and the small finitistic dimension of a ring. We shall work with a slight generalization of this dimension. Let $\mathcal{FP}_{\infty}(R) = \bigcup_{n \geq 0} \mathcal{FP}_n(R)$, where $\mathcal{FP}_n(R)$ denotes the class of $R$-modules of type $\text{FP}_n$ (See Example \ref{ex:FPn} above). The $R$-modules in $\mathcal{FP}_{\infty}(R)$ are known as \emph{modules of type $\text{FP}_\infty$}.

\begin{definition}
Let $R$ be a ring. The \textbf{FP-finitistic dimension} of $R$ is
\[
{\rm FP\mbox{-}findim}(R) := \sup\{ \pd(M) \mbox{ : } M \in \mcP(R)^{\wedge} \cap \mathcal{FP}_{\infty}(R)\}.
\]
Let $\fgMod(R)$ be the class of finitely generated $R$ modules\footnote{The reader should be warned that in \cite{Estrada} the notation $\fgMod(R)$ is used for the subcategory of all $R$-modules admitting a projective resolution consisting of finitely generated modules.} . The \textbf{small finitistic dimension} of $R$ is
\[
{\rm findim}(R) := \sup\{ \pd(M) \mbox{ : } \mcP(R)^{\wedge} \cap \fgMod(R) \}.
\]
\end{definition}

\begin{remark}\label{rem:findim-FPfindim} 
It is known that in the case where $R$ is a left noetherian ring, $\mathcal{FP}_\infty(R) = \fgMod(R)$, and so ${\rm FP\mbox{-}findim}(R) = {\rm findim}(R)$.
\end{remark}

In what follows, let us consider the class
\[
\mcG^{<\infty} := (\mcP(R)^{\wedge} \cap \mathcal{FP}_{\infty}(R))^{\perp_1}.
\]
The class $\mathcal{FP}_{\infty}(R)$ is thick by \cite[Proposition 2.3]{BGH}, and so we can note that $\mcG^{<\infty} = (\mcP(R)^{\wedge} \cap \mathcal{FP}_{\infty}(R))^\perp$. From this equality, and following the arguments in Proposition \ref{prop:findim} and Theorem \ref{theo:findim}, one can show that
\begin{align}
{\rm FP\mbox{-}findim}(R) & = \coresdim_{\mathcal{G}^{<\infty}}(\Mod(R)) = \pd({}^{\perp_1}(\mathcal{G}^{<\infty})). \label{eqn:FP-findim}
\end{align}

The following result is an extension of \cite[Theorem 3.4]{Estrada} in the setting of cotorsion cuts.

\begin{proposition}\label{findimR-mod}
The following assertions are equivalent for any positive integer $n > 0$:
\begin{enumerate}
\item[(a)] ${\rm FP\mbox{-}findim}(R) \leq n$.

\item[(b)] $(\mcP(R)^{\wedge}_n \cup {}^{\perp_1}(\mcG^{<\infty})) \cap \mathcal{FP}_{\infty}(R) \in {\rm lCuts}(\mcP(R)^{\wedge}_n, \mcG^{<\infty})$.

\item[(c)] $\mcP(R)_n^{\wedge} \cap \mathcal{FP}_{\infty}(R) = {}^{\perp_1}(\mcG^{<\infty}) \cap \mathcal{FP}_{\infty}(R)$.

\item[(d)] There exists $\mcS \subseteq \Mod(R)$ such that $\mcS \in {\rm rCuts}(\mcP(R), (\mcG^{<\infty})_n^{\vee})$ with $R^{(R)} \in \mcS$.
\end{enumerate}
\end{proposition}

\begin{proof} \
\begin{itemize}
\item (a) $\Rightarrow$ (c): The equality \eqref{eqn:FP-findim} implies that ${}^{\perp_1}(\mathcal{G}^{< \infty}) \subseteq \mathcal{P}(R)^\wedge_n$, and so 
\[
{}^{\perp_1}(\mathcal{G}^{< \infty}) \cap \mathcal{FP}_{\infty}(R) \subseteq \mathcal{P}(R)^\wedge_n \cap \mathcal{FP}_{\infty}(R).
\] 
The other containment is clear. 

\item (c) $\Rightarrow$ (b): It is easy to verify from $\mcP(R)_n^{\wedge} \cap \mathcal{FP}_{\infty}(R) = {}^{\perp_1}(\mcG^{<\infty}) \cap \mathcal{FP}_{\infty}(R)$ that $(\mcP(R)^{\wedge}_n, \mcG^{<\infty})$ is a left cotorsion pair cut along $(\mcP(R)^{\wedge}_n \cup {}^{\perp_1}(\mcG^{<\infty})) \cap \mathcal{FP}_{\infty}(R)$. Condition \lccpthree \ is also clear from the assumption. 

\item (b) $\Rightarrow$ (a): From $(\mcP(R)^{\wedge}_n \cup {}^{\perp_1}(\mcG^{<\infty})) \cap \mathcal{FP}_{\infty}(R) \in \text{lCuts}(\mcP(R)^{\wedge}_n, \mcG^{<\infty})$ we can easily note that $\mcP(R)_n^{\wedge} \cap \mathcal{FP}_{\infty}(R) = {}^{\perp_1}(\mcG^{<\infty}) \cap \mathcal{FP}_{\infty}(R)$. This implies that $\mcP(R)^\wedge \cap \mathcal{FP}_{\infty}(R) = \mcP(R)^\wedge_n \cap \mathcal{FP}_{\infty}(R)$, and so from \eqref{eqn:FP-findim} we have that
\[
\hspace{1cm} \text{FP-findim}(R) = \pd(\mcP(R)^\wedge \cap \mathcal{FP}_{\infty}(R)) = \pd(\mcP(R)^\wedge_n \cap \mathcal{FP}_{\infty}(R)) \leq n.
\]

\item (a) $\Rightarrow$ (d): Similar to the corresponding implication in Proposition \ref{prop:findimR} and follows by using \cite[Theorem 3.4]{Estrada}.

\item (d) $\Rightarrow$ (a): Suppose there exists $\mcS \in \text{rCuts}(\mcP(R),(\mcG^{<\infty})^\vee_n)$ with $R^{(R)} \in \mcS$. Then, it follows that $\mcS = (\mcG^{<\infty})^\vee_n \cap \mcS$, and so $R^{(R)} \in (\mcG^{<\infty})^\vee_n$. The latter along with \cite[Theorem 3.4 and dual of Proposition 1.11]{Estrada} implies that $\text{FP-findim}(R) \leq n$. 
\end{itemize}
\end{proof}

From the previous result, we can obtain the following characterization for the finiteness of $\text{FP-findim}(R)$, provided that $R$ is coherent, in terms of right cotorsion cuts. This way we extend \cite[Corollary 3.5]{Estrada}.

\begin{corollary}
For any left coherent ring $R$, the following assertions are equivalent:
\begin{enumerate}
\item[(a)] ${\rm FP\mbox{-}findim}(R) \leq n$.

\item[(b)] There exists $\mcS \subseteq \mcC$ such that $\mcS \in {\rm rCuts}(\mcP(R), (\mcG^{<\infty})_n^{\vee})$ with $R \in \mcS$.
\end{enumerate}
\end{corollary}

Recall from Remark \ref{rem:findim-FPfindim} that $\mathcal{FP}_{\infty}(R) = \fgMod(R)$ and $\text{FP-findim}(R) = \text{findim}(R)$ provided that $R$ is a left noetherian ring. Since any left artinian ring is left noetherian, one can deduce the following extension of \cite[Corollary 3.6]{Estrada} by using Propositions \ref{prop:findimR} and \ref{findimR-mod}.

\begin{corollary}
The following statements hold true for any two-sided artinian ring $R$:
\begin{enumerate}
\item[(1)] ${\rm Findim}(R) = n$ if, and only if, $n$ is the smallest positive integer such that there exists $\mcS \subseteq \Mod(R)$, with $R \in \mcS$ and such that $(\mcP(R), \mcG_n^\vee)$ is a complete right cotorsion pair cut along  $\mcS$.

\item[(2)] ${\rm findim}(R) = n$ if, and only if, $n$ is the smallest positive integer such that there exists $\mcS \subseteq \Mod(R)$, with $R \in \mcS$ and such that $(\mcP(R), (\mcG^{<\infty})_n^\vee)$ is a complete right cotorsion pair cut along $\mcS$.
\end{enumerate}
\end{corollary}


\subsection*{Relations with Serre subcategories}

Let $\mcC$ be a locally small abelian category. Recall that a subcategory $\mcS \subseteq \mcC$ is a \emph{Serre subcategory} if for every short exact sequence $0 \to X \to Y \to Z \to 0$ in $\mcC$, one has that $Y \in \mcS$ if, and only if, $X, Z \in \mcS$. In particular, Serre subcategories are clearly thick, and closed under subobjects and quotients. 

If $\mcS \subseteq \mcC$ is a Serre subcategory, we can consider the \emph{Serre quotient} $\mcC / \mcS$, which is an abelian category whose objects are the same objects in $\mcC$, and whose morphisms $X \to Y$ are defined as the direct limit of abelian groups $\varinjlim \Hom_{\mcC}(X', Y / Y')$ with $X'$ and $Y'$ running over the subobjects of $X$ and $Y$, respectively, and such that $X / X' \in \mcS$ and $Y' \in \mcS$. For the Serre quotient $\mcC / \mcS$, there is an associated quotient functor $Q \colon \mcC \longrightarrow \mcC / \mcS$ which sends any object $C \in \mcC$ to itself, and any morphism $f \colon X \to Y$ to its corresponding element in the direct limit $(f_{X',Y'}) \in \varinjlim \Hom_{\mcC}(X', Y / Y')$ for $X' = X$ and $Y' = 0$, that is, $f \mapsto f_{X,0}$. 

In Ogawa's \cite{Ogawa}, the author gives several outcomes from the existence of a right adjoint for $Q$. The purpose of this section is to characterize the latter via complete right cut cotorsion pairs. Let us begin proving the following consequence of having a right adjoint for $Q$.

\begin{proposition}\label{prop:Serre}
Let $\mcS$ be a Serre subcategory of $\mcC$. If the Serre quotient functor $Q \colon \mcC \longrightarrow \mcC / \mcS$ admits a right adjoint, then $(\mcS,\mcS^{\perp_0} \cap \mcS^{\perp_1})$ is a complete right cotorsion pair cut along $\mcS^{\perp_0}$.
\end{proposition}

\begin{proof}
It is clear that the dual of conditions (1) and (2) in Proposition \ref{prop:characterization_ccp} are satisfied. It is only left to show that for every object $M \in \mcS^{\perp_0}$ there exists a short exact sequence $0 \to M \to F \to K \to 0$ with $F \in \mcS^{\perp_0} \cap \mcS^{\perp_1}$ and $K \in \mcS$. So let us take $M \in \mcS^{\perp_0}$. By \cite[Propositions 1.1 and 1.3]{Ogawa}, there exists an exact sequence $S \xrightarrow{f} M \xrightarrow{g} Y$ with $S \in \mcS$ and $Y \in \mcS^{\perp_0} \cap \mcS^{\perp_1}$. Since $M \in \mcS^{\perp_0}$, we have that $f = 0$, and so $g$ is a monomorphism. We can thus consider the short exact sequence $0 \to M \xrightarrow{g} Y \to \Coker(g) \to 0$. Let us now apply again \cite[Propositions 1.1 and 1.3]{Ogawa} to the object $\Coker(g)$. We get an exact sequence $D \xrightarrow{h} \Coker(g) \xrightarrow{i} E$ with $D \in \mcS$ and $E \in \mcS^{\perp_0} \cap \mcS^{\perp_1}$. Let us factor $h$ and $i$ through their images, so that we get the following commutative diagram
\begin{equation}\label{Serre1}
\parbox{4in}{
\begin{tikzpicture}[description/.style={fill=white,inner sep=2pt}] 
\matrix (m) [ampersand replacement=\&, matrix of math nodes, row sep=2.5em, column sep=2.5em, text height=1.25ex, text depth=0.25ex] 
{ 
D \& {} \& \Coker(g) \& {} \& E \& \Coker(i) \\
{} \& K \& {} \& C \& {} \& {} \\
}; 
\path[->] 
(m-1-1) edge (m-1-3) (m-1-3) edge (m-1-5) (m-1-5) edge (m-1-6)
; 
\path[>->]
(m-2-2) edge (m-1-3) (m-2-4) edge (m-1-5)
;
\path[->>]
(m-1-1) edge (m-2-2) (m-1-3) edge (m-2-4)
;
\end{tikzpicture}
}
\end{equation}
where $K := {\rm Im}(h) = \Ker(i)$ and $C := \Coker(\Ker(i)) \simeq \Ker(\Coker(i))$. Notice that $K \in \mcS$ since $D \in \mcS$ and $\mcS$ is closed under quotients. Taking the pullback of $K \to \Coker(g) \leftarrow Y$ yields the following solid diagram:
\begin{equation}\label{Serre}
\parbox{1.5in}{
\begin{tikzpicture}[description/.style={fill=white,inner sep=2pt}] 
\matrix (m) [ampersand replacement=\&, matrix of math nodes, row sep=2.5em, column sep=2.5em, text height=1.25ex, text depth=0.25ex] 
{ 
M \& M \& {} \\
F \& Y \& C \\
K \& \Coker(g) \& C \\
}; 
\path[->] 
(m-2-1)-- node[pos=0.5] {\footnotesize$\mbox{\bf pb}$} (m-3-2)
; 
\path[-,font=\scriptsize]
(m-1-1) edge [double, thick, double distance=2pt] (m-1-2)
(m-2-3) edge [double, thick, double distance=2pt] (m-3-3)
;
\path[>->]
(m-1-1) edge (m-2-1)
(m-1-2) edge (m-2-2)
(m-2-1) edge (m-2-2)
(m-3-1) edge (m-3-2)
;
\path[->>]
(m-2-1) edge (m-3-1)
(m-2-2) edge (m-3-2)
(m-2-2) edge (m-2-3)
(m-3-2) edge (m-3-3)
;
\end{tikzpicture}
}
\end{equation}
We show that $F \in \mcS^{\perp_0} \cap \mcS^{\perp_1}$. Let $S' \in \mcS$ and apply the functor $\Hom_{\mcC}(S',-)$ to the central row in \eqref{Serre}. We get the following exact sequence:
\[
\Hom_{\mcC}(S',F) \rightarrowtail \Hom_{\mcC}(S',Y) \to \Hom_{\mcC}(S',C) \to \Ext^1_{\mcC}(S',F) \to \Ext^1_{\mcC}(S',Y),
\]
where $\Hom_{\mcC}(S',Y) = 0 = \Ext^1_{\mcC}(S',Y)$ since $Y \in \mcS^{\perp_0} \cap \mcS^{\perp_1}$. It follows that $\Hom_{\mcC}(S',F) = 0$ and $\Hom_{\mcC}(S',C) \cong \Ext^1_{\mcC}(S',F)$. Now consider the short exact sequence $0 \to C \to E \to \Coker(i) \to 0$ in \eqref{Serre1}. By applying the functor $\Hom_{\mcC}(S',-)$ to this sequence we obtain the monomorphism 
\[
0 \to \Hom_{\mcC}(S',C) \to \Hom_{\mcC}(S',E),
\] 
where $\Hom_{\mcC}(S',E) = 0$ since $E \in \mcS^{\perp_0}$. Then, $\Ext^1_{\mcC}(S',F) \cong \Hom_{\mcC}(S',C) = 0$ for every $S' \in \mcS$. Therefore, $F \in \mcS^{\perp_0} \cap \mcS^{\perp_1}$, and thus the left-hand column in \eqref{Serre} is the desired exact sequence. 
\end{proof}

In the next result we prove the converse of the previous proposition with an additional condition: we shall need $\mcC$ to be cocomplete. We show that having a right adjoint for the quotient functor $Q \colon \mcC \to \mcC / \mcS$ is equivalent to the existence of $(\mcS,\mcS^{\perp_0} \cap \mcS^{\perp_1})$ as a complete right cotorsion pair cut along $\mcS^{\perp_0}$. 

We need to recall from \cite[Definition 2.1]{KSZ} that two classes of objects $(\mathcal{T,F})$ form a \emph{torsion pair} in $\mcC$ if $\Hom_{\mcC}(\mathcal{T,F}) = 0$ and if for every $C \in \mcC$ there exists a short exact sequence $0 \to T_M \to M \to F_M \to 0$ with $T_M \in \mathcal{T}$ and $F_M \in \mathcal{F}$. The classes $\mathcal{T}$ and $\mathcal{F}$ are called the \emph{torsion class} and the \emph{torsion-free class}, respectively. If $\mcC$ is cocomplete, it is known that $(\mathcal{T,F})$ is a torsion pair if, and only if, $\mathcal{T}$ is closed under extensions, quotients and coproducts (See for instance \cite[Proposition VI. 2.1]{Stenstrom}). In particular, every Serre subcategory $\mcS$ of a cocomplete locally small abelian category $\mcC$, which is closed under coproducts, is a torsion class.

\begin{theorem}\label{theo:Serre}
Let $\mcS$ be a Serre subcategory of a cocomplete abelian category $\mcC$. If $\mcS$ is closed under coproducts, then the following conditions are equivalent:
\begin{enumerate}
\item[(a)] $Q \colon \mcC \to \mcC / \mcS$ admits a right adjoint. 

\item[(b)] $(\mcS,\mcS^{\perp_0} \cap \mcS^{\perp_1})$ is a complete right cotorsion pair cut along $\mcS^{\perp_0}$.   
\end{enumerate}
\end{theorem}

\begin{proof}
The implication (a) $\Rightarrow$ (b) is Proposition \ref{prop:Serre}. For the implication (b) $\Rightarrow$ (a), by \cite[Propositions 1.3]{Ogawa} it suffices to show that for every $M \in \mcC$ there exists an exact sequence $S \xrightarrow{f} M \xrightarrow{g} Y$ where $S \in \mcS$ and $Y \in \mcS^{\perp_0} \cap \mcS^{\perp_1}$. Indeed, for $M \in \mcC$, we have an exact sequence $0 \to S \xrightarrow{f} M \xrightarrow{g} S_0 \to 0$ with $S \in \mcS$ and $S_0 \in \mcS^{\perp_0}$, since $\mcS$ is a torsion class in $\mcC$. On the other hand, since $(\mcS,\mcS^{\perp_0} \cap \mcS^{\perp_1})$ is a complete right cotorsion pair cut along $\mcS^{\perp_0}$, there exists a monomorphism $S_0 \xrightarrow{h} Y$ with $Y \in \mcS^{\perp_0} \cap \mcS^{\perp_1}$. Hence, we can form the exact sequence $S \xrightarrow{f} M \xrightarrow{h \circ g} Y$. 
\end{proof}


\subsection*{Cuts from extriangulated categories}

We conclude this article with a final application of complete cut cotorsion pairs in the context of extriangulated categories. Such categories where introduced by H. Nakaoka and Y. Palu in \cite{NP} as a simultaneous generalization of triangulated categories and exact categories. \\

In what follows, we let $(\mathfrak{A},\mathbb{E},\mathfrak{s})$ denote an extriangulated category. Here, $\mathfrak{A}$ is a skeletally small additive category, $\mathbb{E} \colon \mathfrak{A}^{\rm op} \times \mathfrak{A} \longrightarrow \mathsf{Ab}$ is a biadditive functor with an additive realization $\mathfrak{s}$ satisfying a series of axioms (see \cite[Definition 2.12]{NP} for details). We shall also consider the following categories constructed from $\mathfrak{A}$: 
\begin{itemize}
\item $\mathsf{mod}(\mathfrak{A}^{\rm op})$ denotes the subcategory of the (Grothendieck) category $\Mod(\mathfrak{A}^{\rm op})$ formed by the right $\mathfrak{A}$-modules (or $\mathfrak{A}^{\rm op}$-modules) $F \colon \mathfrak{A}^{\rm op} \longrightarrow \mathsf{Ab}$ which are \emph{finitely presented}, that is, for which  there exists an exact sequence
\[
\Hom_{\mathfrak{A}}(-,A_1) \to \Hom_{\mathfrak{A}}(-,A_0) \to F \to 0
\]
in $\Mod(\mathfrak{A}^{\rm op})$ with $A_0, A_1 \in \mathfrak{A}$.

\item For any subcategory $\mcX \subseteq \Mod(\mathfrak{A}^{\rm op})$ of $\mathfrak{A}^{\rm op}$-modules, $\overrightarrow{\mcX}$ is the subcategory of $\Mod(\mathfrak{A}^{\rm op})$ of direct limits of objects in $\mcX$. 

\item $\text{Lex}(\mathfrak{A}^{\rm op})$ and $\text{lex}(\mathfrak{A}^{\rm op})$ denote the subcategories of $\Mod(\mathfrak{A}^{\rm op})$ and $\mathsf{mod}(\mathfrak{A}^{\rm op})$, respectively, of all left exact $\mathfrak{A}^{\rm op}$-modules. Recall that an $\mathfrak{A}^{\rm op}$-module is \emph{left exact} if it maps kernels in $\mathfrak{A}$ into cokernels in $\mathsf{Ab}$. 

\item $\text{def}(\mathfrak{A}^{\rm op})$ denotes the subcategory of $\mathsf{mod}(\mathfrak{A}^{\rm op})$ consisting of all finitely presented $\mathfrak{A}^{\rm op}$-modules isomorphic to defects (see \cite[Definition 2.4]{Ogawa}).
\end{itemize}

For skeletally small extriangulated categories, we can obtain from the subcategories $\overrightarrow{\text{def}(\mathfrak{A}^{\rm op})}$ and $\text{Lex}(\mathfrak{A}^{\rm op})$ the following example of a complete cut cotorsion pair.

\begin{proposition}\label{prop:defect}
Let $(\mathfrak{A},\mathbb{E},\mathfrak{s})$ be a skeletally small extriangulated category with weak kernels. Then, $(\overrightarrow{{\rm def}(\mathfrak{A}^{\rm op})},{\rm Lex}(\mathfrak{A}^{\rm op}))$ is a complete right cotorsion pair cut along the class $(\overrightarrow{{\rm def}(\mathfrak{A}^{\rm op})})^{\perp_0}$. 
\end{proposition}

\begin{proof}
By \cite[Proposition 2.5]{Ogawa} we know that $\text{def}(\mathfrak{A}^{\rm op})$ is a Serre subcategory of $\mathsf{mod}(\mathfrak{A}^{\rm op})$. On the other hand, by Krause's \cite[Theorem 2.8]{Krause} we have that $\overrightarrow{\text{def}(\mathfrak{A}^{\rm op})}$ is a Serre subcategory of $\Mod(\mathfrak{A}^{\rm op})$, and by \cite[Theorem 3.1]{Ogawa} the quotient functor 
\[
Q \colon \Mod(\mathfrak{A}^{\rm op}) \longrightarrow \frac{\Mod(\mathfrak{A}^{\rm op})}{\overrightarrow{\text{def}(\mathfrak{A}^{\rm op})}}
\] 
admits a right adjoint. It then follows by Proposition \ref{prop:Serre} that 
\[
(\overrightarrow{\text{def}(\mathfrak{A}^{\rm op})},(\overrightarrow{\text{def}(\mathfrak{A}^{\rm op})})^{\perp_0} \cap (\overrightarrow{\text{def}(\mathfrak{A}^{\rm op})})^{\perp_1})
\] 
is a complete right cotorsion pair cut along $(\overrightarrow{\text{def}(\mathfrak{A}^{\rm op})})^{\perp_0}$. Finally, the equality $\text{Lex}(\mathfrak{A}^{\rm op}) = (\overrightarrow{\text{def}(\mathfrak{A}^{\rm op})})^{\perp_0} \cap (\overrightarrow{\text{def}(\mathfrak{A}^{\rm op})})^{\perp_1}$ follows by \cite[Lemma 3.3]{Ogawa}. 
\end{proof}

The following is the finitely presented version of the previous proposition.

\begin{proposition}\label{prop:fp-defect}
Let $(\mathfrak{A},\mathbb{E},\mathfrak{s})$ be a skeletally small extriangulated category with weak kernels. If the quotient functor 
\[
Q \colon \mathsf{mod}(\mathfrak{A}^{\rm op}) \longrightarrow \frac{\mathsf{mod}(\mathfrak{A}^{\rm op})}{{\rm def}(\mathfrak{A}^{\rm op})}
\] 
admits a right adjoint, then $({\rm def}(\mathfrak{A}^{\rm op}),{\rm lex}(\mathfrak{A}^{\rm op}))$ is a complete right cotorsion pair cut along $({\rm def}(\mathfrak{A}^{\rm op}))^{\perp_0}$.
\end{proposition}

\begin{proof}
Follows as Proposition \ref{prop:defect} by using \cite[Propositions 2.5, 2.8 (2), Theorem 2.9]{Ogawa} and Proposition \ref{prop:Serre}. 
\end{proof}

The existence of the previous cut cotorsion pair can be also guaranteed in the particular case where $\mathfrak{A}$ is an exact category, under some mild additional assumptions, as we specify below.

\begin{corollary}
Let $\mathfrak{A}$ be a skeletally small exact category with weak kernels and enough projectives. Then, $({\rm def}(\mathfrak{A}^{\rm op}),{\rm lex}(\mathfrak{A}^{\rm op}))$ is a complete right cotorsion pair cut along the class $({\rm def}(\mathfrak{A}^{\rm op}))^{\perp_0}$.
\end{corollary}

\begin{proof}
By \cite[Propositions 2.16 (1) and 2.17]{Ogawa} we have that the quotient functor $Q \colon \mathsf{mod}(\mathfrak{A}^{\rm op}) \longrightarrow \frac{\mathsf{mod}(\mathfrak{A}^{\rm op})}{\text{def}(\mathfrak{A}^{\rm op})}$ admits a right adjoint. Hence, the result follows from Proposition \ref{prop:fp-defect}. 
\end{proof}

We conclude this section also covering the other particular case where $\mathfrak{A}$ is a triangulated category. We show how to induce from a cotorsion pair $(\mathcal{U,V})$ in $\mathfrak{A}$, the complete right cotorsion pair $(\text{def}(\mathcal{U}^{\rm op}),\text{lex}(\mathcal{V}^{\rm op}))$ cut along $(\text{def}(\mathcal{U}^{\rm op}))^{\perp_0}$ in $\mathsf{mod}(\mathfrak{A}^{\rm op})$. Moreover, we show that $(\mathcal{U,V})$ is a co-t-structure if, and only if, $(\text{def}(\mathcal{U}^{\rm op}),\text{lex}(\mathcal{V}^{\rm op}))$ is also a complete left cotorsion pair cut along the same class $(\text{def}(\mathcal{U}^{\rm op}))^{\perp_0}$. 

In what follows, let us fix a skeletally small triangulated category $\mathfrak{A}$ with translation automorphism $[1] \colon \mathfrak{A} \longrightarrow \mathfrak{A}$. Given two (full) additive subcategories $\mathcal{U, V} \subseteq \mathfrak{A}$, recall that $(\mathcal{U,V})$ is a \emph{cotorsion pair} in $\mathfrak{A}$ if the following two conditions are satisfied: 
\begin{enumerate}
\item $\Hom_{\mathfrak{A}}(U,V') = 0$ for every $U \in \mathcal{U}$ and $V' \in \mathcal{V}[1]$. Here, $\mathcal{V}[1]$ denotes the class of objects in $\mathfrak{A}$ isomorphic to objects of the form $[1](V)$ with $V \in \mathcal{V}$.

\item $\mathfrak{A} = \mathcal{U} \ast \mathcal{V}[1]$, that is, if every object $C \in \mathfrak{A}$ admits a distinguished triangle $ U \to C \to V' \to U[1]$ where $U \in \mathcal{U}$ and $V' \in \mathcal{V}[1]$.  
\end{enumerate}
Following \cite[Section 4]{Ogawa}, if $(\mathcal{U,V})$ is a cotorsion pair in $\mathfrak{A}$, then $\mathcal{U}$ gives rise to an extriangulated category with weak kernels, translation automorphism $[1]|_{\mathcal{U}}$ and biadditive functor 
\[
\mathbb{E}(+,-) := \text{Hom}_{\mathcal{U}}(+,-[1]) \colon \mathcal{U}^{\rm op} \times \mathcal{U} \longrightarrow \mathsf{Ab}.
\] 
Here, $\mathfrak{A}^+ = \mathcal{W} \ast \mathcal{V}[1]$ and $\mathfrak{A}^- = \mathcal{U}[-1] \ast \mathcal{W}$, where $\mathcal{W} = \mathcal{U} \cap \mathcal{V}$. Moreover, by \cite[Proposition 4.2]{Ogawa} the quotient functor $Q \colon \mathsf{mod}(\mathcal{U}^{\rm op}) \longrightarrow \frac{\mathsf{mod}(\mathcal{U}^{\rm op})}{\text{def}(\mathcal{U}^{\rm op})}$ has a right adjoint, and so from Proposition \ref{prop:fp-defect} we deduce the following result.

\begin{corollary}
Let $(\mathcal{U,V})$ be a cotorsion pair in $\mathfrak{A}$. Then, $({\rm def}(\mathcal{U}^{\rm op}),{\rm lex}(\mathcal{U}^{\rm op}))$ is a complete right cotorsion pair cut along $({\rm def}(\mathcal{U}^{\rm op}))^{\perp_0}$. 
\end{corollary}

One interesting feature about the triangulated setting is the following relation between the complete right cut cotorsion pair $({\rm def}(\mathcal{U}^{\rm op}),{\rm lex}(\mathcal{U}^{\rm op}))$ and the notion of co-t-structures. Recall that a \emph{co-t-structure} is a pair $(\mathcal{X,Y})$ of subcategories of $\mathfrak{A}$ such that $(\mcX[1],\mcY)$ is a cotorsion pair in $\mathfrak{A}$ satisfying $\mcX \subseteq \mcX[1]$.

\begin{proposition}\label{prop:triangle_defect}
The following are equivalent for every cotorsion pair $(\mathcal{U,V})$ in $\mathfrak{A}$:
\begin{enumerate}
\item[(a)] $(\mathcal{U,V})$ is a co-t-structure in $\mathfrak{A}$ (that is, $\mathcal{U}[-1] \subseteq \mathcal{U}$).

\item[(b)] $({\rm def}(\mathcal{U}^{\rm op}),{\rm lex}(\mathcal{U}^{\rm op}))$ is a complete cotorsion pair cut along $({\rm def}(\mathcal{U}^{\rm op}))^{\perp_0}$.
\end{enumerate}
\end{proposition}

\begin{proof}
First, suppose condition (a) holds. By the previous corollary, we have that $({\rm def}(\mathcal{U}^{\rm op}),{\rm lex}(\mathcal{U}^{\rm op}))$ is a complete right cotorsion pair cut along $({\rm def}(\mathcal{U}^{\rm op}))^{\perp_0}$. So we focus on showing that $({\rm def}(\mathcal{U}^{\rm op}),{\rm lex}(\mathcal{U}^{\rm op}))$ is a complete left cotorsion pair cut along $({\rm def}(\mathcal{U}^{\rm op}))^{\perp_0}$. Consider the heart of the cotorsion pair $(\mathcal{U,V})$ given by $\underline{\mathcal{H}} = (\mathfrak{A}^+ \cap \mathfrak{A}^-) / \mathcal{W}$. It is known by Nakaoka's \cite[Theorem 6.4]{Nakaoka2011} that $\underline{\mathcal{H}}$ is an abelian category. Moreover, by  \cite[Theorem 4.7]{Ogawa} it is also known that $\underline{\mathcal{H}}$ and $\text{lex}(\mathcal{U}^{\rm op})$ are naturally equivalent. Using the assumption (a) that $(\mathcal{U,V})$ is a co-t-structure in $\mathfrak{A}$, we can note that $\mathfrak{A}^{+} \cap \mathfrak{A}^{-} \subseteq \mathcal{U} \cap \mathcal{V}$, and so $\underline{\mathcal{H}} = 0$, and hence $\text{lex}(\mathcal{U}^{\rm op}) = 0$. On the other hand, by \cite[Proposition 4.2]{Ogawa} and \cite[Theorem IV.4.5]{Popescu} , we can note that for every $X \in \mathsf{mod}(\mathcal{U}^{\rm op})$ there exists an epimorphism $D \twoheadrightarrow X$ with $D \in \text{def}(\mathcal{U}^{\rm op})$, and since $\text{def}(\mathcal{U}^{\rm op})$ is a Serre subcategory, the previous implies that $\mathsf{mod}(\mathcal{U}^{\rm op}) = \text{def}(\mathcal{U}^{\rm op})$. It then follows that $(\text{def}(\mathcal{U}^{\rm op}),\text{lex}(\mathcal{U}^{\rm op})) = (\mathsf{mod}(\mathcal{U}^{\rm op}),0)$, which is clearly a complete left cotorsion pair cut along $(\text{def}(\mathcal{U}^{\rm op}))^{\perp_0}$. 

Now let us assume (b). We thus have that $({\rm def}(\mathcal{U}^{\rm op}),{\rm lex}(\mathcal{U}^{\rm op}))$ is a complete cotorsion pair cut along $({\rm def}(\mathcal{U}^{\rm op}))^{\perp_0}$, and so for every $X \in ({\rm def}(\mathcal{U}^{\rm op}))^{\perp_0}$ there exists an epimorphism $D \twoheadrightarrow X$ with $D \in {\rm def}(\mathcal{U}^{\rm op})$. Again, since ${\rm def}(\mathcal{U}^{\rm op})$ is a Serre subcategory, we have that $X \in ({\rm def}(\mathcal{U}^{\rm op}))^{\perp_0} \cap {\rm def}(\mathcal{U}^{\rm op}) = 0$. If follows that $({\rm def}(\mathcal{U}^{\rm op}))^{\perp_0} = 0$, which in turn and along with \cite[Proposition 2.8 (2)]{Ogawa} implies that ${\rm lex}(\mathcal{U}^{\rm op}) = ({\rm def}(\mathcal{U}^{\rm op}))^{\perp_0} \cap ({\rm def}(\mathcal{U}^{\rm op}))^{\perp_1} = 0$. Hence, the cotorsion pair $(\mathcal{U,V})$ is a co-t-structure in $\mathfrak{A}$ by \cite[Remark 2.6]{Nakaoka2013} and \cite[Theorem 4.7]{Ogawa}. 
\end{proof}




\section*{\textbf{Funding}}

The authors thank Project PAPIIT-Universidad Nacional Aut\'onoma de M\'exico IN100520. The first author thanks Sociedad Matem\'atica Mexicana (SMM)-Fundaci\'on Sof\'ia Kovaleskaia (SK). The third author was partially supported by the following grants and institutions: postdoctoral fellowship (Comisi\'on Acad\'emica de Posgrado - Universidad de la Rep\'ublica), Fondo Vaz Ferreira \# II/FVF/2019/135 (funds are given by the Direcci\'on para el Desarrollo de la Ciencia y el Conocimiento - Ministerio de Educaci\'on y Cultura - Rep\'ublica Oriental del Uruguay, and administered through Fundaci\'on Julio Ricaldoni), Agencia Nacional de Investigaci\'on e Innovaci\'on (ANII), and Programa de Desarrollo de las Ciencias B\'asicas (PEDECIBA).


\bibliographystyle{plain}
\bibliography{bibliohmp18}

\end{document}